\documentclass[UTF-8,reqno]{amsart}
\usepackage{enumerate}
\usepackage{mhequ}
\usepackage{amssymb,url,color, booktabs,nccmath,amsmath}
\usepackage{amsthm}
\usepackage[left=3.2cm,right=3.2cm,top=4cm,bottom=4cm]{geometry}
\usepackage{mathrsfs}
\usepackage{enumitem,dsfont}
\usepackage{color}
\usepackage[colorlinks=true]{hyperref}

\hypersetup{
	linkcolor=blue,          % color of internal links
	citecolor=red,        % color of links to bibliography
	filecolor=blue,      % color of file links
	urlcolor=cyan
}

\definecolor{darkergreen}{rgb}{0.0, 0.5, 0.0}

\numberwithin{equation}{section}
\def\theequation{\arabic{section}.\arabic{equation}}
\newcommand{\be}{\begin{eqnarray}}
	\newcommand{\ee}{\end{eqnarray}}
\newcommand{\ce}{\begin{eqnarray*}}
	\newcommand{\de}{\end{eqnarray*}}

\newtheorem{theorem}{Theorem}[section]

\newtheorem{lemma}[theorem]{Lemma}
\newtheorem{proposition}[theorem]{Proposition}

\newtheorem{Examples}[theorem]{Example}
\newtheorem{corollary}[theorem]{Corollary}

\newtheorem{definition}[theorem]{Definition}
\theoremstyle{definition}
\newtheorem{remark}[theorem]{Remark}

\def\${|\!|\!|}

\DeclareMathOperator{\supp}{supp}

\def\p{\partial}

\def\<{{\langle}}
\def\>{{\rangle}}
\def\({{\Big(}}
\def\){{\Big)}}

\def\bx{{\mathbf{x}}}
\def\tr{\mathrm {tr}}
\def\W{{\mathcal W}}

\def\dif{{\mathord{{\rm d}}}}

\def\min{{\mathord{{\rm min}}}}

\def\={&\!\!=\!\!&}

\def\mN{{\mathbb N}}

\def\mP{{\mathbb P}}

\def\mR{{\mathbb R}}

\def\1{{\mathbf{1}}}

\def\E{\mathbf E}

\def\geq{\geqslant}
\def\leq{\leqslant}

\def\div{\mathord{{\rm div}}}

\def\p{\partial}

\def\<{{\langle}}
\def\>{{\rangle}}
\def\({{\Big(}}
\def\){{\Big)}}

\def\bx{{\mathbf{x}}}
\def\tr{\mathrm {Tr}}
\def\W{{\mathcal W}}

\def\dif{{\mathord{{\rm d}}}}

\def\min{{\mathord{{\rm min}}}}

\def\={&\!\!=\!\!&}
\def\bt{\begin{theorem}}
	\def\et{\end{theorem}}
\def\bl{\begin{lemma}}
	\def\el{\end{lemma}}
\def\br{\begin{remark}}
	\def\er{\end{remark}}
\def\bx{\begin{Examples}}
	\def\ex{\end{Examples}}
\def\bd{\begin{definition}}
	\def\ed{\end{definition}}
\def\bp{\begin{proposition}}
	\def\ep{\end{proposition}}
\def\bc{\begin{corollary}}
	\def\ec{\end{corollary}}

\def\geq{\geqslant}
\def\leq{\leqslant}

\def\div{\mathord{{\rm div}}}

\def\Id{\textrm{Id}}

 \def\R{\mathbb R}
 \def\R{\mathbb R}    
\def\N{\mathbb N}  
   
\def\<{\langle} \def\>{\rangle}

\makeatletter

\newcommand{\Rmnum}[1]{\expandafter\@slowromancap\romannumeral #1@}
\makeatother

\allowdisplaybreaks

\begin{document}
	\fontsize{10.0pt}{\baselineskip}\selectfont
	
	\title[H\"{o}lder continuous solutions to stochastic 3D Euler equations ]{  H\"{o}lder continuous solutions to stochastic 3D Euler equations via stochastic convex integration}
	
	\author{Lin L\"{u}}  
	\address[L. L\"{u}]{School of Mathematics and Statistics, Beijing Institute of Technology, Beijing 100081, China}
	\email{3120235976@bit.edu.cn}

	\begin{abstract}
		In this paper, we are concerned with the three-dimensional Euler equations driven by an additive stochastic forcing. First, we construct global H\"{o}lder continuous (stationary) solutions in $C(\mathbb{R};C^{\vartheta})$ space for some $\vartheta>0$  via a different method from \cite{LZ24}. Our approach is based on applying stochastic convex integration to the construction of Euler flows in \cite{DelSze13} to derive uniform moment estimates independent of time. Second, for any divergence-free H\"{o}lder continuous initial condition, we show the existence of infinitely many global-in-time probabilistically strong and analytically weak solutions in $L^p_{\rm{loc}}([0,\infty);C^{\vartheta'}) \cap C_{\rm{loc}}([0,\infty);H_{\sigma}^{-1})$ for all $p\in [1,\infty)$ and some $\vartheta'>0$.
	\end{abstract}
	
	
	\maketitle
	\tableofcontents
	\section{Introduction}
	The Euler equations serve as the classical model depicting the motion of an inviscid, incompressible, homogeneous fluid, and their ill/well-posedness is always intriguing. In the deterministic case, there is a vast amount of literature on the mathematical theory for the Euler equations; see \cite{BM02, BT07, Che98, Con07} and references therein. Notably, significant progress over the past decade has been made by employing the convex integration method to establish the non-uniqueness of weak solutions to the Euler equations. This breakthrough triggered a series of works, including \cite{Buc15, BDLIS16, DLS10, DLS12, DelSze13, DLS14, DS17}, culminating in the resolution of the flexible part of Onsager’s conjecture by Isett in \cite{Ise18} (see also \cite{BDLSV19} for the extension of the result to admissible solutions). Additionally, recent works \cite{GKN23, GR23, NV23} have resolved Onsager's conjecture for 2D and the strong Onsager theorem.

	Given the extensive application of stochastics in fluid dynamics, there arises a crucial necessity to improve the mathematical foundations of stochastic partial differential equations governing fluid flow, especially focusing on inviscid models such as stochastic Euler equations. Early works on stochastic Euler equations have been extensively treated in both two and three dimensions; see, for example, \cite{BF99, BFM16, BP01, CFH19, GV!4, Kim09, MV00}. In particular, Glatt-Holtz and Vicol \cite{GV!4} obtained local existence and uniqueness of probabilistically strong solutions to stochastic Euler equations in three dimensions for nonlinear multiplicative noise, and global well-posedness in two dimensions for additive and linear multiplicative noise. Most recently, the convex integration has already been applied in stochastic cases; for example, \cite{BFH20, CFF19, HZZ22a, HZZ22b, HLP22, LZ24} concerning the Euler system, and \cite{CDZ22, HZZ22b, HZZ21markov, HZZ23b, HZZ19, Umb23} concerning the Navier--Stokes system. In addition, the Cauchy problem associated with stochastic systems has been studied in \cite{HZZ21markov, HZZ23b, Umb23}, establishing the existence of infinitely many global-in-time probabilistically strong and analytically weak solutions to 3D Navier--Stokes equations driven by additive noise of trace class, space-time white noise, and Stratonovich transport noise. One should also mention that in \cite{CDZ22, HZZ22b}, the authors developed a new stochastic convex integration method to derive global solutions without using stopping times.
	
	In the present work, we consider the incompressible Euler equations with additive noise, on the three-dimensional torus $\mathbb{T}^3=\mathbb{R}^3/(2\pi\mathbb{Z})^3$ 
	\begin{equation}\label{eul1}
		\begin{aligned}
			\dif u+\div(u\otimes u)\,\dif t+\nabla P\,\dif t&=\dif B,
			\\
			\div \, u&=0,
		\end{aligned}
	\end{equation}
	where $u$ denotes the velocity vector field, and $P$ is the pressure scalar
	field.  Here $B$ is a $GG^*$-Wiener process on a given filtered probability space $(\Omega,\mathcal{F},(\mathcal{F}_{t})_{t\in\mR},\mathbf{P})$ and $G$ is a Hilbert-Schmidt operator from $U$ to $L_{\sigma}^2$ for some Hilbert space $U$.
	
	In the recent works \cite{HZZ22b} and \cite{LZ24}, global and stationary solutions to equation \eqref{eul1} have been constructed, respectively, with Sobolev regularity and Hölder continuity in space. Our approach in \cite{LZ24} involved combining stochastic convex integration with pathwise estimates during the iteration. This was  prompted by the necessity of pathwise lower bounds for the gradient of the solutions to certain transport equations, which are used to cancel the transport error during the iteration and improve the regularity of the solutions to the Euler system. Consequently, we incorporated a cut-off technique throughout the iteration. It is natural to ask whether we can directly construct Hölder continuous solutions via stochastic convex integration, and the first aim of this paper is to address this issue. Our first main result echoes that of \cite[Theorem 1.2]{LZ24} and is stated as follows, with its proof detailed in Theorem~\ref{v}.
	\begin{theorem}
		\label{Theorem 1.1}
		Suppose that $\tr((-\Delta)^{3/2+\varkappa}GG^*)<\infty$ for some $\varkappa>0$, then
	\begin{enumerate}[label=(\alph*)]
				\item  there exists an $(\mathcal{F}_{t})_{t\in\mR}$-adapted process $u$ which belongs to $C(\R;C^\vartheta)$ for some $\vartheta >0$ and is an analytically weak solution to \eqref{eul1} with $(\mathcal{F}_{t})_{t\in\mR}$ being the normal filtration generated by the Wiener process  $B$;
				\item  there exists a stationary solution $\bar{u} \in C(\mR;C^{\frac\vartheta2})$ to the stochastic Euler equations \eqref{eul1}.
			\end{enumerate}
			Moreover, the solutions are non-unique, as their energy can be determined by various energy functions.
	\end{theorem}
	
	In Section~\ref{ci2}, we prove this result using stochastic convex integration. The approach for substantiating the main iterative Proposition~\ref{p:iteration} closely resembles that of \cite{DelSze13}, which pioneered the use of convex integration to produce H\"{o}lder continuous solutions to deterministic Euler equations. The construction of building blocks for the convex integration scheme follows from \cite{DelSze13} and can be achieved without requiring the transport equations. As a result, no pathwise estimates are introduced in the inductive assumption and the stochastic convolution $z$ is not cut off as in \cite{LZ24}. As previously emphasized in \cite{LZ24}, our focus lies on entire solutions that satisfy the equations for all times $t\in\mR$. Accordingly, the norms within the stochastic convex integration scheme still retain the form
	$$
	\sup_{t\in\R}\mathbf{E}\left[\sup_{t\leq s \leq t+1}\|u(s)\|_{C^{\vartheta}}^p\right]<\infty,
	$$
	with $p\geq1$ and some $\vartheta>0$.
	Such bounds provide uniform moment estimates locally in $C(\R;C^{\vartheta})$ and guarantee the convergence of the corresponding ergodic averages. However, it is noteworthy that the regularity of the solutions presented in this paper does not surpass those obtained by combining stochastic convex integration with pathwise estimates, as outlined in \cite{LZ24}.
	
	\begin{remark}\label{remark}Note that the value of $\vartheta$ presented in \cite{LZ24} can be chosen within $\big(0,\min{\{\frac{\varkappa}{120{\cdot}7^5},\frac{1}{3{\cdot}7^5}\}}\big)$ under the same assumption as Theorem \ref{Theorem 1.1}.
		In comparison, the value of $\vartheta$ in this paper decreases since we require to estimate all the moments for stochastic convex integration, and we can set
		$\vartheta\in \big(0,\min{\{\frac{\varkappa}{73{\cdot}62^3},\frac{1}{7{\cdot}62^3}\}}\big) $.
	\end{remark}
	
	As we will see in the proof of Theorem~\ref{Theorem 1.1} below, the initial value for the solutions is a part of the construction and is not prescribed. 
	Therefore, the second aim of this paper is to investigate the Cauchy problem associated with \eqref{eul1}.  Our second result reads as follows and is proved in Section~\ref{sec 4}.
	\begin{theorem}\label{Thm1.3}
		Suppose that $\tr((-\Delta)^{3/2+\varkappa}GG^*)<\infty$ for some $\varkappa>0$,
		let $u_{0}\in C^\varkappa$ $\mathbf{P}$-a.s. be a  divergence-free initial condition independent of the Wiener process $B$. There exist infinitely many probabilistically strong and analytically weak solutions to \eqref{eul1} of class $u\in L^p_{\rm{loc}}([0,\infty);C^\vartheta) \cap C_{\rm{loc}}([0,\infty);H_{\sigma}^{-1})$ $\mathbf{P}$-a.s. for every $p\in[1,\infty)$, $\vartheta\in \big(0, \min{\{\frac{\varkappa}{12{\cdot}30^3p},\frac{1}{3{\cdot}30^3p}\}} \big) $ with initial condition $u|_{t=0}=u_0$. 
	\end{theorem}
	
	The proof of Theorem~\ref{Thm1.3} is based on the method in \cite[Section 5]{HZZ21markov}. The main challenge addressed by \cite{HZZ21markov} lies in controlling the interaction between the noise and the convex integration scheme. This was achieved by introducing a sequence of stopping times and gluing solutions across different time intervals. In the present paper, we do not employ the approach of gluing together another convex integration solution to obtain global-in-time solutions as in \cite{HZZ21markov}. Instead, we are inspired by \cite{HLP22} and \cite{ Umb23}, where a genuine convex integration procedure was developed and global solutions were constructed directly by using the stopping times only to control the growth of solutions in suitable H\"{o}lder spaces and providing uniform bound in $\omega \in \Omega$ for the solutions. This is more similar to the stochastic convex integration we have used in the first part, where global solutions are produced by controlling the growth in expectation of the norms of the solutions. Notably, this result serves as the stochastic counterpart of \cite{Wie11}, which established global weak solutions to the 3D Euler equations for arbitrary divergence-free initial data in $L^2$.

	\subsection*{Organization of the paper}
	In Section~\ref{notation}, we collect some basic notation used throughout the paper. Section~\ref{ci2} is devoted to Theorem~\ref{Theorem 1.1}: the stochastic convex integration is developed and employed to construct  global H\"{o}lder continuous solutions. Section~\ref{sec 4} is concerned with the proof of Theorem~\ref{Thm1.3}.
	In Appendix~\ref{Bw2}, we recall the construction of Beltrami waves from \cite{DelSze13}. In Appendix~\ref{ap:A}, we estimate amplitude functions used in the convex integration construction.

	\section{Preliminaries}\label{notation}
	Throughout this paper, we use the notation $a\lesssim b$ to indicate that there exists a constant $c>0$ such that $a\leq cb$. 
	\subsection{Function spaces}\label{s:2.1}
	For $p\in [1,\infty]$, we write $L^p:=L^p(\mathbb{T}^3;\mR^3)$ for
	the space of $L^p$-integrable functions from $\mathbb{T}^3$ to $\mR^3$ equipped with the usual $L^p$-norm. We set $L^{2}_{\sigma}:=\{f\in L^2; \int_{\mathbb{T}^{3}} f\,\dif x=0,\div f=0\}$. For $s\geq0$, $p\geq1$ the Sobolev space is denoted by $W^{s,p}:=\{f\in L^p; \|f\|_{W^{s,p}}:= \|(I-\Delta)^{{s}/{2}}f\|_{L^p}<\infty\}$. When $p=2$, $s\geq 0$, we write $H^s:=W^{s,2}$.  In particular, we write $H_{\sigma}^s:=W^{s,2}\cap L^2_\sigma$ for $s\geq0$. For $s<0$, we denote $H_{\sigma}^s$ to be the dual space of $H_{\sigma}^{-s}$. 
	
	We write $\mathcal{S}^{3\times3}$ for the set of symmetric $3\times3$ matrices and $\mathcal{S}_0^{3\times3}$ for the set of symmetric trace-free $3\times3$ matrices. We denote the trace-free part of the tensor product by $\mathring{\otimes}$. For a tensor $T$,  its traceless component is represented as $\mathring{T}:=T-\frac13\tr(T)\rm{Id}$. We denote the Helmholtz projection by $\mathbb{P}$. 
	We recall the inverse divergence operator $\mathcal{R}$ given in \cite[Definition 4.2]{DelSze13}, which acts on mean zero vector fields $v$ as
	\begin{equation*}
		(\mathcal{R}v)^{kl}=(\partial_k\Delta^{-1}v^l+\partial_l\Delta^{-1}v^k)-\frac{1}{2}(\delta_{kl}+\partial_k\partial_l\Delta^{-1})\div\Delta^{-1}v,
	\end{equation*}
	for $k,l\in\{1,2,3\}$. From \cite[Lemma 4.3]{DelSze13}, it follows that $\mathcal{R}v(x)\in \mathcal{S}_0^{3\times3}$ for each $x\in\mathbb{T}^3$ and $\div(\mathcal{R} v)=v$. By \cite[Theorem B.3]{CL22} we have for $1\leq p\leq \infty$
	\begin{align}\label{eR}
		\|\mathcal{R}f\|_{L^p(\mathbb{T}^3)}\lesssim \|f\|_{L^p(\mathbb{T}^3)}.
	\end{align}

	Let $(E,\|\cdot \|_E)$ be a real Banach space. For any $t\in\mR$, we denote by $C_tE:=C([t,t+1];E)$ the space of continuous functions $f: [t,t+1]\to E$, endowed with the norm $\|f\|_{C_tE}:=\sup_{s\in[t,t+1]}\|f(s)\|_{E}$.
	For $\kappa\in(0,1)$, we write $C^\kappa_tE$ for the space of $\kappa$-H\"{o}lder continuous functions from $[t,t+1]$ to $E$, equipped with the norm $$\|f\|_{C^\kappa_t E}:=\sup_{s,r\in [t,t+1],s\neq r}\frac{\|f(s)-f(r)\|_E}{|r-s|^\kappa}+\|f\|_{C_tE}\, .$$ 
	For a domain $D\subset\R$, we denote by $C^{N}_{D,x}$ the space of $N$-th differentiable functions from $D\times\mathbb{T}^{3}$ to $\mR^3$ with $N\in\N_{0}:=\mN\cup \{0\}$. These spaces are equipped with the norms
	$$
	\|f\|_{C^N_{D,x}}=\sum_{\substack{0\leq n+|\alpha|\leq N\\ n\in\N_{0},\alpha\in\N^{3}_{0} }}\sup_{t\in D}\|\partial_t^n \partial_x^\alpha f\|_{ L^\infty}\, .
	$$
	When $D=[t,t+1]$, we write $\|f\|_{C^{N}_{t,x}}:=\|f\|_{C^{N}_{[t,t+1],x}}$ for the sake of simplicity. Similarly, we write $C^N:=C^N(\mathbb{T}^3;\mR^3)$ for $N\in \N_{0}$, endowed with the norms
	\begin{align*}
		\|f\|_{C^0}:=\sup_{x\in \mathbb{T}^3}|f(x)|, \qquad
		\|f\|_{C^N}:=\sum_{\substack{0\leq |\alpha|\leq N,\,  \alpha\in\N^{3}_{0} }}\| \partial_x^\alpha f\|_{ L^\infty}, \quad N\geq 1.
	\end{align*}
	We also denote by $C^\kappa:=C^\kappa(\mathbb{T}^3;\mR^3)$ for $\kappa \in (0,1)$ the space of $\kappa$-H\"{o}lder continuous functions from $\mathbb{T}^{3}$ to $\mR^3$, equipped with the norm 
	\begin{align*}
		\|f\|_{C^\kappa}:= \|f\|_{C^0}+\sup_{x,y\in \mathbb{T}^3,x\neq y}\frac{|f(x)-f(y)|}{|x-y|^\kappa} \, .
	\end{align*}

	\subsection{Probabilistic elements}\label{s:2.2}
	Concerning the driving noise, we assume that $B$ is $\R^{3}$-valued two-sided  $GG^*$-Wiener process on some probability space $(\Omega, \mathcal{F},\mathbf{P})$ with spatial zero mean and divergence-free. Here, $G$ is a Hilbert--Schmidt operator mapping from $U$ to $L_{\sigma}^2$ for some Hilbert space $U$.
	In addition, we assume $\tr((-\Delta)^{3/2+\varkappa}GG^*)<\infty$ for some given $\varkappa>0$.
	
	For a given probability measure $\mathbf{P}$,  we use $\mathbf{E}$ to denote the expectation under $\mathbf{P}$. Given a Banach space $(E,\|\cdot \|_{E})$, for $p\in[1,\infty)$ and $\delta \in (0,\frac12)$, we denote the norms
	$$\$u\$^p_{E,p}:=\sup_{t\in \mathbb{R}}\mathbf{E}\left[\sup_{s\in [t,t+1]}\|u(s)\|^p_E\right], \quad  \$u\$^p_{C_t^{1/2-\delta}E,p}:=\sup_{t\in \mR} \mathbf{E}\left[ \|u\|^p_{C_t^{1/2-\delta}E}\right].$$
	In Section~\ref{ci2}, we shall use the corresponding norms with $E$ replaced by $L^1$, $C^0$, $C^\kappa$ and $H^{\frac32+\kappa}$ for some $\kappa>0$. Additionally, for $p\in[1,\infty)$, we denote 
	$$\$u\$^p_{C^1_{t,x},p}:=\sup_{t\in \mathbb{R}}\mathbf{E}\left[ \|u(s)\|^p_{C^1_{[t,t+1],x}}\right].$$

	\section{Construction of H\"{o}lder continuous solutions via stochastic convex integration}\label{ci2}
	This section is devoted to proving Theorem~\ref{Theorem 1.1} via a different approach from \cite[Theorem 1.2]{LZ24}. In Section~\ref{413}, we adopt the construction method of H\"{o}lder continuous Euler flows from \cite{DelSze13}, which eliminates the need for transport equations. Therefore, the inductive iteration no longer requires pathwise estimates as in \cite[Proposition 3.2]{LZ24} but only uses moment bounds as in \cite{HZZ22b}. The primary challenge of this approach is the presence of quadratic nonlinearity, which leads to superlinear estimates.
	Specifically, estimating any $p$-th moment at the level $q+1$ inevitably involves higher moments at the level $q$.
	Thus, meticulous tracking of all the moment bounds is crucial, with particular attention to the involved constants, what they depend on, and their precision. Otherwise, it would not be possible to close the iteration. Compared to \cite{LZ24}, this approach may lead to lower regularity solutions since estimating all the moments is necessary.
	
	In this section, we fix a probability space $(\Omega,\mathcal{F},\mathbf{P})$ and a $\mathbb{R}^3$-valued two-sided $GG^*$-Wiener process $B$ with mean zero and divergence-free. We let $(\mathcal{F}_t)_{t\in \mR}$ be the normal filtration generated by $B$, that is, the  canonical right continuous filtration augmented by all the $\mathbf{P}$-negligible sets (see e.g. \cite[Section 2.1]{LR15}). 
	At the first step, we decompose the stochastic Euler system \eqref{eul1} into two parts, one is linear and contains the stochastic integral, whereas the second one is a nonlinear but random PDE. To be more specific, we consider the stochastic linear equation
	\begin{equation}\label{li:sto}
		\aligned
		\dif z+z \dif t&=\dif B,
		\\\div z&=0,
		\endaligned
	\end{equation}
	where $z$ is divergence-free by the assumptions on the noise $B$, and $v$ solves the nonlinear equation
	\begin{equation}\label{nonlinear}
		\aligned
		\partial_t v-z+\div((v+z)\otimes (v+z))+\nabla P&=0,
		\\\div v&=0.
		\endaligned
	\end{equation}
	We denote  by $P$ the pressure term associated with $v$.
	
	Using the  factorization method, it is standard to derive the regularity of the stochastic convolution $z$ on the given stochastic basis $(\Omega,\mathcal{F}, (\mathcal{F}_t)_{t\in \mR}, \mathbf{P})$. In particular, the following result follows from \cite[Proposition 3.1]{LZ24}.
	\begin{lemma}\label{Le1}
		Suppose that $\tr((-\Delta)^{3/2+\varkappa}GG^*)<\infty$ for some  given $\varkappa>0$. Then for any $\delta\in(0,\frac12)$, $p\geq 2$
		\begin{align}\label{z1/2}
			\sup_{t\in \mR} \mathbf{E}\left[ \|z\|^p_{C_{t}^{1/2-\delta}{C^\varkappa}}\right] \lesssim \sup_{t\in \mR} \mathbf{E}\left[ \|z\|^p_{C_{t}^{1/2-\delta}{H^{3/2+\varkappa}}}\right] \leq (p-1)^{p/2}L^p,
		\end{align}
		where $L\geq1$ depends on $\tr((-\Delta)^{3/2+\varkappa}GG^*)$, $\delta$ and is independent of $p$.
	\end{lemma}

	Let us now explain how the convex integration iteration is set up. First, we consider a frequency $\lambda_q  $ and an amplitude $\delta_q  $, which are given by 
	$$\lambda_0=a,\quad \lambda_q=a^{b^{q}}, \ q\geq
	1, \qquad   \delta_1=6L^2 , \quad \delta_q=\frac12  \lambda_2^{2\beta}\lambda_{q}^{-2\beta}, \ q\geq 2.$$
	Here, the parameters $a,b \in \mathbb{N}$ are selected to be sufficiently large, while $\beta \in (0,1)$ is chosen as a small parameter. Further details on the choice of these parameters are provided during the construction.
	
	The iteration is indexed by a parameter $q\in \mathbb{N}_0$.  At each step $q$, we construct a pair $(v_q ,\mathring{R}_q)$ to solve the following system
	\begin{equation}\label{euler}
		\aligned
		\partial_t v_q -z_q+\div((v_q+z_q)\otimes (v_q+z_q))+\nabla p_q&=\div
		\mathring{R}_q,
		\\\div v_q&=0.
		\endaligned
	\end{equation}
	In the above, we define $z_q=\mathbb{P}_{\leq f(q)}z$ with $\mathbb{P}_{\leq f(q)}$ being the Fourier multiplier operator. Here, we choose
	$f(q)=\lambda_{q+1}^{\gamma/8}$ and $\gamma$ is a small parameter that will be fixed in Section~\ref{311}. 	In addition, $\mathring{R}_q\in \mathcal{S}_0^{3\times3}$ and we put the trace part into the pressure. Moreover, by Sobolev embedding and \eqref{z1/2}, it is easy to obtain for any $p\geq2$ and $\delta \in (0,\frac12)$
	\begin{equation}\label{estimate-zq'}
		\aligned
		\$z_q\$_{C^0,p}&\lesssim \$z\$_{H^{3/2+\varkappa},p}\lesssim (p-1)^{1/2}L,
		\\ \$z_q\$_{C^1,p} &\lesssim \lambda_{q+1}^{\frac{\gamma}{8}} \$z\$_{H^{3/2+\varkappa},p} \lesssim (p-1)^{1/2}L\lambda_{q+1}^{\frac{\gamma}{8}},
		\\ \$z_q\$_{C^{1/2-\delta}_tC^0,p}&\lesssim \$z\$_{C^{1/2-\delta}_tH^{3/2+\varkappa},p} \lesssim (p-1)^{1/2}L.
		\endaligned
	\end{equation}
	
	Under the above assumptions, our main iteration reads as follows, whose proof will be completed in Sections~\ref{311}--\ref{s:en}.
	\begin{proposition}\label{p:iteration} 
		Assume that $\tr((-\Delta)^{3/2+\varkappa} GG^*)<\infty$ for some $\varkappa>0$ and fix a constant $r\in(1,\frac32)$. Let $e:\R\to(0,\infty)$ be a smooth function  satisfying $\bar e\geq e(t)\geq \underline{e}\geq 4{\cdot}48{ \cdot}(2\pi)^3r_R^{-1}L^2$  with $\|e'\|_{C^{0}}\leq \tilde e$ for some constants $\overline{e},\underline{e},\tilde{e}>0$, where $r_R>0$ is a small constant (see \eqref{choice rR} below) and $L$ is from Lemma~\ref{Le1}. There exists a choice of parameters $a,b,\beta$ and $\gamma$ such that the following holds true: Let $(v_q ,\mathring{R}_q)$ for some $q\in \mathbb{N}_0$ be an $(\mathcal{F}_t)_{t\in  \mR}$-adapted solution to \eqref{euler} satisfying 
		\begin{subequations}
			\begin{align}
				\$v_q\$_{C^0,n} &\leq \lambda_q^{\frac{n}{3}\beta},\label{vqa}
				\\ \$v_q\$_{C_{t,x}^1,n}&\leq \lambda_q^{1+n\beta},\label{vqb}
				\\ \$\mathring{R}_q\$_{C^0,n} &\leq 
				\begin{cases}
					\delta_{q+1}, & n=1,2,\\
					\lambda_q^{n\beta}, &n\geq 3,
				\end{cases}\label{vqc}
				\\ \$\mathring{R}_q\$_{L^1,1}&\leq
				\frac1{48}\delta_{q+2}r_R\underline{e}.\label{vqd}
			\end{align}
		\end{subequations}
		Moreover, for any $t\in \mR$
		\begin{equation}\label{vqe}
			\aligned
			\frac34\delta_{q+1}e(t)\leq e(t)-\E\|(v_q+z_q)(t)\|_{L^2}^2\leq \frac54\delta_{q+1}e(t),\quad & q\geq1,
			\\ 	\E\|v_0(t)+z_0(t)\|^2_{L^2} \leq \frac12e(t), \quad &q=0. 
			\endaligned
		\end{equation}
		Then there exist $(\mathcal{F}_t)_{t\in \mR}$-adapted processes $(v_{q+1}, \mathring{R}_{q+1})$ which solve \eqref{euler}, obey \eqref{vqa}-\eqref{vqd} and \eqref{vqe} at level $q+1$ and we have
		\begin{equation}\label{vq+1-vq2}
			\$v_{q+1}-v_q\$_{C^0,2r}\leq \bar{M}\delta_{q+1}^\frac{1}{2},
		\end{equation}
		where the universal constant $\bar{M}$ is fixed throughout the iteration (see \eqref{def M2} below).
	\end{proposition}
	
	\subsection{Proof of the main result} 
	Our main result is the following theorem, which implies Theorem~\ref{Theorem 1.1}.
	
	\bt\label{v}
	Assume that $\tr((-\Delta)^{3/2+\varkappa} GG^*)<\infty$ for some $\varkappa>0$ and fix a constant $r\in(1,\frac32)$. Let $e:\R\to(0,\infty)$ be a smooth function  satisfying $\bar e\geq e(t)\geq \underline{e}\geq 4{\cdot}48 {\cdot}(2\pi)^3c_R^{-1}L^2$  with $\|e'\|_{C^{0}}\leq \tilde e$ for some constants $\overline{e},\underline{e},\tilde{e}>0$. There exists an $(\mathcal{F}_t)_{t\in\mR}$-adapted process $u$ which is an analytically weak solution to \eqref{eul1} and belongs to $C(\mR;C^{\vartheta})$ $\mathbf{P}$-a.s.  for some $\vartheta>0$. Moreover, we have
	\begin{align}\label{vtheta:1}\$u\$_{C^{\vartheta},2r}<\infty,\end{align} 
	and for all $t\in\mR$
	
	\begin{align}\label{eq:K1}
		\mathbf{E}\|u(t)\|_{L^2}^2=e(t).
	\end{align}
	\et
	
	\begin{proof}
		We begin the initial iteration from $v_{0}\equiv 0$ on $\mR$ and take
		$\mathring{R}_0=z_{0}\mathring\otimes z_{0}-\mathcal{R}z_{0}$.
		By using \eqref{eR}, \eqref{estimate-zq'}, and $\delta_1=6L^2$, we have 
		\begin{align*}
			\$\mathring{R}_0\$_{C^0,n}\leq \$z_0\$^2_{C^0,2n}+ \$z_0\$_{C^0,n}
			\leq 2(2n-1)L^2 \leq
			\begin{cases}
				\delta_1,&n=1,2,\\
				\lambda_0^{n\beta},&n\geq 3,
			\end{cases}
		\end{align*}
		where we chose $a$ sufficiently large such that $2(2n-1)L^2\leq a^{n\beta}$ in the last inequality. Taking $n=1$ we also have
		\begin{align*}
			\$\mathring{R}_0\$_{L^1,1} \leq (2\pi)^3\$\mathring{R}_0\$_{C^0,1}\leq 2{\cdot}(2\pi)^3L^2\leq \frac{1}{96}r_R\underline{e}=\frac{1}{48}\delta_2r_R\underline{e},
		\end{align*}
		since $\delta_2=\frac12$ and $4{\cdot}48 {\cdot}(2\pi)^3r_R^{-1}L^2\leq \underline{e}$. Hence, \eqref{vqc} and \eqref{vqd} hold at the level  $q=0$. 
		Similarly, by the Sobolev embedding and \eqref{estimate-zq'} we also have
		\begin{align*}
			\E\|z_0(t)\|^2_{L^2} \leq (2\pi)^3\E\|z_0(t)\|^2_{C^0}\leq (2\pi)^3 \E\|z(t)\|^2_{H^{3/2+\varkappa}} \leq(2\pi)^3L^2\leq \frac{1}{2}e(t).
		\end{align*}
		Thus, \eqref{vqe} holds at the level $q=0$. As we will see in \eqref{defzetaq} below, $\zeta_0$ is well-defined and the iteration can  be carried out.
		
		By using Proposition \ref{p:iteration} with the initial iteration $(v_0,\mathring{R}_0)$, we obtain a sequence of solutions $(v_q,\mathring{R}_q)$ satisfying \eqref{vqa}--\eqref{vqd}, \eqref{vqe} and \eqref{vq+1-vq2}. We use \eqref{vqb}, \eqref{vq+1-vq2}, and interpolation to deduce for any $\vartheta \in (0,\frac{\beta}{1+4\beta})$
		\begin{align*}
			\sum_{q\geq0}\$v_{q+1}-v_q\$_{C^{\vartheta},2r}
			&\leq \sum_{q\geq0}\$v_{q+1}-v_q\$^{1-\vartheta}_{C^0,2r} \$v_{q+1}-v_q\$^{\vartheta}_{C^1,2r}
			\\ &\lesssim \sum_{q\geq0} \delta_{q+1}^{\frac{1-\vartheta}{2}} \lambda_{q+1}^{\vartheta+3\vartheta \beta} 
			\leq \sqrt{6}La+\lambda_2^\beta \sum_{q\geq1} \lambda_{q+1}^{4\vartheta \beta +\vartheta-\beta} <\infty.
		\end{align*}
		Hence, the above series is summable, and we may define a limiting function $v=\lim_{q\to \infty}v_q$ which lies
		in $L^{2r}(\Omega;C(\mR;C^{\vartheta}))$. The limit
		$v$ is $(\mathcal{F}_t)_{t\in\mR}$-adapted follows from that $v_q$ is $(\mathcal{F}_t)_{t\in\mR}$-adapted for every $q\in\mathbb{N}_{0}$. For the same $\vartheta$ as above and any $p\geq1$, using \eqref{estimate-zq'}, \eqref{z2} and interpolation again we have 
		\begin{align*}
			\sum_{q\geq0}\$z_{q+1}-z_q\$_{C^{\vartheta},p}
			&\leq \sum_{q\geq0}\$z_{q+1}-z_q\$^{1-\vartheta}_{C^0,p} \$z_{q+1}-z_q\$^{\vartheta}_{C^1,p}
			\\ &\lesssim p\sum_{q\geq0} \lambda_{q+1}^{-\frac{\gamma}{8}\varkappa(1-\vartheta)} \lambda_{q+2}^{\frac{\gamma}{8}\vartheta} 
			= p\sum_{q\geq1} \lambda_{q+1}^{-\frac{\gamma}{8}\varkappa(1-\vartheta)+\frac{\gamma}{8}b\vartheta}<\infty.
		\end{align*}
		The  last inequality is ensured by $\vartheta<\frac{\varkappa}{b+\varkappa}$, thus we obtain $\lim_{q\to \infty}z_q=z$ in $L^p(\Omega;C(\mR;C^\vartheta))$ for any $p\geq1$. 
		Furthermore, by \eqref{vqc} we deduce $\lim_{q\to \infty}\mathring{R}_q=0$ in $L^1(\Omega;C(\mR;C^0))$. Thus, $v$ is an analytically  weak solution to \eqref{nonlinear}, and letting $u=v+z$ we obtain an $(\mathcal{F}_{t})_{t\in\mR}$-adapted analytically weak solution to \eqref{eul1}. Moreover, the estimate \eqref{vtheta:1} for $u$ holds and \eqref{eq:K1} follows from \eqref{vqe}. This concludes the proof of part (a) of Theorem~\ref{Theorem 1.1}. Utilizing the above result and applying the same argument as in \cite[Theorem 6.1]{LZ24}, we can also derive stationary solutions $\bar{u} \in C(\mR;C^{\frac{\vartheta}{2}})$ to \eqref{eul1}. This completes the proof of part (b) of Theorem~\ref{Theorem 1.1}.
	\end{proof}

	In the remainder of this section, we aim to prove Proposition~\ref{p:iteration}
	in several main steps, which are similar to many convex integration schemes. First, we fix the parameters used during the construction and continue with a mollification step in Section~\ref{311}. After that, we introduce the new iteration $v_{q+1}$ in Section~\ref{413}, which is the main part of the construction that differs from the convex integration scheme in \cite{LZ24}. In this scheme, we adopt a similar approach as \cite{DelSze13} to construct the H\"{o}lder continuous Euler flows. We do not incorporate the solutions for the transport equations into the new perturbation $w_{q+1}$. Hence, pathwise estimates are no longer needed and we only use moment bounds in the inductive iteration. Section~\ref{315} mainly contains the inductive moment estimates of $v_{q+1}$. In Section~\ref{316}, we show how to construct the new stress $\mathring{R}_{q+1}$. We establish the inductive moment estimates on $\mathring{R}_{q+1}$ in Section~\ref{31}. Finally, in Section~\ref{s:en}  we show how the energy is controlled.
	
	\subsection{Choice of parameters and Mollification}\label{311}
	In this section, we mainly choose parameters and conduct mollification.
	First, for a fixed integer $b\geq 62$,  we choose $\beta, \gamma \in (0,1)$ to be small parameters such that
	\begin{equation*}
		3b^3\beta +2b\gamma<1, \qquad
		2b^2\beta+\beta<\frac{\gamma\varkappa}{8},
	\end{equation*}
	which can be obtained by choosing $\gamma=\frac{1}{4b}$ and $\beta$ sufficiently small satisfying
	$\beta<\min{\{\frac{\varkappa}{72{\cdot}b^3},\frac{1}{6{\cdot}b^3}\}}$. 
	The last free parameter $a\in 2^{32b\mathbb{N}} $ is chosen large enough such that for any $ n \in \mathbb{N}$
	\begin{align*}
		a^{(b-61)\beta n}>(8n)^3L^6,
	\end{align*}
	where $L$ is the constant in Lemma~\ref{Le1}. The choice of $a$ also ensures $f(q)\in \mathbb{N}$.
	In the following sections, we increase $a$ to absorb implicit constants.

	Next, we proceed with mollification and related estimates.  In order to avoid a loss of derivative, we replace $v_q$ with a mollified velocity field $v_\ell$. To this end, we choose a small parameter $\ell= \lambda_q^{-4}$ and define a mollification of $v_q$, $\mathring{R}_q$ and $z_q$ by
	$$v_\ell=(v_q*_x\phi_\ell)*_t\varphi_\ell,\qquad
	\mathring{R}_\ell=(\mathring{R}_q*_x\phi_\ell)*_t\varphi_\ell,\qquad
	z_\ell=({z_q}*_x\phi_\ell)*_t\varphi_\ell,$$
	where $\phi_\ell=\frac{1}{\ell^3}\phi(\frac{\cdot}{\ell})$ is a space mollifier on $\mathbb{R}^3$ and $\varphi_\ell=\frac{1}{\ell}\varphi(\frac{\cdot}{\ell})$ is a time mollifier with support in $(0,1)$.
	From \eqref{euler}, it follows that $(v_\ell ,\mathring{R}_\ell)$ satisfies
	\begin{equation}\label{mollification}
		\aligned
		\partial_t v_\ell -z_\ell +\div((v_\ell+z_\ell)\otimes (v_\ell+z_\ell))+\nabla p_\ell&=\div (\mathring{R}_\ell+R^\textrm{com}),
		\\\div v_\ell &=0,
		\endaligned
	\end{equation}
	where we have defined
	\begin{equation}\label{Rcom}
		R^{\textrm{com}}=(v_\ell+z_\ell)\mathring{\otimes}(v_\ell+z_\ell)-((v_q+z_q)\mathring{\otimes}(v_q+z_q))*_x\phi_\ell*_t\varphi_\ell,
	\end{equation}
	\begin{equation*}
		p_\ell=(p_q*_x\phi_\ell)*_t\varphi_\ell-\frac{1}{3}\big(|v_\ell+z_\ell|^2-(|v_q+z_q|^2*_x\phi_\ell)*_t\varphi_\ell).
	\end{equation*}
	Using a standard mollification estimate,  we have for any $t\in \R$
	\begin{align}\label{error}
		\|(v_q-v_\ell)(t)\|_{L^\infty}\lesssim \ell \|v_q\|_{C^1_{[t-1,t],x}}.
	\end{align}
	Taking expectation and using inductive assumption \eqref{vqb} and $2r<3$, we obtain
	\begin{align}\label{vq-vl}
		\$v_q-v_\ell\$_{C^0,2r}\lesssim \ell \$v_q\$_{C^1_{t,x},2r}  \leq \ell\$v_q\$_{C^1_{t,x},3} \leq\lambda_{q}^{-3+3\beta} \leq \frac{1}{4}\delta_{q+1}^\frac{1}{2},
	\end{align}
	where we used $b\beta+3\beta < 1$ in the last inequality and we chose $a$ sufficiently large to absorb the implicit constant.
	\subsection{Perturbation of the velocity}\label{413}
	As usual in convex integration schemes, our idea is to define $v_{q+1}$ as a perturbation of $v_\ell$. To this end, we employ the H\"{o}lder continuous Euler flows constructed in \cite[Section 4.1]{DelSze13} and the building blocks of the perturbation are Beltrami waves recalled in Appendix~\ref{Bw2}. We proceed with the construction by defining the amplitude functions $a_k$ and providing some estimates for $a_k$. 
	\subsubsection{Construction of amplitude functions $a_k$}\label{313} 
	We begin with recalling the construction of transport coefficients $\phi_{k}^{(j)}$ from \cite{DelSze13}. We choose two constants $c_1$ and $c_2$ such that $\frac{\sqrt{3}}{2}<c_1<c_2<1$ and fix a function $\varphi \in C^{\infty}_c(B_{c_2}(0))$ which is nonnegative and identically 1 on the ball $B_{c_1}(0)$. We next consider the lattice $\mathbb{Z}^3 $ and its quotient by $(2\mathbb{Z})^3$, and denote by $\mathcal{C}_j$, $j=1,\dots,8$ the eight equivalence classes of $\mathbb{Z}^3/(2\mathbb{Z})^3$. For each $k\in \mathbb{Z}^3$ we denote by $\varphi_k$ the function
	$$\varphi_k(y):=\varphi(y-k).$$
	Observe that, if $k\neq l \in \mathcal{C}_j$, then $|k-l|\geq2>2c_2$, hence $\varphi_k$ and $\varphi_l$ have disjoint supports. On the other hand, the function 
	\begin{align*}
		\psi:=\sum_{k\in \mathbb{Z}^3}\varphi_k^2
	\end{align*}
	is smooth, bounded and bounded away from zero. We then define, for $\tau \in \mR$ and $y\in \mR^3$:
	\begin{align*}
		\alpha_k(y)&:=\frac{\varphi_k(y)}{\sqrt{\psi(y)}}, \qquad k\in \mathbb{Z}^3,
		\\ \phi_k^{(j)}(\tau,y)&:=\sum_{l\in\mathcal{C}_j}\alpha_l(\mu y)e^{-i(k\cdot\frac{l}{\mu})\tau}, \quad k\in \mathbb{Z}^3, \ j=1, \dots ,8.
	\end{align*}
	For any fixed $j$ and $(\tau,y)\in \mR \times \mR^3$, observe that there exists at most one $l\in \mathcal{C}_j$ such that $\alpha_l(\mu y)\neq0$ and this $l$ has the property that $|\mu y-l|<1$. Thus, in a neighborhood of $(\tau,y)$, we have
	\begin{align}\label{pro phi}
		\partial_\tau \phi_k^{(j)}+i(k\cdot y)\phi_k^{(j)}=ik\cdot\left( y-\frac{l}{\mu} \right) \phi_k^{(j)}.
	\end{align}
	Moreover, since $\alpha_l$ and $\alpha_{\tilde{l}}$ have disjoint supports for $l\neq \tilde{l}\in \mathcal{C}_j$, it follows that for all $\tau,y$ as above and $k\in \mathbb{Z}^3$
	\begin{align*}
		\sum_{j=1}^{8}|\phi_k^{(j)}(\tau,y)|^2=\sum_{j=1}^{8}\sum_{l\in\mathcal{C}_j}\alpha_l(\mu y)^2=1,
	\end{align*}
	and 
	\begin{align}\label{phi1}
		\sup_{y,\tau}|D_y^r \phi_k^{(j)}(\tau,y)|\leq C\mu^r, \quad j=1, \dots,8,
	\end{align}
	with the constant $C=C(r,|k|)$ depending only on $r\in\mathbb{N}_0$ and $|k|$. More generally, in \cite[Proposition 4.2]{DLS14} it is proved that for any $r,h\in \N_0$ and $k\in\mathbb{Z}^3$ the derivatives of $\phi_k^{(j)}$ with respect to $\tau$ are controlled on the
	set $|y|\leq Y$ by
	\begin{align}\label{phi2}
		\sup_{|y|\leq Y,\tau}\big|D_y^r \partial_\tau^h \phi_k^{(j)}(\tau,y)\big|\leq \sup_{|y|\leq Y}C (|y|+1)^h\mu^r \leq C (Y+1)^h\mu^r , \quad j=1, \dots ,8,
	\end{align}
	where $C = C(r, h, |k|)$ and $Y> 0$ is any given constant. We also recall the following estimate on the material derivative of $\phi_{k}^{(j)}$ from \cite[(27)]{DelSze13}. For any $r\in \N_0$ and $k\in\mathbb{Z}^3$, there exists a constant $C=C(r,|k|)$ such that for any $j=1,\dots,8$ it holds
	\begin{align}\label{phi3}
		\sup_{y,\tau}\big|D_y^r( \partial_\tau \phi_k^{(j)}+i(k\cdot y)\phi_k^{(j)})\big|\leq C\mu^{r-1}.
	\end{align}
	In the sequel, we take $\mu=\lambda_{q}$ at the $q$ step.
	
	In order to define the amplitude functions, we first apply the Lemma~\ref{Belw2} with $N=8$ to obtain $\lambda_0>1$, $r_0>0$ and pairwise
	disjoint families $\Lambda_j \subset \{k\in \mathbb{Z}^3:|k|=\lambda_0\}$ for $j\in \{
	1,\dots,8\}$ together with corresponding smooth functions 
	$	\gamma_k^{(j)}\in C^\infty(B_{r_0}(\mathrm{Id}))$, $k\in \Lambda_j$. We then define $\rho$ as
	\begin{equation}\label{eq:rho3}
		\rho(s,y):=\sqrt{\ell^2+|\mathring{R}_\ell(s,y)|^2}+r_0\zeta_\ell(s),
	\end{equation}\label{zetaq}
	where $	\zeta_\ell:=\zeta_q*_s\varphi_\ell$ and
	\begin{align}\label{defzetaq}
		\zeta_q(s):=\frac1{3{\cdot} (2\pi)^3} \Big[ e(s)(1-\delta_{q+2})-\E\|v_q(s)+{z_q}(s)\|_{L^2}^2\Big],
	\end{align}
	and we take the constant $r_R$ in \eqref{vqd} satisfying 
	\begin{align}\label{choice rR}
		r_R<r_0.
	\end{align}
	Finally, the amplitude functions $a_k$, $k\in \Lambda:=\cup_{j=1}^8\Lambda_j$, are defined by
	\begin{equation}\label{defak}
		\aligned
		a_k(s,y,\tau)&=\mathbf{1}_{\{k\in \Lambda_j\}}r_0^{-\frac{1}{2}}\rho^\frac12 (s,y)\gamma_k^{(j)}\left(\Id- \frac{r_0\mathring{R}_\ell(s,y)}{\rho(s,y)}\right)
		\phi_k^{(j)}( \tau,(v_\ell+z_\ell)(s,y)) 
		\\&=\mathbf{1}_{\{k\in \Lambda_j\}}h_k(s,y)\phi_k^{(j)}( \tau,(v_\ell+z_\ell)(s,y)) ,
		\endaligned
	\end{equation}
	where $h_k(s,y)=r_0^{-\frac{1}{2}}\rho^\frac12 (s,y)\gamma_k^{(j)}\left(\Id- \frac{r_0\mathring{R}_\ell(s,y)}{\rho(s,y)}\right)$. By the definition of $\rho$, we have
	\begin{align*}
		\left\|\Id-\left(\Id- \frac{r_0\mathring{R}_\ell(s,y)}{\rho(s,y)} \right)\right\|_{C^0} \leq r_0.
	\end{align*}
	As a result, $\mathrm{Id}-r_0\rho^{-1}\mathring{R}_{\ell}$ lies in the domain of the function $\gamma_{k}^{(j)}$
	which ensures the conditions in Lemma~\ref{Belw2} are satisfied. Furthermore, we provide bounds for $C_{s,y}^N$ norms of the amplitude functions $a_k$ and their material derivatives $\partial_\tau a_k+ik\cdot(v_\ell+z_\ell)a_k$. This is done in forthcoming Proposition ~\ref{estiak}, whose proof is given in Appendix~\ref{ap:A}.
	\begin{proposition}\label{estiak}
		Let $a_k$ be given by \eqref{defak}, $k\in \Lambda$. Then we obtain for $N\geq1$
		\begin{align}
			\|a_k\|_{C^N_{s,y}}&\lesssim\ell^{-2N-\frac{1}{2}}(1+\|v_q\|_{C^0_{[s-1,s+1],y}}+\|z_q\|_{C^0_{[s-1,s+1],y}})^N(1+\|\mathring{R}_q\|_{C^0_{[s-1,s+1],y}})^{N+1},\label{a_k1}
			\\	\|\partial_\tau a_k\|_{C^N_{s,y}} 
			&\lesssim\ell^{-2N-\frac{1}{2}}(1+\|v_q\|_{C^0_{[s-1,s+1],y}}+\|z_q\|_{C^0_{[s-1,s+1],y}})^{N+1}(1+\|\mathring{R}_q\|_{C^0_{[s-1,s+1],y}})^{N+1},	\label{a_k2}
		\end{align}
		\begin{equation}\label{derivat a_k}
			\aligned
			\|&\partial_\tau a_k+ik\cdot( v_\ell+z_\ell)a_k\|_{C^N_{s,y}} 
			\\&\lesssim \ell^{-2N-\frac{1}{2}}(1+\|v_q\|_{C^0_{[s-1,s+1],y}}+\|z_q\|_{C^0_{[s-1,s+1],y}})^{N}(1+\|\mathring{R}_q\|_{C^0_{[s-1,s+1],y}})^{N+1}.
			\endaligned
		\end{equation}
		And for $N=0$ we have
		\begin{align}
			\|a_k\|_{C^0_{s,y}}&\lesssim \|\mathring{R}_q\|^{\frac{1}{2}}_{C^0_{[s-1,s+1],y}}+\ell^{\frac{1}{2}}+\delta_{q+1}^\frac12,\label{a_k0}
			\\
			\|\partial_\tau a_k\|_{C^0_{s,y}}&\lesssim (1+\|\mathring{R}_q\|^{\frac{1}{2}}_{C^0_{[s-1,s+1],y}})(1+\|v_q\|_{C^0_{[s-1,s+1],y}}+\|z_q\|_{C^0_{[s-1,s+1],y}}),\label{tauak0}
			\\
			\|\partial_\tau a_k+ik\cdot( v_\ell+z_\ell)a_k\|_{C^0_{s,y}} &\lesssim \lambda_{q}^{-1}(\|\mathring{R}_q\|^{\frac{1}{2}}_{C^0_{[s-1,s+1],y}}+\ell^\frac{1}{2}+\delta_{q+1}^\frac12), 	\label{deriva ak0}
		\end{align}
		where all the implicit constants are independent of $q$.
	\end{proposition}
	
	\subsubsection{Construction of $v_{q+1}$}\label{314}
	With these preparations in hand, let us now proceed with the construction of the perturbation $w_{q+1}$, which then defines the next iteration by $v_{q+1}:=v_{\ell}+w_{q+1}$. First, we define the principal part $w_{q+1}^{(p)}$ of the perturbation $w_{q+1}$ as
	\begin{equation}\label{defwp}
		\aligned
		w_{q+1}^{(p)}(t,x)&:=W(t,x,\lambda_{q+1}t,\lambda_{q+1}x),
		\\W(s,y,\tau,\xi)&:=\sum_{k\in\Lambda} a_k(s,y,\tau)B_k e^{ik\cdot \xi},
		\endaligned
	\end{equation}
	where $B_k$ is a suitable vector defined in Appendix~\ref{Bw2}. Since the coefficients $a_k$ are $(\mathcal{F}_t)_{t\in \mR}$-adapted and $B_k e^{i\lambda_{q+1} k\cdot x}$ is a deterministic function we deduce that
	$w_{q+1}^{(p)}$ is also $(\mathcal{F}_t)_{t\in \mR}$-adapted.
	By a direct computation, we deduce that
	\begin{align*}
		w_{q+1}^{(p)}(t,x)=\frac{1}{\lambda_{q+1}}\textrm{curl} \left( \sum_{k\in \Lambda}  ia_k \frac{k\times B_k}{|k|^2}e^{i\lambda_{q+1}k\cdot x}\right) - \frac{1}{\lambda_{q+1}}\sum_{k\in \Lambda} \nabla a_k \times \frac{B_k}{|k|}e^{i\lambda_{q+1}k\cdot x}.
	\end{align*}
	Hence, we also define the incompressibility corrector by
	\begin{align}\label{defwc}
		w_{q+1}^{(c)}(t,x)=\frac{1}{\lambda_{q+1}}\sum_{k\in \Lambda}  \nabla a_k \times \frac{B_k}{|k|}e^{i\lambda_{q+1}k\cdot x}.
	\end{align}
	Similarly, $w_{q+1}^{(c)}$ is also $(\mathcal{F}_t)_{t\in \mR}$-adapted. 
	Then we define the total perturbation as 
	\begin{equation}\label{wq+1}
		w_{q+1}:=w_{q+1}^{(p)}+w_{q+1}^{(c)}=\frac{1}{\lambda_{q+1}}\textrm{curl} \left( \sum_{k\in \Lambda}  ia_k \frac{k\times B_k}{|k|^2}e^{i\lambda_{q+1}k\cdot x}\right),
	\end{equation}
	which is mean zero, divergence-free and $(\mathcal{F}_t)_{t\in\mR}$-adapted. The new velocity $v_{q+1}$ is defined as
	\begin{equation}\label{vq}
		v_{q+1}:=v_\ell+w_{q+1}.
	\end{equation}
	Thus, $v_{q+1}$ is also $(\mathcal{F}_t)_{t\in\mR}$-adapted.
	\subsection{Inductive estimates for $v_{q+1}$}\label{315}
	In this section, we shall verify the inductive estimates \eqref{vqa} and \eqref{vqb} on the level $q+1$. We first estimate $\$w_{q+1}\$_{C^0,n}$, it follows from \eqref{a_k0}, \eqref{defwp} and \eqref{A5} that for any $t\in \mR$
	\begin{align*}
		\|w_{q+1}^{(p)}\|_{C^0_{t,x}}\lesssim M|\Lambda| (r_0^{-\frac12}\ell^{\frac{1}{2}}+r_0^{-\frac12}\|\mathring{R}_q\|^{\frac{1}{2}}_{C^0_{[t-1,t+1],x}}+\bar{e} \delta_{q+1}^{\frac12} ),
	\end{align*}
	where $M$ is a universal constant given in \eqref{A5} and $|\Lambda|$ is the
	cardinality of the set $\Lambda$.
	Taking expectation and using \eqref{vqc} implies for $n\geq1$
	\begin{equation}\label{wp0}
		\begin{aligned}
			\$w_{q+1}^{(p)}\$_{C^0,n}&\lesssim r_0^{-\frac12}|\Lambda|M (\$\mathring{R}_q\$_{C^0,n}^\frac{1}{2}+\ell^\frac{1}{2})+|\Lambda|M\bar{e} \delta_{q+1}^{\frac12}
			\\&  \leq r_0^{-\frac12}|\Lambda|M (\lambda_{q}^{\frac{n}{2}\beta}+\ell^\frac{1}{2})+|\Lambda|M \bar{e} \delta_{q+1}^{\frac12}
			\leq \frac13 \lambda_{q+1}^{\frac{n}{3}\beta},
		\end{aligned}
	\end{equation}
	where we used $b\geq3$ and $a$ sufficiently large to absorb the constant. 
	Similarly, using \eqref{vqc} and $\$\mathring{R}_q\$_{C^0,r}\leq\$\mathring{R}_q\$_{C^0,2}\leq \delta_{q+1}$, we also obtain
	\begin{equation}\label{wpr}
		\aligned
		\$w_{q+1}^{(p)}\$_{C^0,2r}&\lesssim r_0^{-\frac12}|\Lambda|M (\$\mathring{R}_q\$_{C^0,r}^\frac{1}{2}+\ell^\frac{1}{2})+|\Lambda|M\bar{e} \delta_{q+1}^{\frac12}
		\\&\leq r_0^{-\frac12}|\Lambda|M (\delta_{q+1}^{\frac{1}{2}}+\ell^\frac{1}{2})+|\Lambda|M\bar{e} \delta_{q+1}^{\frac12}
		\leq \frac{1}{4}\bar{M}\delta_{q+1}^{\frac{1}{2}},
		\endaligned
	\end{equation}
	where we used $\ell\leq \delta_{q+1}$, and $\bar{M}$ is a universal constant satisfying  
	\begin{align}\label{def M2}
		4|\Lambda|M\bar{e}+ 8r_0^{-\frac12}|\Lambda|M<\bar{M}.
	\end{align}
	We also chose $a$ sufficiently large to absorb the constant. 
	We then use \eqref{a_k1} and \eqref{defwc} to obtain
	\begin{align*}
		\|w_{q+1}^{(c)}\|_{C^0_{t,x}}
		& \lesssim \frac{1}{\lambda_{q+1}}\sum_{k\in \Lambda} \left\|  \nabla a_k \times \frac{B_k}{|k|}e^{i\lambda_{q+1}k\cdot x}\right\| _{C^0_{t,x}}
		\\&\lesssim
		\frac{\ell^{-\frac{5}{2}}}{\lambda_{q+1}} \left( 1+\|v_q\|_{C^0_{[t-1,t+1],x}}+\|z_q\|_{C^0_{[t-1,t+1],x}}\right) \left(1+ \|\mathring{R}_q\|_{C^0_{[t-1,t+1],x}}\right)^2.
	\end{align*}
	Taking expectation and 
	using \eqref{estimate-zq'}, \eqref{vqa} and \eqref{vqc}, we obtain for $n\geq1$
	\begin{equation}\label{wc0}
		\begin{aligned}
			\$w_{q+1}^{(c)}\$_{C^0,n}&\lesssim \frac{\ell^{-\frac{5}{2}}}{\lambda_{q+1}}\left( 1+\$v_q\$_{C^0,2n}+\$z_q\$_{C^0,2n}\right)\left( 1+\$\mathring{R}_q\$^2_{C^0,4n}\right) 
			\\&\lesssim \frac{\lambda_q^{10}}{\lambda_{q+1}}\left( 1+\lambda_q^{n\beta}+\sqrt{2n}L\right)\left( 1+\lambda_q^{8n\beta}\right)
			\lesssim \frac{\lambda_q^{10+9n\beta}}{\lambda_{q+1}}\sqrt{2n}L
			\leq \frac13 \lambda_{q+1}^{\frac{n}{3}\beta},
		\end{aligned}
	\end{equation}
	where we used  $b>27$ and $\sqrt{2n}L<a^{(\frac{b}{3}-9)n\beta}$ to have $\sqrt{2n}L\lambda_{q}^{10+9n\beta}<\lambda_{q}^{b+\frac{n}{3}b\beta}$ and $a$ was chosen sufficiently large to absorb the constant. Taking $n=3$ in the second inequality of \eqref{wc0} and since $2r<3$ we also obtain 
	\begin{align}\label{wcr}
		\$w_{q+1}^{(c)}\$_{C^0,2r}\leq \$w_{q+1}^{(c)}\$_{C^0,3}\lesssim \frac{\lambda_q^{10+27\beta}}{\lambda_{q+1}} \leq \frac{1}{4}\delta_{q+1}^{\frac{1}{2}},
	\end{align}
	which requires $b>27$ and $10+b\beta+27\beta<b$ to have $\lambda_q^{10+27\beta}<\lambda_q^{b-b\beta}$ and $a$ sufficiently large to absorb the constant.
	
	For $C^1_{t,x}$-norm of $w_{q+1}$, we apply chain rule with \eqref{a_k1}, \eqref{a_k0} and \eqref{tauak0} to deduce
	\begin{equation}\label{wp 1}
		\aligned
		\|w_{q+1}^{(p)}\|_{C_{t,x}^1} \lesssim &\sum_{k\in \Lambda}
		\left(\|\nabla a_k\|_{C^0_{t,x}}+\|\partial_s a_k\|_{C^0_{t,x}}+\lambda_{q+1}(\|a_k\|_{C^0_{t,x}}+\|\partial_\tau a_k\|_{C^0_{t,x}}) \right)
		\\  \lesssim &\ell^{-\frac{5}{2}}\left( 1+\|v_q\|_{C^0_{[t-1,t+1],x}}+\|z_q\|_{C^0_{[t-1,t+1],x}} \right) \left(1+ \|\mathring{R}_q\|_{C^0_{[t-1,t+1],x}}\right)^2
		\\ & +\lambda_{q+1}\left( 1+\|v_q\|_{C^0_{[t-1,t+1],x}}+\|z_q\|_{C^0_{[t-1,t+1],x}} \right)\left(1+\|\mathring{R}_q\|_{C^0_{[t-1,t+1],x}}^{\frac{1}{2}}\right).
		\endaligned
	\end{equation}
	Taking expectation and using \eqref{estimate-zq'}, \eqref{vqa} and \eqref{vqc}, we obtain for $n\geq1$
	\begin{equation}\label{wp1}
		\aligned
		\$w_{q+1}^{(p)}\$_{C_{t,x}^1,n}\lesssim &\ell^{-\frac{5}{2}}\left( 1+\$v_q\$_{C^0,2n}+\$z_q\$_{C^0,2n}\right)(1+\$\mathring{R}_q\$^2_{C^0,4n}) \\ &+\lambda_{q+1}\left( 1+\$v_q\$_{C^0,2n}+\$z_q\$_{C^0,2n}\right)\left( 1+\$\mathring{R}_q\$_{C^0,n}^{\frac{1}{2}}\right) 
		\\\lesssim&\lambda_q^{10}\left( 1+\lambda_q^{n\beta}+\sqrt{2n}L\right)\left( 1+\lambda_q^{8n\beta}\right)  +\lambda_{q+1}\lambda_q^{\frac{n}{2}\beta}\left( 1+\lambda_q^{n\beta}+\sqrt{2n}L\right)
		\\\lesssim& \lambda_{q}^{10+9n\beta}\sqrt{2n}L +\lambda_{q+1}\lambda_q^{\frac{3}{2}n\beta}\sqrt{2n}L
		\leq \frac13\lambda_{q+1}^{1+n\beta},
		\endaligned
	\end{equation}
	where we used $b>10$, $a^{(b-9)n\beta}>\sqrt{2n}L$ to have $\sqrt{2n}L\lambda_{q}^{10+9n\beta}<\lambda_{q}^{b+bn\beta}$ and $\sqrt{2n}L\lambda_{q}^{\frac32n\beta}<\lambda_{q+1}^{n\beta}$  in the last inequality. We also chose $a$ sufficiently large to absorb the constant. Applying chain rule with \eqref{a_k1}, \eqref{a_k2} and \eqref{a_k0} again we deduce
	\begin{equation}\label{wc 1}
		\aligned
		\|w_{q+1}^{(c)}\|_{C_{t,x}^1} \lesssim &\frac{1}{\lambda_{q+1}} \sum_{k\in\Lambda} \left(  \|\nabla^2 a_k\|_{C^0_{t,x}}+\|\partial_s \nabla a_k\|_{C^0_{t,x}}
		+ \lambda_{q+1}(\|\partial_\tau \nabla a_k\|_{C^0_{t,x}}+  \| \nabla a_k\|_{C^0_{t,x}})\right) 
		\\  \lesssim &\frac{\ell^{-\frac{9}{2}}}{\lambda_{q+1}} \left( 1+ \|v_q\|_{C^0_{[t-1,t+1],x}}+\|z_q\|_{C^0_{[t-1,t+1],x}} \right)^2\left(1+\|\mathring{R}_q\|_{C^0_{[t-1,t+1],x}}\right)^3
		\\&+\ell^{-\frac{5}{2}}\left( 1+\|v_q\|_{C^0_{[t-1,t+1],x}}+\|z_q\|_{C^0_{[t-1,t+1],x}} \right)^2 \left(1+ \|\mathring{R}_q\|_{C^0_{[t-1,t+1],x}}\right)^2
		\\\lesssim& \ell^{-\frac52}  \left( 1+ \|v_q\|_{C^0_{[t-1,t+1],x}}+\|z_q\|_{C^0_{[t-1,t+1],x}} \right)^2\left(1+\|\mathring{R}_q\|_{C^0_{[t-1,t+1],x}}\right)^3.
		\endaligned
	\end{equation}
	Taking expectation and using \eqref{estimate-zq'}, \eqref{vqa} and \eqref{vqc}, we obtain for $n\geq1$
	\begin{equation}\label{wc1}
		\aligned
		\$w_{q+1}^{(c)}\$_{C_{t,x}^1,n}\lesssim & \lambda_q^{10}\left( 1+\$v_q\$_{C^0,4n}^2+\$z_q\$_{C^0,4n}^2\right)\left(1+\$\mathring{R}_q\$^3_{C^0,6n} \right) 
		\\ \leq& \lambda_q^{10}\left( 1+\lambda_q^{3n\beta}+ 4nL^2\right) \left(1+\lambda_q^{18n\beta} \right) 
		\\\lesssim &\lambda_q^{10+21n\beta}+ \lambda_q^{10+18n\beta}4nL^2
		\leq \frac13 \lambda_{q+1}^{1+n\beta},
		\endaligned
	\end{equation}	
	which requires $b>21$ and $a^{(b-21)n\beta}>4nL^2$ to have $\lambda_{q}^{10+21n\beta}<\lambda_q^{b+bn\beta}$ and $4nL^2\lambda_{q}^{10+18n\beta}<\lambda_q^{b+bn\beta}$. We also chose $a$ sufficiently large to absorb the constant.
	
	With these estimates in hand, we combine \eqref{vq-vl}, \eqref{wpr} with \eqref{wcr}  to deduce 
	\begin{align*}
		\$v_{q+1}-v_q\$_{C^0,2r}\leq \$v_q-v_\ell\$_{C^0,2r}+\$w_{q+1}^{(p)}\$_{C^0,2r}+\$w_{q+1}^{(c)}\$_{C^0,2r} \leq \bar{M}\delta_{q+1}^{\frac{1}{2}},
	\end{align*}
	which implies \eqref{vq+1-vq2} holds.
	Using \eqref{vqa}, \eqref{wp0} and \eqref{wc0}, we obtain
	\begin{align}\label{vq +1}
		\$v_{q+1}\$_{C^0,n} \leq \$v_q\$_{C^0,n}+ \$w_{q+1}^{(p)}\$_{C^0,n}+\$w_{q+1}^{(c)}\$_{C^0,n}\leq\lambda_q^{\frac{n}{3}\beta} + \frac{2}{3} \lambda_{q+1}^{\frac{n}{3}\beta} \leq \lambda_{q+1}^{\frac{n}{3}\beta},
	\end{align}
	which implies \eqref{vqa} holds at the level $q+1$.
	Using \eqref{vqb}, \eqref{wp1} and \eqref{wc1} we obtain
	\begin{align*}
		\$v_{q+1}\$_{C^1_{t,x},n} \leq \$v_q\$_{C^1_{t,x},n}+ \$w_{q+1}^{(p)}\$_{C^1_{t,x},n}+\$w_{q+1}^{(c)}\$_{C^1_{t,x},n}\leq \lambda_q^{1+n\beta}+\frac23 \lambda_{q+1}^{1+n\beta} \leq \lambda_{q+1}^{1+n\beta},
	\end{align*}
	which implies \eqref{vqb} holds at the level $q+1$.

	\subsection{Definition of the Reynolds stress $\mathring{R}_{q+1}$}\label{316}
	By subtracting system \eqref{mollification} from equation \eqref{euler} at level $q+1$, we derive
	\begin{equation}\label{def R q+1}
		\aligned
		& \div \mathring{R}_{q+1}-\nabla p_{q+1}
		\\& =\underbrace{(\partial_t+(v_\ell+z_\ell)\cdot \nabla)w_{q+1}^{(p)}}
		_{\div(R^{\textrm{trans}})}+ \underbrace{ \div (w_{q+1}^{(p)} \otimes w_{q+1}^{(p)}+\mathring{R}_\ell)}
		_{\div(R^{\textrm{osc}})+\nabla p^{\textrm{osc}}}+	\underbrace{\partial_tw_{q+1}^{(c)}}_{\div(R^{\textrm{error}}_1)}+\underbrace{\div (w_{q+1}^{(p)}\otimes(v_\ell+z_\ell))}_{\div (R^{\textrm{error}}_2)}
		\\&
		+\underbrace{\div(v_{q+1}+z_\ell)\otimes w_{q+1}^{(c)}+w_{q+1}^{(c)}\otimes (v_{q+1}+z_\ell)-w_{q+1}^{(c)} \otimes w_{q+1}^{(c)})}_{\div(R^{\textrm{error}}_3)+\nabla p^{\textrm{error}}_3}
		\\&+\underbrace{\div(v_{q+1} \otimes (z_{q+1}- z_{\ell} )+(z_{q+1} -z_{\ell} )\otimes v_{q+1}+z_{q+1}\otimes z_{q+1}-z_{\ell} \otimes z_{\ell})-z_{q+1}+z_{\ell}}_{\div(R^{\textrm{cor}})+\nabla p^{\textrm{cor}}}
		\\ &+\div(R^{\textrm{com}}) -\nabla p_{\ell}.
		\endaligned
	\end{equation}
	Here, $R^{\text{com}}$ is defined as in \eqref{mollification}. By using the inverse divergence operator $\mathcal{R}$ introduced in Section~\ref{s:2.1}, we define 
	\begin{equation*}
		\aligned
		R^{\textrm{trans}}&:=\mathcal{R}\left((\partial_t+(v_\ell+z_\ell)\cdot \nabla)w_{q+1}^{(p)}\right),
		\\ R^{\textrm{osc}}&:=\mathcal{R}\left[ \div\left( w_{q+1}^{(p)}\otimes w_{q+1}^{(p)}+\mathring{R}_\ell-\frac{1}{2}( |w_{q+1}^{(p)}|^2-r_0^{-1}\rho)\Id  \right) \right] ,
		\\ R^{\textrm{error}}_1&:=\mathcal{R}\left( \partial_tw_{q+1}^{(c)} \right),
		\\ R^{\textrm{error}}_2&:= \mathcal{R}\left[ \div (w_{q+1}^{(p)}\otimes(v_\ell+z_\ell))\right],
		\\ R^{\textrm{error}}_3&:=(v_{q+1}+z_\ell)\mathring{\otimes} w_{q+1}^{(c)}+w_{q+1}^{(c)}\mathring{\otimes} (v_{q+1}+z_\ell)-w_{q+1}^{(c)} \mathring{\otimes} w_{q+1}^{(c)},
		\\ R^{\textrm{cor}}&:=v_{q+1}\mathring \otimes (z_{q+1}- z_{\ell}) +(z_{q+1}-z_{\ell})\mathring \otimes v_{q+1}+z_{q+1}\mathring \otimes z_{q+1}-z_{\ell}\mathring \otimes z_{\ell}-\mathcal{R}(z_{q+1}-z_{\ell}),
		\\  p^{\textrm{osc}}&:=\frac{1}{2}(|w_{q+1}^{(p)}|^2-r_0^{-1}\rho),
		\\  p^{\textrm{error}}_3&:=\frac{1}{3}(2(v_{q+1}+z_\ell)\cdot w_{q+1}^{(c)} -|w_{q+1}^{(c)}|^2),
		\\ p^{\textrm{cor}}&:=\frac{1}{3}(v_{q+1} \cdot z_{q+1}-v_{q+1} \cdot z_{\ell} +z_{q+1} \cdot v_{q+1}-z_{\ell} \cdot v_{q+1}+z_{q+1} \cdot z_{q+1}-z_{\ell}\cdot z_{\ell}).
		\endaligned
	\end{equation*}
	Finally, we define the Reynolds stress at the level $q+1$ by
	\begin{equation}\label{Reynold}
		\mathring{R}_{q+1}=R^{\textrm{trans}}+R^{\textrm{osc}}+ R^{\textrm{error}}_1+ R^{\textrm{error}}_2+ R^{\textrm{error}}_3+R^{\textrm{com}}+R^{\textrm{cor}},
	\end{equation}
	and the corresponding pressure is defined by
	\begin{equation*}
		p_{q+1}=p_{\ell}-p^{\textrm{osc}}-p^{\textrm{error}}_3-
		p^{\textrm{cor}}.
	\end{equation*}
	
	Arguing as in \cite[Proposition 6.1]{DelSze13}, we also deduce the following result, which provides bounds of the Fourier coefficients (with respect to $\xi \in \mathbb{T}^3$) of the matrix-valued field $W\otimes W$. The proof is also given in Appendix~\ref{ap:A}.
	\begin{proposition} \label{cor:U}
		Let $  W=W(s,y,\tau,\xi )$ be defined by \eqref{defwp}. Then the matrix-valued field $W \otimes W$ can be written as
		
		\begin{align}\label{WxW1}
			(W\otimes W)(s,y,\tau,\xi)=r_0^{-1}\rho(s,y) \mathrm{Id} -\mathring{R}_\ell(s,y) +\sum_{1\leq|k|\leq 2\lambda_0} U_k(s,y,\tau)e^{ik\cdot \xi},
		\end{align}
		with the coefficients $U_k \in C^\infty( \R \times \mathbb{T}^3 \times\R, \mathcal{S}^{3 \times 3})$, satisfying\footnote{Here, $\mathrm{Tr}: \mathcal{S}^{3 \times 3}\to \R$  denotes the trace of a matrix.}
		\begin{align}\label{eq:proU}
			U_k k = \tfrac{1}{2} \tr(U_k)k.
		\end{align}
		In addition, for $N\geq2$ we have
		\begin{align}\label{esti:UN}
			\| U_k\|_{C^0_sC^N_y}\lesssim \ell^{-2N-1}(1+\|v_q\|_{C^0_{[s-1,s+1],y}}+\|z_q\|_{C^0_{[s-1,s+1],y}})^N(1+\|\mathring{R}_q\|_{C^0_{[s-1,s+1],y}})^{N+2},
		\end{align}
		and 
		\begin{align}\label{esti:U1}
			\|U_k\|_{C^0_sC^1_y}\lesssim \ell^{-\frac52}(1+\|v_q\|_{C^0_{[s-1,s+1],y}}+\|z_q\|_{C^0_{[s-1,s+1],y}})(1+\|\mathring{R}_q\|_{C^0_{[s-1,s+1],y}}^{\frac12})^5,
		\end{align}
		where the implicit constants are independent of $q$.
	\end{proposition}
	
	\subsection{Inductive estimates for $\mathring{R}_{q+1}$}\label{31}
	In this section, we shall verify the  estimates \eqref{vqc} and \eqref{vqd} at the $q+1$ level. To this end, we first recall the following stationary phase lemma (see for example \cite[Proposition 5.2]{DelSze13}) adapted to our setting. 
	\begin{lemma}\label{lemma1}
		Let $ k \in \mathbb{Z}^3 / \{0\}$ be fixed. Assume $a\in C^{\infty}(\mathbb{T}^3;\mathbb{R}^3)$, let $F(x):=a(x)e^{i\lambda k\cdot x}$. 
		Then, with the inverse divergence operator $\mathcal{R}$ defined in Section~\ref{s:2.1}, for any $m\in \mathbb{N}$, we have 
		\begin{align*}
			\left| \int_{\mathbb{T}^3}a(x)e^{i\lambda k\cdot x} \dif x\right|  \lesssim
			\frac{\|a\|_{C^m}}{\lambda^m},
		\end{align*}
		and
		\begin{align*}
			\left\|\mathcal{R}(F)\right\|_{C^\alpha}=\left\|\mathcal{R}\left(a(x)e^{i\lambda k\cdot x}\right)\right\|_{C_x^\alpha}\lesssim \frac{\|a\|_{C_x^0}}{\lambda^{1-\alpha}}+\frac{\|a\|_{C_x^m}}{\lambda^{m-\alpha}}+\frac{\|a\|_{C_x^{m,\alpha}} }{\lambda^m},
		\end{align*}
		where the implicit constant depends on $\alpha$ and m (in particular, not on the frequency $\lambda$).
	\end{lemma}
	
	\subsubsection{Estimates on the transport error} Inspecting the definition of $w_{q+1}^{(p)}$, we have
	\begin{equation}\label{Rtra}
		\aligned
		\mathcal{R}\left( (\partial_t+(v_\ell+z_\ell)\cdot \nabla)w_{q+1}^{(p)}\right) =&\lambda_{q+1} \mathcal{R}\left( \sum_{k\in \Lambda} (\partial_\tau a_k+ik\cdot(v_\ell+z_\ell)a_k)B_ke^{i\lambda_{q+1}k\cdot x}\right) 
		\\&\quad +\mathcal{R}\left( \sum_{k\in \Lambda} (\partial_sa_k+(v_\ell+z_\ell)\cdot\nabla a_k)B_ke^{i\lambda_{q+1}k\cdot x}\right) 
		\\=&: R^{\textrm{trans}}_1+R^{\textrm{trans}}_2 .
		\endaligned
	\end{equation}
	We estimate the above two terms separately. For the first term on the right-hand side, we use \eqref{derivat a_k}, \eqref{deriva ak0} and apply Lemma~\ref{lemma1} with $m=2$ to have
	\begin{equation}\label{Rtran 1}
		\aligned
		\|\lambda_{q+1}& \mathcal{R}((\partial_\tau a_k+ik\cdot(v_\ell+z_\ell)a_k)B_ke^{i\lambda_{q+1}k\cdot x}) \|_{C^0_{t,x}}
		\\& \lesssim
		\frac{\lambda_{q+1}\|\partial_\tau a_k+ik\cdot(v_\ell+z_\ell)a_k\|_{C^0_{t,x}}}{\lambda_{q+1}^{1-\alpha}}+\frac{\lambda_{q+1}\|\partial_\tau a_k+ik\cdot(v_\ell+z_\ell)a_k\|_{C^0_tC_x^3}}{\lambda_{q+1}^{2-\alpha}}
		\\& \lesssim \frac{1}{\lambda_{q}\lambda_{q+1}^{-\alpha}}\big(\|\mathring{R}_q\|^{\frac{1}{2}}_{C^0_{[t-1,t+1],x}}+\ell^{\frac{1}{2}}+\delta_{q+1}^{\frac12} \big) 
		\\&\quad +\frac{\ell^{-\frac{13}{2}}}
		{\lambda_{q+1}^{1-\alpha}}(1+\|v_q\|_{C_{[t-1,t+1],x}^0}+\|z_q\|_{C_{[t-1,t+1],x}^0})^3(1+\|\mathring{R}_q\|_{C^0_{[t-1,t+1],x}})^4.
		\endaligned
	\end{equation}
	Then taking expectation and using \eqref{estimate-zq'}, \eqref{vqa} and \eqref{vqc}, we obtain
	\begin{equation}\label{tr1}
		\aligned
		\$R^{\textrm{trans}}_1\$_{C^0,n}
		&\lesssim \frac{\$\mathring{R}_q\$^{\frac{1}{2}}_{C^0,n}+\delta_{q+1}^{\frac12}}{\lambda_q\lambda_{q+1}^{-\alpha}}+\frac{\ell^{-\frac{13}{2}}}{\lambda_{q+1}^{1-\alpha}}\left(1+\$v_q\$_{C^0,6n}^3+\$z_q\$_{C^0,6n}^3\right)\left(1+\$\mathring{R}_q\$^4_{C^0,8n} \right) 
		\\&	\lesssim  \frac{ \lambda_q^{\frac{n}{2}\beta}+\delta_{q+1}^{\frac12} }{\lambda_q \lambda_{q+1}^{-\alpha}}+ \frac{\lambda_q^{26}}{\lambda_{q+1}^{1-\alpha}}\left( 1+\lambda_q^{6n\beta}+(6n)^2L^3\right)\left( 1+\lambda_q^{32n\beta}\right)  
		\\&\lesssim  \frac{ \lambda_q^{\frac{n}{2}\beta}+\delta_{q+1}^{\frac12} }{\lambda_q \lambda_{q+1}^{-\alpha}} + \frac{\lambda_q^{26+38n\beta}}{\lambda_{q+1}^{1-\alpha}} + \frac{\lambda_q^{26+32n\beta}}{\lambda_{q+1}^{1-\alpha}}(6n)^2L^3
		\leq\begin{cases}
			\frac{1}{10}\delta_{q+2},&n=1,2,\\
			\frac{1}{10} \lambda_{q+1}^{n\beta},&n\geq3,
		\end{cases}
		\endaligned
	\end{equation}
	where we used $b>38$, $3b^2\beta+b\alpha<1$ to have $\lambda_q^{\beta-1+b\alpha}<\lambda_{q}^{-2b^2\beta}$, $\lambda_q^{26+76\beta+b\alpha}<\lambda_{q}^{b-2b^2\beta}$ and used  $a^{(b-38)\beta n}>(6n)^2L^3$ to have $(6n)^2L^3\lambda_{q}^{26+38n\beta+b\alpha}<\lambda_{q}^{b+bn\beta}$ and we chose $a$ sufficiently large to absorb the constant.
	
	Similarly, for the second term on the right-hand side of \eqref{Rtra}, using \eqref{a_k1}, \eqref{a_k0} and applying Lemma~\ref{lemma1} with $m=1$, we have
	\begin{equation}\label{Rtran 2}
		\aligned
		\|\mathcal{R}( (\partial_s&a_k+(v_\ell+z_\ell)\cdot\nabla a_k)B_ke^{i\lambda_{q+1}k\cdot x}) \|_{C^0_{t,x}}
		\\  & \lesssim \frac{\|( \partial_sa_k+(v_\ell+z_\ell)\cdot \nabla a_k)\| _{C^0_{t,x}}}{\lambda_{q+1}^{1-\alpha}}+\frac{\|( \partial_sa_k+(v_\ell+z_\ell)\cdot\nabla a_k)\| _{C^0_tC^2_x}}{\lambda_{q+1}^{1-\alpha}}
		\\ & \lesssim
		\frac{\ell^{-\frac{5}{2}}}{\lambda_{q+1}^{1-\alpha}}(1+\|v_q\|_{C^0_{[t-1,t+1],x}}+\|z_q\|_{C^0_{[t-1,t+1],x}})^2(1+\|\mathring{R}_q\|_{C^0_{[t-1,t+1],x}})^2
		\\  & \quad + \frac{\ell^{-\frac{13}{2}}}{\lambda_{q+1}^{1-\alpha}}(1+\|v_q\|_{C^0_{[t-1,t+1],x}}+\|z_q\|_{C^0_{[t-1,t+1],x}})^4(1+\|\mathring{R}_q\|_{C^0_{[t-1,t+1],x}})^4.
		\endaligned
	\end{equation}
	Taking expectation and using \eqref{estimate-zq'}, \eqref{vqa} and \eqref{vqc}, we obtain
	\begin{equation}\label{tran2}
		\aligned
		\$R^{\textrm{trans}}_2 \$_{C^0,n}
		& \lesssim\frac{\ell^{-\frac{13}{2}}}{\lambda_{q+1}^{1-\alpha}}\left( 1+\$v_q\$^4_{C^0,8n}+\$z_q\$^4_{C^0,8n}\right) \left( 1+\$\mathring{R}_q\$^4_{C^0,8n}\right) 
		\\&\leq \frac{\lambda_q^{26}}{\lambda_{q+1}^{1-\alpha}}\left( 1+\lambda_q^{11n\beta} +L^4(8n)^2\right) \left( 1+\lambda_q^{32n\beta}\right) 
		\\&\lesssim \frac{\lambda_q^{26+43n\beta}}{\lambda_{q+1}^{1-\alpha}}+\frac{\lambda_q^{26+32n\beta}}{\lambda_{q+1}^{1-\alpha}}(8n)^2L^4
		\leq\begin{cases}
			\frac{1}{10}\delta_{q+2},&n=1,2,\\
			\frac{1}{10}\lambda_{q+1}^{n\beta},&n\geq3,
		\end{cases}
		\endaligned
	\end{equation}	
	where we used $b>43$, $3b^2\beta+b\alpha<1$ to have  $\lambda_q^{26+86\beta+b\alpha}<\lambda_{q}^{b-2b^2\beta}$ and used  $a^{(b-43)\beta n}>(8n)^2L^4$ to have $(8n)^2L^4\lambda_{q}^{26+43n\beta+b\alpha}<\lambda_{q}^{b+bn\beta}$ and we chose $a$ sufficiently large to absorb the constant. Taking $n=1$ in the second inequality of \eqref{tr1} and \eqref{tran2}, we obtain
	\begin{align}\label{RtraL}
		\$R^{\textrm{trans}}\$_{L^1,1}\lesssim \frac{ \lambda_q^{\beta}} {\lambda_q \lambda_{q+1}^{-\alpha}} +  \frac{\lambda_q^{26+43\beta}}{\lambda_{q+1}^{1-\alpha}} + \frac{\lambda_q^{26+32\beta}}{\lambda_{q+1}^{1-\alpha}}64L^4\leq \frac{1}{7 {\cdot }48} r_R\delta_{q+3}\underline{e},
	\end{align} 
	which requires $a$ sufficiently large to absorb the constant and $b>2$, $3b^3\beta+b\alpha<1$  to have $\lambda_{q}^{\beta-1+b\alpha}<\lambda_{q}^{-2b^3\beta}$ and $\lambda_{q}^{26+43\beta+b\alpha}<\lambda_{q}^{b-2b^3\beta}$ in the last inequality.
	
	\subsubsection{Estimates on the oscillation error}
	Recalling the definition of $w_{q+1}^{(p)}$ in \eqref{defwp}, we write
	\begin{align*}
		w_{q+1}^{(p)}(t,x)=W(t,x,\lambda_{q+1}t,\lambda_{q+1}x).
	\end{align*} 
	By \eqref{WxW1} we have
	\begin{align}\label{wpxwp}
		w_{q+1}^{(p)}\otimes w_{q+1}^{(p)}(t,x)=r_0^{-1}\rho(t,x) \mathrm{Id} -\mathring{R}_\ell (t,x) +\sum_{1\leq|k|\leq 2\lambda_0} U_k(t,x,\lambda_{q+1}t)e^{i\lambda_{q+1}k\cdot x}.
	\end{align}
	We then take trace on the both sides of \eqref{wpxwp} and use \eqref{eq:proU} to rewrite $\div(R^{\textrm{osc}})$ as
	\begin{align*}
		&\div \left( w_{q+1}^{(p)}\otimes w_{q+1}^{(p)}+\mathring{R}_\ell-\frac{1}{2}( |w_{q+1}^{(p)}|^2-r_0^{-1}\rho)\Id \right) 
		\\&=\div\left( w_{q+1}^{(p)}\otimes w_{q+1}^{(p)}-(r_0^{-1} \rho\Id-\mathring{R}_\ell)-\frac{1}{2}( |w_{q+1}^{(p)}|^2-3r_0^{-1}\rho)\Id  \right) 
		\\&=\div \left[ 
		\sum_{1\leq|k|\leq 2\lambda_0}   \left( U_k-\frac{1}{2}(\tr U_k)\Id\right)e^{i\lambda_{q+1}k\cdot x}\right] 
		\stackrel{\eqref{eq:proU}}{=}\sum_{1\leq|k|\leq 2\lambda_0}  \div\left[U_k-\frac{1}{2}(\tr U_k)\Id \right] e^{i\lambda_{q+1}k\cdot x}.
	\end{align*}
	Using \eqref{esti:UN} and \eqref{esti:U1} we have
	\begin{align*}\label{Uk13}
		\begin{aligned}
			\left\|\div\left( U_k-\frac{1}{2}(\tr U_k)\Id \right)\right\|_{C^0_{t,x}}\lesssim
			\ell^{-\frac{5}{2}}(1+\|v_q\|_{C_{[t-1,t+1],x}^0}+\|z_q\|_{C_{[t-1,t+1],x}^0}) (1+\|\mathring{R}_q\|^\frac{1}{2}_{C^0_{[t-1,t+1],x}})^5,
			\\ \left\|\div\left( U_k-\frac{1}{2}(\tr U_k)\Id \right)\right\|_{C^0_tC^2_x}\lesssim \ell^{-7}(1+\|v_q\|_{C_{[t-1,t+1],x}^0}+\|z_q\|_{C_{[t-1,t+1],x}^0})^3 (1+\|\mathring{R}_q\|_{C^0_{[t-1,t+1],x}})^5.
		\end{aligned}
	\end{align*}
	Applying Lemma~\ref{lemma1} with $m=1$ implies
	\begin{equation}\label{Rosc}
		\aligned
		\|R^{\textrm{osc}}\|_{C^0_{t,x}}&\lesssim \frac{1}{\lambda_{q+1}^{1-\alpha}}\left( \left\|\div\left( U_k-\frac{1}{2}(\tr U_k)\Id \right)\right\|_{C^0_{t,x}}+\left\|\div\left( U_k-\frac{1}{2}(\tr U_k)\Id \right)\right\|_{C^0_tC^2_x}\right) 
		\\&\lesssim \frac{	\ell^{-\frac{5}{2}}}{\lambda_{q+1}^{1-\alpha}}(1+\|v_q\|_{C_{[t-1,t+1],x}^0}+\|z_q\|_{C_{[t-1,t+1],x}^0}) (1+\|\mathring{R}_q\|^\frac{1}{2}_{C^0_{[t-1,t+1],x}})^5
		\\&\quad +\frac{\ell^{-7}}{\lambda_{q+1}^{1-\alpha}}(1+\|v_q\|_{C_{[t-1,t+1],x}^0}+\|z_q\|_{C_{[t-1,t+1],x}^0})^3 (1+\|\mathring{R}_q\|_{C^0_{[t-1,t+1],x}})^5.
		\endaligned
	\end{equation}
	Taking expectation and using \eqref{estimate-zq'}, \eqref{vqa} and \eqref{vqc} we obtain
	\begin{equation}\label{osc}
		\aligned
		\$R^{\textrm{osc}}\$_{C^0,n}   &  \lesssim \frac{ \ell^{-\frac{5}{2}}}{\lambda_{q+1}^{1-\alpha}} \left(1+\$v_q\$_{C^0,2n}+\$z_q\$_{C^0,2n}\right) \left(1+\$ \mathring{R}_q\$^\frac{5}{2}_{C^0,5n} \right) 
		\\   &  +    \frac{\ell^{-7}}{\lambda_{q+1}^{1-\alpha}}\left(1+ \$v_q\$^3_{C^0,6n}+\$z_q\$^3_{C^0,6n}\right)\left( 
		1+\$ \mathring{R}_q\$^5_{C^0,10n} \right) 
		\\   &  \leq  \frac{\lambda_q^{28}}{\lambda_{q+1}^{1-\alpha}}\left( 1+\lambda_q^{6n\beta}+L^3(6n)^2\right) \left(1+\lambda_q^{50n\beta} \right) 
		\\  &  \lesssim \frac{\lambda_q^{28+56n\beta}}{\lambda_{q+1}^{1-\alpha}}+\frac{\lambda_q^{28+50n\beta}}{\lambda_{q+1}^{1-\alpha}}L^3(6n)^2
		\leq  \begin{cases}
			\frac{1}{10}\delta_{q+2},&n=1,2,\\
			\frac{1}{10}\lambda_{q+1}^{n\beta},&n\geq3,
		\end{cases}
		\endaligned
	\end{equation}
	where we have used $b>56$, $3b^2\beta +b\alpha<1$ to have $\lambda_{q}^{28+112\beta+b\alpha}<\lambda_{q}^{b-2b^2\beta}$ and used  $a^{(b-56)\beta n}>(6n)^2L^3$ to have $L^3(6n)^{\frac32}\lambda_{q}^{28+56n\beta+b\alpha}<\lambda_{q}^{b+bn\beta}$. Then taking $n=1$ in the second inequality of \eqref{osc}, we can also obtain 
	\begin{align}\label{RoscL}
		\$R^{\textrm{osc}}\$_{L^1,1}\lesssim \frac{\lambda_q^{28+56\beta}}{\lambda_{q+1}^{1-\alpha}}+\frac{\lambda_q^{28+50\beta}}{\lambda_{q+1}^{1-\alpha}}36L^3\leq \frac{1}{7{\cdot}48}r_R\delta_{q+3}\underline{e},
	\end{align}
	where we used $b>29$, $3b^3\beta+b\alpha<1$ to have $\lambda_{q}^{28+56\beta+b\alpha}<\lambda_{q}^{b-2b^3\beta}$ and we chose $a$ sufficiently large to absorb the constant.
	
	\subsubsection{Estimates on $R^{\textrm{error}}_1$}
	We start with the observation that 
	\begin{align*}
		\partial_t w_{q+1}^{(c)}
		= \sum_{k\in \Lambda} \left( \frac{\partial_s\nabla a_k}{\lambda_{q+1}}+\partial_\tau \nabla a_k\right) \times \frac{B_k}{|k|}e^{i\lambda_{q+1}k\cdot x}.
	\end{align*}
	Applying Lemma~\ref{lemma1} with $m=1$ and using \eqref{a_k1}, \eqref{a_k2} and \eqref{a_k0} we have 
	\begin{equation}\label{error 1}
		\aligned
		\|\mathcal{R}(\partial_tw_{q+1}^{(c)})\|_{C^0_{t,x}}&\lesssim \frac{1}{\lambda_{q+1}^{1-\alpha}}\left( \left\|\frac{\partial_s\nabla a_k}{\lambda_{q+1}}+\partial_\tau \nabla a_k\right\|_{C^0_{t,x}}+ \left\|\frac{\partial_s\nabla a_k}{\lambda_{q+1}}+\partial_\tau \nabla a_k\right\|_{C^0_tC^2_x}\right) 
		\\ &\lesssim \frac{\ell^{-\frac{17}{2}}}{\lambda_{q+1}^{2-\alpha}}(1+\|v_q\|_{C_{[t-1,t+1],x}^0}+\|z_q\|_{C_{[t-1,t+1],x}^0})^4(1+\|\mathring{R}_q\|_{C^0_{[t-1,t+1],x}})^5
		\\&+\frac{\ell^{-\frac{13}{2}}}{\lambda_{q+1}^{1-\alpha}}(1+\|v_q\|_{C_{[t-1,t+1],x}^0}+\|z_q\|_{C_{[t-1,t+1],x}^0})^4(1+\|\mathring{R}_q\|_{C^0_{[t-1,t+1],x}})^4.
		\endaligned
	\end{equation}
	Taking expectation and using \eqref{estimate-zq'}, \eqref{vqa} and \eqref{vqc} we obtain
	\begin{equation}\label{error1}
		\aligned
		\$\mathcal{R}(\partial_t w_{q+1}^{(c)})\$_{C^0,n}  & \lesssim \frac{ \ell^{-\frac{17}{2}}}{\lambda_{q+1}^{2-\alpha}}\left( 1+\$v_q\$^4_{C^0,8n}+\$z_q\$^4_{C^0,8n} \right) \left( 1+\$\mathring{R}_q\$^5_{C^0,10n}\right) 
		\\ &\quad +\frac{\ell^{-\frac{13}{2}}}{\lambda_{q+1}^{1-\alpha}}(1+\$v_q\$^4_{C^0,8n}+\$z_q\$^4_{C^0,8n})\left(1+\$\mathring{R}_q\$^4_{C^0,8n}\right)
		\\  & \leq\frac{\lambda_q^{34}}{\lambda_{q+1}^{1-\alpha}}\left( 1+\lambda_q^{11n
			\beta}+L^4(8n)^2\right) \left( 1+\lambda_q^{50n\beta} \right) 
		\\ & \lesssim \frac{\lambda_q^{34+61n\beta}}{\lambda_{q+1}^{1-\alpha}}+\frac{\lambda_q^{34+50n\beta}}{\lambda_{q+1}^{1-\alpha}}L^4(8n)^2
		\leq  \begin{cases}
			\frac{1}{10}\delta_{q+2},&n=1,2,\\
			\frac{1}{10}\lambda_{q+1}^{n\beta},&n\geq3,
		\end{cases}
		\endaligned
	\end{equation}
	where we used $b>61$, $3b^2\beta+b\alpha <1$ to have $\lambda_{q}^{34+122\beta+b\alpha}<\lambda_{q}^{b-2b^2\beta}$, and used  $a^{(b-61)\beta n}>(8n)^2L^4$ to have $(8n)^2L^4\lambda_{q}^{34+61n\beta+b\alpha}<\lambda_{q}^{b+bn\beta}$, and $a$ was chosen sufficiently large to absorb the constant. Then taking $n=1$ in the third inequality of \eqref{error1} we have
	\begin{align}\label{error1L}
		\$R^{\textrm{error}}_1\$_{L^1,1}\lesssim
		\frac{\lambda_q^{34+61\beta}}{\lambda_{q+1}^{1-\alpha}}+\frac{\lambda_q^{34+50\beta}}{\lambda_{q+1}^{1-\alpha}}64L^4 \leq  \frac{1}{7{\cdot}48}r_R\delta_{q+3}\underline{e},
	\end{align}
	where we used $b>35$, $3b^3\beta+b\alpha<1$ to have $\lambda_{q}^{34+61\beta+b\alpha}<\lambda_{q}^{b-2b^3\beta}$, and we chose $a$ sufficiently large to absorb the constant.
	
	\subsubsection{Estimates on $R^{\textrm{error}}_2$}
	By Lemma~\ref{Belw11}, we have $B_k\cdot k=0$. Hence, we can write
	\begin{align*}
		\div (w_{q+1}^{(p)}\otimes(v_\ell+z_\ell))&=w_{q+1}^{(p)}\cdot \nabla (v_\ell+z_\ell)+ \div (w_{q+1}^{(p)} )(v_\ell+z_\ell)
		\\&=\sum_{k\in \Lambda}  \left[ a_k(B_k\cdot \nabla)(v_\ell+z_\ell) + (v_\ell+z_\ell)(B_k\cdot \nabla a_k) \right] e^{i\lambda_{q+1}k\cdot x}.		 
	\end{align*}
	We  apply Lemma~\ref{lemma1} with $m=1$ and use \eqref{a_k1}, \eqref{a_k0} to have
	\begin{equation}\label{error 2}
		\aligned
		&\|\mathcal{R}(\div (w_{q+1}^{(p)}\otimes(v_\ell+z_\ell)))\|_{C^0_{t,x}} 
		\\&\lesssim  \frac{1}{\lambda_{q+1}^{1-\alpha}}\left(\|a_k\cdot \nabla(v_\ell+z_\ell) + (v_\ell+z_\ell)\cdot \nabla a_k\|_{C^0_{t,x}} + \|a_k\cdot \nabla(v_\ell+z_\ell) + (v_\ell+z_\ell)\cdot \nabla a_k\|_{C^0_tC^2_x}\right) 
		\\&\lesssim \frac{\ell^{-\frac{11}{2}}}{\lambda_{q+1}^{1-\alpha}} (1+\|v_q\|_{C_{[t-1,t+1],x}^0}+\|z_q\|_{C_{[t-1,t+1],x}^0})^3(1+\|\mathring{R}_q\|_{C^0_{[t-1,t+1],x}})^3
		\\&+ \frac{\ell^{-\frac{13}{2}}}{\lambda_{q+1}^{1-\alpha}}(1+\|v_q\|_{C^0_{[t-1,t+1],x}}+\|z_q\|_{C^0_{[t-1,t+1],x}})^4(1+\|\mathring{R}_q\|_{C^0_{[t-1,t+1],x}})^4.
		\endaligned
	\end{equation}
	Taking expectation and using \eqref{estimate-zq'}, \eqref{vqa} and \eqref{vqc}, we obtain
	\begin{equation}\label{error2}
		\aligned
		\$\mathcal{R}(\div (w_{q+1}^{(p)}\otimes(v_\ell+z_\ell)))\$_{C^0,n} &\lesssim \frac{ \ell^{-\frac{11}{2}}}{\lambda_{q+1}^{1-\alpha}}\left( 1+ \$v_q\$^3_{C^0,6n}+\$z_q\$^3_{C^0,6n}\right)\left( 
		1+\$ \mathring{R}_q\$^3_{C^0,6n}\right) 
		\\&\quad + \frac{\ell^{-\frac{13}{2}}}{\lambda_{q+1}^{1-\alpha}}\left( 1+\$v_q\$^4_{C^0,8n}+\$z_q\$^4_{C^0,8n}\right) \left( 1+\$\mathring{R}_q\$^4_{C^0,8n}\right) 
		\\&\lesssim \frac{\lambda_q^{26}}{\lambda_{q+1}^{1-\alpha}}\left( 1+\lambda_q^{11n\beta}+L^4(8n)^2\right)\left( 1+\lambda_q^{32n\beta} \right)  
		\\&\lesssim \frac{\lambda_q^{26+43n\beta}}{\lambda_{q+1}^{1-\alpha}}+\frac{\lambda_q^{26+32n\beta}}{\lambda_{q+1}^{1-\alpha}}L^4(8n)^2
		\leq \begin{cases}
			\frac{1}{10}\delta_{q+2},&n=1,2,\\
			\frac{1}{10}\lambda_{q+1}^{n\beta},&n\geq3,
		\end{cases}
		\endaligned
	\end{equation}
	where we used $b>43$, $3b^2\beta+b\alpha <1$ to have $\lambda_{q}^{26+86\beta+b\alpha}<\lambda_{q}^{b-2b^2\beta}$ and used $a^{(b-43)\beta n}>(8n)^2L^4$ to have
	$L^4(8n)^2\lambda_{q}^{26+43n\beta+b\alpha}<\lambda_{q}^{b+bn\beta}$ and we chose $a$ sufficiently large to absorb the constant. Then taking $n=1$ in the third inequality of \eqref{error2} implies 
	\begin{align}\label{error2L}
		\$R^{\textrm{error}}_2\$_{L^1,1}\lesssim \frac{\lambda_q^{26+43\beta}}{\lambda_{q+1}^{1-\alpha}} \leq  \frac{1}{7{\cdot}48}r_R\delta_{q+3}\underline{e},
	\end{align}
	which requires $b>43$, $3b^3\beta+b\alpha<1$ to have $\lambda_{q}^{26+43\beta+b\alpha}<\lambda_{q}^{b-2b^3\beta}$ and
	$a$ large enough to absorb the constant.
	
	\subsubsection{Estimates on $R^{\textrm{error}}_3$} Using \eqref{estimate-zq'}, \eqref{wc0}, \eqref{vq +1} and H\"{o}lder's inequality, we obtain
	\begin{equation}\label{error3}
		\aligned
		&\$w_{q+1}^{(c)}\mathring{\otimes}(v_{q+1}+z_\ell)+(v_{q+1}+z_\ell)\mathring{\otimes} w_{q+1}^{(c)}- w_{q+1}^{(c)}\mathring{\otimes} w_{q+1}^{(c)}\$_{C^0,n}
		\\ &\lesssim \$w_{q+1}^{(c)}\$_{C^0,2n}\left( \$v_{q+1}\$_{C^0,2n}+ \$z_q\$_{C^0,2n}+\$w_{q+1}^{(c)}\$_{C^0,2n}\right) 
		\\&\lesssim \left(  \frac{\lambda_q^{10+18n\beta}}{\lambda_{q+1}} + \frac{\lambda_q^{10+16n\beta}}{\lambda_{q+1}}\sqrt{4n}L\right) \left( \ \frac{\lambda_q^{10+18n\beta}}{\lambda_{q+1}} + \frac{\lambda_q^{10+16n\beta}}{\lambda_{q+1}}\sqrt{4n}L+ \lambda_{q+1}^{\frac{2}{3}n\beta}+\sqrt{2n}L \right) 
		\\&\lesssim \frac{\lambda_q^{20+36n\beta}}{\lambda_{q+1}^2}4nL^2+ \frac{\lambda_q^{10+18n\beta}}{\lambda_{q+1}}\lambda_{q+1}^{\frac{2}{3}n\beta}8nL^2
		\leq 
		\begin{cases}
			\frac{1}{10}\delta_{q+2},& n=1,2,\\
			\frac{1}{10}\lambda_{q+1}^{n\beta},&n\geq 3,
		\end{cases}
		\endaligned
	\end{equation}
	where we used $b>54$, $2b^2\beta+2b\beta<1$ to have $\lambda_{q}^{10+36\beta+\frac43b\beta}<\lambda_q^{b-2b^2\beta}$ and used $a^{(\frac{b}{3}-18)\beta n}>8nL^2$ to have $8nL^2\lambda_{q}^{10+18n\beta}<\lambda_{q}^{b+\frac13bn\beta}$ and chose $a$ sufficiently large to absorb the constant. Then taking $n=1$ in the third inequality of \eqref{error3} we obtain
	\begin{align}\label{error3L}
		\$R^{\textrm{error}}_3\$_{L^1,1}\lesssim 
		\frac{\lambda_q^{20+36\beta}}{\lambda_{q+1}^2}+\frac{\lambda_q^{10+18\beta}}{\lambda_{q+1}}\lambda_{q+1}^{\frac{2}{3}\beta}
		\leq  \frac{1}{7{\cdot}48}r_R\delta_{q+3}\underline{e},
	\end{align}
	where we used $b>54$, $2b^3\beta+2b\beta<1$ to have $\lambda_{q}^{10+18\beta+2b^3\beta+b\beta}<\lambda_{q}^{b/3}$ and chose $a$ sufficiently large to absorb the constant.
	
	\subsubsection{Estimates on the commutator error}Finally, it remains to estimate $R^{\textrm{com}}$ and $R^{\textrm{cor}}$. We begin with $R^{\textrm{cor}}$. By applying a standard mollification estimate, for any $t\in \mR$ and $\delta \in (0,\frac{1}{12})$, we have 
	\begin{align*}
		\|z_\ell(t)-z_q(t)\|_{L^\infty}\lesssim \ell\|z_q\|_{C^0_{[t-1,t]}C^1_x}+\ell^{1/2-\delta}\|z_q\|_{C^{1/2-\delta}_{[t-1,t]}C^0_x}.
	\end{align*} 
	Taking expectation and using \eqref{estimate-zq'}, we obtain for $n\geq1$
	\begin{align}\label{z1}
		\$z_\ell -z_q\$_{C^0,n}\lesssim \ell \$z_q\$_{C^1,n}+ \ell^{1/2-\delta}	\$z_q\$_{C_t^{1/2-\delta}C_x^0,n}\lesssim \lambda_q^{-1}\lambda_{q+1}^{\frac{\gamma}{8}}\sqrt{2n}L.
	\end{align}
	By the definition of $z_q$, we have 
	\begin{align*}
		z_{q+1}(t,x)-z_q(t,x)=\sum_{\lambda_{q+1}^{\frac{\gamma}{8}}< |k|\leq\lambda_{q+2}^{\frac{\gamma}{8}}} e^{ik\cdot x}\hat{z}(t,k),
	\end{align*}
	where $\hat{z}$ is the Fourier transform of $z$, and $k\in \mathbb{Z}^3$. By a direct calculation with H\"{o}lder's inequality, it follows that for any $t\in \mR$, 
	\begin{equation}\label{zq+1 - zq}
		\aligned
		\|z_{q+1}(t)-z_q(t)\|_{L^\infty} 
		\lesssim\lambda_{q+1}^{-\frac{\gamma\varkappa}{8}} \|z(t)\|_{H^{\frac{3}{2}+\varkappa}}.
		\endaligned
	\end{equation}
	Then taking expectation and using \eqref{estimate-zq'} we obtain for $n\geq1$
	\begin{align}\label{z2}
		\$z_{q+1}-z_q\$_{C^0,n}\lesssim \lambda_{q+1}^{-\frac{\gamma\varkappa}{8}}	\$z\$_{H^{\frac32+\varkappa},n}\lesssim \lambda_{q+1}^{-\frac{\gamma\varkappa}{8}}\sqrt{2n}L.
	\end{align}
	Combining \eqref{vqa}, \eqref{z1} with \eqref{z2}, we obtain
	\begin{equation}\label{Rcom11}
		\aligned
		\$v_{q+1}\otimes(z_{q+1}-z_\ell)\$_{C^0,n}&\leq  \$v_{q+1}\$_{C^0,2n}\left( \$z_{q+1}-z_q\$_{C^0,2n}+\$z_\ell -z_q\$_{C^0,2n}\right) 
		\\ &\lesssim \sqrt{2n}L \lambda_{q+1}^{\frac{2}{3}n\beta}(\lambda_{q+1}^{-\frac{\gamma\varkappa}{8}}+\lambda_q^{-1}\lambda_{q+1}^{\frac{\gamma}{8}}).
		\endaligned
	\end{equation}
	Similarly, using \eqref{estimate-zq'}, \eqref{z1} and \eqref{z2}, we also have 
	\begin{align}\label{Rcom12}
		\$ z_{q+1}\otimes z_{q+1}-z_\ell \otimes z_\ell\$_{C^0,n} 
		&\leq \$ z_{q+1}\otimes (z_{q+1}-z_\ell )\$_{C^0,n} +   \$  (z_{q+1}-z_\ell )\otimes z_\ell\$_{C^0,n} \notag
		\\ &\leq (\$z_{q+1}\$_{C^0,2n}+\$z_{\ell}\$_{C^0,2n}) \left( \$z_{q+1}-z_q\$_{C^0,2n}+\$z_\ell -z_q\$_{C^0,2n}\right) \notag
		\\ &\leq \$z\$_{H^{\frac32+\varkappa},2n} \left( \$z_{q+1}-z_q\$_{C^0,2n}+\$z_\ell -z_q\$_{C^0,2n}\right) 
		\\& \lesssim 2nL^2(\lambda_{q+1}^{-\frac{\gamma\varkappa}{8}}+\lambda_q^{-1}\lambda_{q+1}^{\frac{\gamma}{8}}).\notag
	\end{align}  
	Hence, it follows from \eqref{Rcom11} and \eqref{Rcom12} that
	\begin{align}\label{Rcom1}
		\$ R^{\textrm{cor}}\$_{C^0,n} &\lesssim \$v_{q+1}\otimes(z_{q+1}-z_\ell)\$_{C^0,n} + \$ z_{q+1}\otimes z_{q+1}-z_\ell \otimes z_\ell\$_{C^0,n} +	\$\mathcal{R}(z_{q+1}-z_{\ell})\$_{C^0,n}\notag
		\\ &\lesssim \sqrt{2n}L \lambda_{q+1}^{\frac{2}{3}n\beta}(\lambda_{q+1}^{-\frac{\gamma\varkappa}{8}}+\lambda_q^{-1}\lambda_{q+1}^{\frac{\gamma}{8}})+2nL^2(\lambda_{q+1}^{-\frac{\gamma\varkappa}{8}}+\lambda_q^{-1}\lambda_{q+1}^{\frac{\gamma}{8}})
		\\&\lesssim2nL^2\lambda_{q+1}^{\frac{2}{3}n\beta}(\lambda_{q+1}^{-\frac{\gamma\varkappa}{8}}+\lambda_q^{-1}\lambda_{q+1}^{\frac{\gamma}{8}})
		\leq \begin{cases}
			\frac{1}{10}\delta_{q+2} ,& n=1,2,\\
			\frac{1}{10}\lambda_{q+1}^{n\beta},&n\geq3,
		\end{cases}\notag
	\end{align}
	where we used $2b\beta+2\beta<\frac{\gamma\varkappa}{8}$,
	$2b^2\beta+2b\beta+b\gamma<1$ to have $\lambda_{q}^{\frac43b\beta-\frac{\gamma\varkappa}{8}b}<\lambda_{q}^{-2b^2\beta}$ and $\lambda_{q}^{\frac43b\beta+\frac{\gamma}{8}b}<\lambda_{q}^{1-2b^2\beta}$. We also used $2nL^2<a^{\frac{b}{3}n\beta}$ to have $2nL^2\lambda_{q}^{\frac23nb\beta}<\lambda_{q}^{nb\beta}$ and $a$ was chosen sufficiently large to absorb the constant.
	Taking $n=1$ in the third inequality of \eqref{Rcom1}, we obtain
	\begin{align}\label{Rcom1L1}
		\$ R^{\textrm{cor}}\$_{L^1,1}\lesssim 2L^2\lambda_{q+1}^{\frac{2}{3}\beta}(\lambda_{q+1}^{-\frac{\gamma\varkappa}{8}}+\lambda_q^{-1}\lambda_{q+1}^{\frac{\gamma}{8}}) \leq \frac{1}{7{\cdot}48}r_R\delta_{q+3}\underline{e},
	\end{align}
	where we used $2b^2\beta+\beta<\frac{\gamma\varkappa}{8}$, $2b^3\beta+b\beta+b\gamma<1$ to have $\lambda_{q}^{\frac23b\beta-\frac{\gamma\varkappa}{8}b}<\lambda_{q}^{-2b^3\beta}$ and $\lambda_{q}^{\frac23b\beta+\frac{\gamma}{8}b}<\lambda_{q}^{1-2b^3\beta}$ and we chose $a$ sufficiently large to absorb the constant.
	
	We next estimate $R^{\textrm{com}}$, using standard mollification estimates we obtain for any $t\in \mR$ and $\delta\in(0,\frac{1}{12})$
	\begin{equation}\label{moll} 
		\aligned
		\|&R^\textrm{com}(t)\|_{L^\infty}
		\\ &\lesssim   
		\left(\ell \|v_q\|_{C_{[t-1,t],x}^1}+\ell\|z_q\|_{C^0_{[t-1,t]}C_x^1}
		+\ell^{\frac{1}{2}-\delta}\|z_q\|_{C^{\frac{1}{2}-\delta}_{[t-1,t]}C_x^0}\right)\left(\|z_q\|_{C^0_{[t-1,t],,x}}+\|v_q\|_{C^0_{[t-1,t],x}}\right).
		\endaligned
	\end{equation}
	Taking expectation and using  \eqref{estimate-zq'}, \eqref{vqa} and \eqref{vqb}, we obtain
	\begin{equation}\label{Rcomn}
		\aligned
		\$R^\textrm{com}\$_{C^0,n}
		&\lesssim 
		\left(\ell \$v_q\$_{C_{t,x}^1,2n} + \ell\$z_q\$_{C^1,2n}+\ell^{\frac{1}{2}-\delta}\$z_q\$_{C_t^{\frac12-\delta}C_x^0,2n}\right)\left(  \$v_q\$_{C^0,2n} +\$z_q\$_{C^0,2n}\right)   
		\\&\lesssim ( \lambda_q^{-3+2n\beta}+\lambda_q^{-1}\lambda_{q+1}^{\frac{\gamma}{8}}\sqrt{2n}L) ( \lambda_{q}^{\frac{2}{3}n\beta}+\sqrt{2n}L) 
		\\&\lesssim 2nL^2 (\lambda_q^{-3+3n\beta}+\lambda_q^{-1+\frac{2}{3}n\beta}\lambda_{q+1}^{\frac{\gamma}{8}})
		\leq 
		\begin{cases}
			\frac{1}{10}\delta_{q+2},&n=1,2,\\
			\frac{1}{10}\lambda_{q+1}^{n\beta},&n\geq3,
		\end{cases}
		\endaligned
	\end{equation}
	where we used $2b^2\beta+6\beta+b\gamma<1$ to have $\lambda_{q}^{-3+6\beta}<\lambda_{q}^{-2b^2\beta}$ and $\lambda_{q}^{-1+\frac43\beta+\frac{\gamma}{8}b}<\lambda_{q}^{-2b^2\beta}$. We also used $2nL^2<a^{(b-3)n\beta}$ to have $2nL^2\lambda_{q}^{3n\beta}<\lambda_{q}^{bn\beta}$ and chose $a$ sufficiently large to absorb the constant. Taking $n=1$ in the third inequality of \eqref{Rcomn} we obtain
	\begin{align}\label{RcomL1}
		\$R^\textrm{com}\$_{L^1,1}\lesssim 2L^2 ( \lambda_q^{-3+3\beta}+\lambda_q^{-1+\beta}\lambda_{q+1}^{\frac{\gamma}{8}}) \leq \frac{1}{7{\cdot}48}r_R\delta_{q+3}\underline{e},
	\end{align}
	which requires $2b^3\beta+3\beta+b\gamma<1$ to have $2L^2\lambda_{q}^{\beta+\frac{\gamma}{8}b-1}<\lambda_{q}^{-2b^3\beta}$ and $a$ large enough to absorb the constant in the last inequality. 
	
	At last, combining \eqref{tran2}, \eqref{osc}, \eqref{error1}, \eqref{error2}, \eqref{error3}, \eqref{Rcom1} and \eqref{Rcomn} we obtain 
	\begin{align*}
		\$\mathring{R}_{q+1}\$_{C^0,n}\leq
		\begin{cases}
			\delta_{q+2},&n=1,2,\\
			\lambda_{q+1}^{n\beta},&n\geq 3.
		\end{cases}
	\end{align*}
	Combining \eqref{RtraL}, \eqref{RoscL}, \eqref{error1L}, \eqref{error2L}, \eqref{error3L}, \eqref{Rcom1L1} and \eqref{RcomL1}, we obtain
	\begin{align*}
		\$\mathring{R}_{q+1}\$_{L^1,1}\leq \frac{1}{48}r_R\delta_{q+3}\underline{e}.
	\end{align*}
	Hence, we have verified \eqref{vqc} and \eqref{vqd} hold at the level $q+1$.
	
	\subsection{Inductive estimates for the energy}\label{s:en}
	To conclude the proof of Proposition~\ref{p:iteration}, we shall verify \eqref{vqe} holds at the level $q+1$. 
	\bp It holds for $t\in\mR$
	\begin{align}\label{1/4}
		\Big| e(t)(1-\delta_{q+2})-\E\|v_{q+1}(t)+{z_{q+1}}(t)\|_{L^2}^2\Big|\leq& \frac14\delta_{q+2}e(t).
	\end{align}
	
	\ep
	\begin{proof}
		By the definition of $\zeta_q$ in \eqref{defzetaq}, we find
		\begin{align}\label{Eq+1}
			\begin{aligned}
				&\left| e(t)(1-\delta_{q+2})- \E\|v_{q+1}(t)+{z_{q+1}}(t)\|_{L^2}^2 \right| \\&\leq \E \left| \int_{\mathbb{T}^3} |w_{q+1}^{(p)}|^2-3\zeta_q\dif x \right| +\E \left|  \int_{\mathbb{T}^3}|w_{q+1}^{(c)}|^2+2w_{q+1}^{(p)}\cdot w_{q+1}^{(c)} \dif x \right| 
				\\&\qquad +\E \left| \int_{\mathbb{T}^3}|v_\ell-v_q+z_{q+1}-z_q|^2 \dif x  \right| + 2\E \left|  \int_{\mathbb{T}^3}(v_\ell-v_q+z_{q+1}-z_q)(v_q+z_q)\dif x  \right| 
				\\&\qquad + 2\E \left| \int_{\mathbb{T}^3}(v_\ell+z_{q+1})\cdot w_{q+1} \dif x \right| .
			\end{aligned}
		\end{align}
		Let us begin with the bound of the first term on the right-hand side of \eqref{Eq+1}. Taking trace on the both sides of \eqref{wpxwp} and using the fact that $\mathring{R}_\ell$ is traceless, we obtain for any $t\in\mR$
		\begin{align*}
			|w_{q+1}^{(p)}|^2-3\zeta_q=3r_0^{-1}\sqrt{\ell^2+|\mathring{R}_\ell|^2}+3(\zeta_\ell-\zeta_q)+\sum_{k\in \Lambda}\tr(U_k) e^{i\lambda_{q+1}k\cdot x}.
		\end{align*}
		Hence, we have
		\begin{align}\label{Eq1}
			\begin{aligned}
				&\left| \E\int_{\mathbb{T}^3} |w_{q+1}^{(p)}|^2-3\zeta_q\dif x \right| 
				\\&\leq 3 (2\pi)^3r_0^{-1}\ell+3r_0^{-1}\mathbf{E}\|\mathring{R}_\ell\|_{L^1}+3(2\pi)^3|\zeta_\ell-\zeta_q|+\E\sum_{k\in \Lambda}\left| \int_{\mathbb{T}^3}\tr(U_k) e^{i\lambda_{q+1}k\cdot x}\dif x\right|.
			\end{aligned}
		\end{align}
		We next estimate each term separately.
		From the choice of parameters, it follows that
		\begin{align*}
			3 r_0^{-1}(2\pi)^3\ell=3 r_0^{-1}(2\pi)^3\lambda_{q}^{-4}\leq\frac{1}{48}\delta_{q+2}e(t),
		\end{align*}
		which requires $b^2\beta<1$ and $a$ large enough to absorb the constant. Using \eqref{vqd} on $\mathring{R}_q$, $\supp\varphi_\ell\subset [0,\ell]$ and $r_R<r_0$, we deduce for $t\in\mR$
		\begin{align*}
			3r_0^{-1}\mathbf{E}\|\mathring{R}_\ell(t)\|_{L^1}\leq  \frac{1}{16}\delta_{q+2}e(t).
		\end{align*}
		For the third term in \eqref{Eq1}, we use \eqref{estimate-zq'}, \eqref{vqa} and \eqref{vqb} to have 
		\begin{align*}
			3(2\pi)^3|\zeta_\ell-\zeta_q|&\lesssim \ell \|e'\|_{C^{0}_{[t-1,t]}}+\ell  \E\|v_q\|_{C^1_{[t-1,t],x}}(\|v_q\|_{C_{[t-1,t],x}^0}+{\|z_q\|_{C_{[t-1,t],x}^0}})
			\\&\qquad+\ell^{1/2-\delta}\E\|{z_q}\|_{C_{[t-1,t]}^{1/2-\delta}C_x^0}(\|v_q\|_{C_{[t-1,t],x}^0}+\|{z_q}\|_{C_{[t-1,t],x}^0})
			\\&\lesssim \ell \tilde e+(\ell \$v_q\$_{C_{t,x}^1,2}+\ell^{1/2-\delta}\$z_q\$_{C_t^{1/2-\delta}C_x^0,2})(\$v_q\$_{C^0,2}+\$z_q\$_{C^0,2})
			\\&\lesssim \ell \tilde e+(\ell \lambda_q^{1+2\beta}+\ell^{1/2-\delta}\sqrt{2}L)(\lambda_{q}^\beta+\sqrt{2}L)\\
			&\lesssim \lambda_q^{-4}\tilde{e}+\lambda_{q}^{-1+3\beta}L^2
			\leq \frac{1}{ 48}\delta_{q+2}e(t),
		\end{align*}
		which requires $3b^2\beta<1$ to have $\lambda_q^{-1+3\beta}<\lambda_{q}^{-2b^2\beta}$ and $a$  large enough to absorb the constant.
		
		For the last term in \eqref{Eq1}, we apply the first part of Lemma~\ref{lemma1} with $m=1$ and use \eqref{estimate-zq'}, \eqref{vqa} and \eqref{vqc} to obtain 
		\begin{equation}\label{eq:Uk}
			\aligned
			\sum_{k\in \Lambda} \E&\left| \int_{\mathbb{T}^3}\tr(U_k) e^{i\lambda_{q+1}k\cdot x}\dif x\right| \lesssim \frac{1}{\lambda_{q+1}}\sum_{k\in \Lambda} \E\|\tr(U_k)\|_{C^0_tC_x^1}
			\\& \lesssim \frac{\ell^{-\frac{5}{2}}}{\lambda_{q+1}}\E(1+\|v_q\|_{C_{[t-1,t+1],x}^0}+\|z_q\|_{C_{[t-1,t+1],x}^0}) (1+\|\mathring{R}_q\|^\frac{1}{2}_{C_{[t-1,t+1],x}^0})^5
			\\ &\lesssim \frac{\ell^{-\frac{5}{2}}}{\lambda_{q+1}}(1+\$v_q\$_{C^0,2}+\$z_q\$_{C^0,2}) (1+\$\mathring{R}_q\$_{C^0,5}^{\frac{5}{2}})
			\\&\lesssim \frac{\lambda_{q}^{10}}{\lambda_{q+1}}(1+\lambda_q^{\beta}+\sqrt{2}L)(1+\lambda_{q}^{13\beta})
			\lesssim \frac{\lambda_{q}^{10+14\beta}L}{\lambda_{q+1}} \leq\frac{1}{48}\delta_{q+2}e(t),
			\endaligned
		\end{equation}
		where we used $b>11$ and $3b^2\beta<1$ to have $\lambda_{q}^{10+14\beta}<\lambda_{q}^{b-2b^2\beta}$ and we chose $a$ large enough to absorb the constant. 
		
		Returning to \eqref{Eq+1}, it remains to control 
		\begin{align*}
			&\E \left|  \int_{\mathbb{T}^3}|w_{q+1}^{(c)}|^2+2w_{q+1}^{(p)}\cdot w_{q+1}^{(c)} \dif x \right| 
			+\E \left| \int_{\mathbb{T}^3}|v_\ell-v_q+z_{q+1}-z_q|^2 \dif x  \right| 
			\\&\qquad+ 2\E \left| \int_{\mathbb{T}^3}(v_\ell-v_q+z_{q+1}-z_q)(v_q+z_q)\dif x  \right| 
			+2\E \left| \int_{\mathbb{T}^3}(v_\ell+z_{q+1})\cdot w_{q+1} \dif x \right| .
		\end{align*}
		Using the estimates \eqref{wpr} and \eqref{wc0} we have 
		\begin{align*}
			&\E \left|  \int_{\mathbb{T}^3}|w_{q+1}^{(c)}|^2+2w_{q+1}^{(p)}\cdot w_{q+1}^{(c)} \dif x \right|  
			\lesssim \$w_{q+1}^{(c)}\$^2_{C^0,2}+\$w_{q+1}^{(p)}\$_{C^0,2}\$w_{q+1}^{(c)}\$_{C^0,2}
			\\ &\lesssim   \frac{\lambda_q^{10+18\beta}}{\lambda_{q+1}} + \frac{\lambda_q^{10+16\beta}}{\lambda_{q+1}} +\frac{\lambda_q^{20+36\beta}}{\lambda_{q+1}^2} + \frac{\lambda_q^{20+32\beta}}{\lambda_{q+1}^2}\leq \frac{1}{48}\delta_{q+2}e(t),
		\end{align*}
		which requires $b>11$, $3b^2\beta<1$ to have $\lambda_q^{10+36\beta}<\lambda_{q}^{b-2b^2\beta}$ and $a$ sufficiently large to absorb the constant. Next, we use \eqref{vqa}, \eqref{vqb}, \eqref{vq-vl} and \eqref{z2} to obtain
		\begin{align*}
			&\E\left| \int_{\mathbb{T}^3}|v_\ell-v_q+z_{q+1}-z_q|^2 \dif x  \right| 
			+ 2\E \left| \int_{\mathbb{T}^3}(v_\ell-v_q+z_{q+1}-z_q)(v_q+z_q)\dif x  \right| 
			\\&\lesssim \E\|v_\ell-v_q+z_{q+1}-z_q\|_{C_{t,x}^0}^2+\E\|v_\ell-v_q+z_{q+1}-z_q\|_{C_{t,x}^0}\|z_q+v_q\|_{C_{t,x}^0}
			\\&\lesssim \$v_\ell-v_q\$^2_{C^0,2}+\$z_{q+1}-z_q\$^2_{C^0,2}+(\$v_\ell-v_q\$_{C^0,2}+\$z_{q+1}-z_q\$_{C^0,2})(\$v_q\$_{C^0,2}+\$z_q\$_{C^0,2})
			\\&\lesssim \ell^2\$v_q\$^2_{C^1_{t,x},2}+\lambda_{q+1}^{-\frac{\gamma\varkappa}{4}}	\$z\$_{H^{\frac32+\varkappa},2}^2+\left(\ell\$v_q\$_{C^1_{t,x},2}+\lambda_{q+1}^{-\frac{\gamma\varkappa}{8}}	\$z\$_{H^{\frac32+\varkappa},2}\right)
			\left(\lambda_{q}^\beta+2L\right)
			\\& \lesssim \lambda_{q}^{-3+4\beta} +\lambda_{q}^{\beta}\lambda_{q+1}^{-\frac{\gamma\varkappa}{8}}\leq \frac{1}{48}\delta_{q+2}e(t),
		\end{align*}
		where we used $2b^2\beta+\beta<\min{\{\frac{\gamma\varkappa}{8}b,1\}}$ to have $\lambda_{q}^{\beta-\frac{\gamma\varkappa}{8}b}<\lambda_{q}^{-2b^2\beta}$ and $\lambda_{q}^{-3+4\beta}<\lambda_{q}^{-2b^2\beta}$, and we chose $a$ large enough to absorb the constant. 
		
		Regarding the last term of \eqref{Eq+1}, it follows from
		\eqref{wq+1} that $w_{q+1}$ may be written as the curl of a vector field
		\begin{equation*}
			w_{q+1}=\frac{1}{\lambda_{q+1}}\textrm{curl} \left( \sum_{k\in \Lambda}  ia_k \frac{k\times B_k}{|k|^2}e^{i\lambda_{q+1}k\cdot x}\right).
		\end{equation*}
		Integrating by parts  and using the estimates \eqref{estimate-zq'}, \eqref{vqb}, \eqref{vqc} and \eqref{a_k0} we obtain
		\begin{equation}\label{vl w q+1}
			\aligned
			2\E \left| \int_{\mathbb{T}^3}(v_\ell+z_{q+1})\cdot w_{q+1} \dif x \right| &\lesssim \frac{1}{\lambda_{q+1}} \E\left(  \sum_{k\in \Lambda}\left\| a_k\frac{B_k}{|k|}e^{i\lambda_{q+1}k\cdot x}\right\| _{C_{t,x}^0}\|v_\ell+z_{q+1}\|_{C^0_tC^1_x}\right) 
			\\&\lesssim \frac{1}{\lambda_{q+1}}\$a_k\$_{C^0,2}(\$v_q\$_{C^1_{t,x},2}+\$z_{q+1}\$_{C^1_x,2})
			\\&\lesssim \frac{1}{\lambda_{q+1}}\left(\lambda_{q}^{1+2\beta}+\lambda_{q+1}^{\frac{\gamma}{8}} \$z_{q+1}\$_{H^{\frac32+\varkappa},2} \right)\$\mathring{R}_q\$_{C^0,1}^{\frac12}
			\\&\lesssim \frac{1}{\lambda_{q+1}}(\lambda_{q}^{1+2\beta}+\lambda_{q+1}^{\frac{\gamma}{8}}  ) \leq \frac{1}{48}\delta_{q+2}e(t),
			\endaligned
		\end{equation}
		which requires $3b^2\beta +b\gamma<b$ and $a$ sufficiently large to absorb the constant. Summarizing the above estimates, we have
		\begin{align*}
			\Big| e(t)(1-\delta_{q+2})-\E\|v_{q+1}(t)+{z_{q+1}}(t)\|_{L^2}^2\Big|\leq\frac14\delta_{q+2}e(t).
		\end{align*} Hence, we complete the proof of Proposition~\ref{p:iteration}.
	\end{proof}

	\section{Construction of solutions with prescribed initial condition}\label{sec 4}
	This section aims to prove Theorem~\ref{Thm1.3}, which establishes the existence of global-in-time, probabilistically strong and analytically weak solutions to \eqref{eul1} for every given divergence-free initial condition in $C^{\varkappa}$. In order to prescribe the initial value of the solution, we follow the approach of \cite{HZZ21markov} and refine the previous construction. Our convex integration scheme relies on perturbations defined for every time $t>0$, after which we select a sequence of stopping times $T_L$ to give uniform bounds in $\omega \in \Omega$ on suitable norms of the solutions up to time $T_L$, depending on $L$. This approach was previously pursued in \cite{HLP22} and \cite{Umb23}.
	
	First, we introduce some necessary modifications to the notation. Given $\varkappa>0$, let $u_{0}\in C^{\varkappa}$ be $\mathbf{P}$-a.s. independent of the given Wiener process $B$. Denote by $(\mathcal{F}_{t})_{t\geq0}$ the augmented joint canonical filtration on $(\Omega,\mathcal{F})$ generated by $B$ and $u_{0}$. Therefore, $B$ is an $(\mathcal{F}_{t})_{t\geq0}$-Wiener process and $u_{0}$ is $\mathcal{F}_0$-measurable. 
	Given a Banach space $(E,\|\cdot\|_E)$ and $T>0$, we write $C_TE=C([0,T];E)$ equipped with the supremum norm $\|f\|_{C_TE}=\sup_{t\in[0,T]}\|f(t)\|_{E}$. For $\kappa\in(0,1)$, we  define $C^\kappa_TE$ as the space of $\kappa$-H\"{o}lder continuous functions from $[0,T]$ to $E$, endowed with the norm $\|f\|_{C^\kappa_TE}=\sup_{s,t\in[0,T],s\neq t}\frac{\|f(s)-f(t)\|_E}{|t-s|^\kappa}+\|f(t)\|_{C_TE}$.
	For $T>0$, we denote by  $C^{N}_{T,x}$ the space of $C^{N}$-functions on $[0,T]\times\mathbb{T}^{3}$, $N\in\N_{0}$, equipped with the norm
	$$
	\|f\|_{C^N_{T,x}}=\sum_{\substack{0\leq n+|\alpha|\leq N\\ n\in\N_{0},\alpha\in\N^{3}_{0} }}\|\partial_t^n \partial_x^\alpha f\|_{C_T L^\infty}.
	$$
	
	Following \cite{HZZ21markov}, we incorporate the initial value into the linear part $z$ by defining $z$ as the solution of the stochastic linear equation originating from $u_0$:
	\begin{equation}\label{li:sto}
		\aligned
		\dif z+z \dif t&=\dif B,
		\\\div z&=0,
		\\ z(0)&=u_0,
		\endaligned
	\end{equation}
	and let $u$ be any solution of \eqref{eul1} with the same initial condition. Then the difference $v:=u-z$ solves
	\begin{equation}\label{nonlinear1}
		\aligned
		\partial_t v-z+\div((v+z)\otimes (v+z))+\nabla P&=0,
		\\\div v&=0,
		\\ v(0)&=0.
		\endaligned
	\end{equation}
	Here, $z$ is divergence-free under the assumptions on the noise $B$, and we denote the pressure term associated with $v$ by $P$.
	
	As before, the iteration is indexed by a parameter $q\in\mathbb{N}_{0}$. At each step $q$, a pair $(v_q, \mathring{R}_q)$ is constructed to solve the following system
	\begin{equation}\label{induction ps}
		\aligned
		\partial_tv_q-z_q+\div((v_q+{z_q})\otimes (v_q+{z_q}))+\nabla p_q&=\div \mathring{R}_q,
		\\
		\div v_q&=0,\\
		v_{q}(0)&=0.
		\endaligned
	\end{equation}
	Here, we decompose $z=z^{in}+Z$ with $z^{in}(t)=e^{-t}u_{0}$ and define $z_q:=z^{in}+Z_q=z^{in}+\mP_{\leq f(q)}Z$ with $f(q)=\lambda_{q+1}^{\gamma/8}$.
	Using the same argument as Lemma~\ref{Le1}, we have the following result.
	\bp\label{pro5.1}
	Suppose that $\tr((-\Delta)^{3/2+\varkappa} GG^*)<\infty$ for the given $\varkappa>0$ as above. Then for any   $\delta \in (0,\frac{1}{2})$ and $T>0$
	$$
	\E\left[\|Z\|_{C_T^{1/2-\delta}H^{3/2+\varkappa}}\right]<\infty.
	$$
	\ep
	
	By the Sobolev embedding, we know that $\|f\|_{L^\infty}\leq C_S\|f\|_{H^{3/2+\varkappa}}$ for $\varkappa>0$ and some constant $C_{S}\geq1$.
	We define the following stopping time for $0<\delta<\frac1{12}$ and $L\in \mathbb{N}$
	\begin{equation}\label{stopping time ps}
		\aligned
		T_L:=& \inf\{t\geq0: \|Z\|_{C_t^{1/2-\delta}H^{3/2+\varkappa}}\geq L/C_S\}
		\wedge  L,
		\endaligned
	\end{equation}
	which is $\mathbf{P}$-a.s. strictly positive such that $T_L\to \infty$ almost surely as $L\to \infty$.
	Furthermore, it holds for $t\in[0, T_L]$ that
	\begin{equation}\label{z ps}
		\| Z_q(t)\|_{L^\infty}\leq L, \quad\|\nabla Z_q(t)\|_{L^\infty}\leq L\lambda_{q+1}^{\frac\gamma8}, \quad \|Z_q\|_{C_t^{1/2-\delta}L^\infty}\leq L.
	\end{equation}
	
	Moreover, without loss of generality, we can suppose that  for the above given $\varkappa>0$
	\begin{align}\label{eq:u0} 
		\|u_{0}\|_{C^{\varkappa}}\leq N,
		\quad
		\mathbf{P}-a.s. 
	\end{align}
	for some finite constant $N$. 
	Indeed, for a general initial condition $u_0 \in C^{\varkappa}$  $\mathbf{P}$-a.s.,  one defines $\Omega_N := \{ N-1 \leq \| u_0 \|_{C^{\varkappa}} < N\} \in \mathcal{F}_0$.
	Then, given the existence of infinitely many solutions $u^N$ on each $\Omega_N$, one can define $u := \sum_{N \in \N} u^N \mathbf{1}_{\Omega_N}$, solving the equation with initial condition $u_0$.
	We maintain this additional assumption on the initial condition throughout the convex integration step in Proposition~\ref{p:iteration'}. 
	We also denote a deterministic constant $M_L:=(L+N)^2+L+N$, which ensures that for any $t\in [0,T_L]$,
	\begin{align}\label{induction z}
		\|z_q(t)\|_{L^\infty}\leq \|z^{in}(t)\|_{L^\infty}+\|Z_q(t)\|_{L^\infty}\leq M_L^{1/2}.
	\end{align}
	
	The frequency $\lambda_q$ and the mollification parameter $\ell$ all  retain the same structure as in Section \ref{311}, and we modify the value of $\delta_q$ as
	\begin{align*}
		\delta_{1}=M_L, \qquad \delta_{q}=\lambda_{q}^{-2\beta},\ q\geq2.
	\end{align*}
	We also denote $\tau_{q}=\lambda_{q+1}^{-2\beta}$, $q\in \mathbb{N}_0\cup\{-1\}$, $\sigma_q=\lambda_{q}^{-2\beta}$, $q\in \mathbb{N}_0\setminus\{2\}$ and $\sigma_2=K$, where $K\geq 1$ is a constant used  in the proof of Theorem~\ref{Thm1.3} to distinguish different solutions. Moreover, we define the parameter
		\begin{align*}
			L_q= M_L^{m^{q}},
		\end{align*} 
		where  $L\in  \mathbb{N}$ and the coeﬃcient $m\in \mathbb{N}$ satisfies $m<b$. The parameter $L_q$ serves to control the deterioration of iterative estimates over increasingly larger time intervals of the form $[0,T_L]$, $L\in  \mathbb{N}$.
	In addition, the values of the determining parameters $a,b,m,\gamma$ and $\beta$ may differ, since further conditions need to be satisfied. Details are provided in Section~\ref{s:par} below.

	Under the above assumptions, our main iteration is given as follows. The proof of this result is presented in Section~\ref{s:it'} below.

	\begin{proposition}\label{p:iteration'}
		Suppose that $\tr((-\Delta)^{3/2+\varkappa} GG^*)<\infty$ for the given $\varkappa>0$ as above.	Let $N\geq 1$ and assume \eqref{eq:u0}. There exists a choice of parameters $a, b,m,\gamma$ and $\beta$ such that the following holds true: Let $(v_{q},\mathring{R}_{q})$ for some $q\in\N_{0}$ be an $(\mathcal{F}_{t})_{t\geq 0}$-adapted solution to \eqref{induction ps} on $[0,T_L]$ for every $L\in  \mathbb{N}$ and satisfy
		\begin{equation}\label{inductionv ps}
			\|v_{q}(t)\|_{L^{\infty}}\leq\begin{cases}
				2\bar{M}L_q^{1/2}\sum_{ r=1}^{q}(\delta_{r}^{1/2}+\sigma_{r}^{1/2})+3\bar{M}M_0 L_q^{1/2},&t\in (\tau_{q}\wedge T_L,  T_L],\\
				0, &t\in [0,\tau_{q}\wedge T_L],
			\end{cases}
		\end{equation}
		where $M_0:=K+\sum_{r\geq1}\delta_r^{\frac12}<\infty$, and $\bar{M}$ is a universal constant which will be fixed in \eqref{bar M2} below. Suppose further that
		\begin{align}\label{inductionv C1}
			\|v_q\|_{C^1_{t,x}}&\leq  L_q\lambda_q^\frac32\delta_q^{\frac12},\quad t\in[0, T_L],
		\end{align}
		\begin{align}\label{eq:R}
			\|\mathring{R}_q(t)\|_{L^\infty}\leq
			L_q\delta_{q+1},\quad t\in (\tau_{q-1}\wedge T_L,T_L],
		\end{align}
		\begin{align}\label{bd:R}
			\|\mathring{R}_{q}(t)\|_{L^\infty}\leq L_q \sum_{r=1}^{q+1}\delta_r,\quad t\in[0, T_L].
		\end{align}
		Then there exists an $(\mathcal{F}_{t})_{t\geq 0}$-adapted process  $(v_{q+1},\mathring{R}_{q+1})$ which solves \eqref{induction ps} and for every $L\in \mathbb{N}$
		\begin{equation}\label{iteration ps}
			\|v_{q+1}(t)-v_q(t)\|_{L^\infty}\leq \begin{cases}
				\bar{M} L_{q+1}^{1/2} (\delta_{q+1}^{\frac12}+\sigma_{q+1}^{\frac12}),& t\in (2\tau_{q-1}\wedge T_L,T_L],\\
				\bar{M} L_{q+1}^{1/2}(M_0+\delta_{q+1}^{\frac12}+\sigma_{q+1}^{\frac12}),& t\in (\tau_{q+1}\wedge T_L,2\tau_{q-1}\wedge T_L],\\
				0,
				&t\in [0,\tau_{q+1}\wedge T_L],
			\end{cases}
		\end{equation}
		\begin{equation}\label{iteration R}
			\|\mathring{R}_{q+1}(t)\|_{L^\infty}\leq\begin{cases}
			L_{q+1}	\delta_{q+2},& t\in (\tau_q\wedge T_L,T_L],\\
			L_{q+1}	\delta_{q+2}+\sup_{s\in[t-\ell,t]}\|\mathring{R}_{q}(s)\|_{L^\infty},&t\in [0,\tau_q\wedge T_L].
			\end{cases}
		\end{equation}
		Furthermore, $(v_{q+1},\mathring{R}_{q+1})$ obeys \eqref{inductionv ps},  \eqref{inductionv C1}, \eqref{eq:R} and \eqref{bd:R} with $q$ replaced by $q+1$ and for $t\in (2\tau_{q-1}\wedge T_L,T_L]$, we have
		\begin{align}\label{p:gamma}
			\big|\|v_{q+1}\|_{L^2}^2-\|v_q\|_{L^2}^2-3\sigma_{q+1}\big|\leq M_1 L_{q+1} \delta_{q + 1},
		\end{align}
		where $M_1$ is a universal constant (see \eqref{def M_1} below). 
	\end{proposition}
	
	\begin{remark}\label{remark 2}
		In comparison to the previous results, we obtain solutions only in $L^p([0,T_L];C^\vartheta)$ for any $p\in[1,\infty)$, instead of achieving continuity in time. (The spatial H\"older regularity $\vartheta$ depends on $p$; see the proof of Theorem~\ref{Thm1.3} below for details).
		This limitation arises from the introduction of the cut-off function $\chi$ defined in \eqref{def:chi}, which ensures that the support of the new perturbation remains away from zero.
		This operation makes it difficult to absorb the $L^\infty$-norm of $\mathring{R}_{q+1}(t)$ into small parameter $\delta_{q+2}$ for $t\in [0,\tau_{q}\wedge T_L)$, as discussed in Section~\ref{415} below. Consequently, we impose assumption \eqref{bd:R} on $\|\mathring{R} _q\|_{L^\infty}$, which prevents $v_q$ from converging uniformly in the $L^\infty$-norm within $[0, T_L]$, see \eqref{iteration ps}.
	\end{remark}

	\begin{proof}[Proof of Theorem~\ref{Thm1.3}] 
		We aim to apply Proposition~\ref{p:iteration'} iteratively to obtain a sequence of solutions $(v_q,\mathring{R}_q)$ to the system \eqref{induction ps}.
		
		\emph{Step 1.} We define the initial step of the iteration $v_0$ to be identically zero
		on $[0,T_{L}]$. In that case, we have
		$\mathring{R}_0=z_{0}\mathring\otimes z_{0}-\mathcal{R}z_0$ so that for any $t\in [0,T_L]$
		\begin{align*}
			\|\mathring{R}_0(t)\|_{L^\infty}\leq \|z_0(t)\|_{L^\infty}^2+\|z_0(t)\|_{L^\infty}\leq (L+N)^2+L+N\leq L_0\delta_{1}.
		\end{align*}
		Hence, \eqref{eq:R} and \eqref{bd:R} are satisfied at the level  $q=0$.
		
		\emph{Step 2.}
		Let the additional assumption  \eqref{eq:u0} be satisfied  for some $N\geq 1$. We  iteratively apply Proposition~\ref{p:iteration'} and obtain $(\mathcal{F}_{t})_{t\geq 0}$-adapted process $(v_q,\mathring{R}_{q})$ satisfying the inductive assumption. Using \eqref{inductionv C1}, {\eqref{iteration ps}} and interpolation, we
		have for each $p\in[1,\infty)$, $\vartheta>0 $  and $q\geq2$
		\begin{align*}
			\int_0^{T_L}  \| v_{q + 1} (t) - v_q (t) \|_{C^{\vartheta}}^p \dif t
			&	\lesssim	\int_0^{T_L}\| v_{q + 1} (t) - v_q (t) \|_{L^\infty}^{(1-\vartheta)p} \| v_{q + 1} - v_q \|_{C^1_{[0,T_L],x}}^{p\vartheta}  \dif t 
			\\ &\lesssim \int_{\tau_{q+1} \wedge T_L}^{2 \tau_{q-1} \wedge T_L} L_{q+1}^p
				\bar{M}^p(M_0+\delta_{q+1}^{\frac12}+\sigma_{q+1}^\frac12)^{(1-\vartheta)p} \lambda_{q+1}^{\frac32 p\vartheta} \delta_{q+1}^{\frac12 p\vartheta} \dif t
			\\ &\qquad \qquad + \int_{2 \tau_{q-1} \wedge T_L}^{T_L} L_{q+1}^p \bar{M}^p ( \delta^{\frac12}_{q + 1}+\sigma_{q+1}^\frac12)^{(1-\vartheta)p}  \lambda_{q+1}^{\frac32 p\vartheta} \delta_{q+1}^{\frac12 p\vartheta
				\dif t}
			\\ & \lesssim (3\bar{M})^{p}\left( M_0^pL_{q+1}^p \lambda_{q}^{\frac32bp\vartheta-2 \beta-b\beta p\vartheta}
				+ L_{q+1}^pL\lambda_{q+1}^{\frac32p \vartheta-p\beta} \right) 
			\\ & \lesssim (3\bar{M}M_0)^{p} M_L^{2pm^{q+1}}\left(a^{b^q(\frac32bp\vartheta-2 \beta-b\beta p\vartheta)}+a^{b^{q+1}(\frac32p \vartheta-p\beta)} \right).
		\end{align*}
		Now choose $0<\vartheta<\min{\{\frac23\beta, \frac{4\beta}{3pb} \}}$ so that 
		\begin{align*}
			\tfrac32bp\vartheta-2 \beta-b\beta p\vartheta=:-\gamma_1<0,\qquad \tfrac32p \vartheta-p\beta=:-\gamma_2<0,
		\end{align*}
		and define
		\begin{align*}
				q_0=\max\left\{ q\in \mathbb{N}: M_L^{2pm^{q+1}}> a^{\gamma_1b^q} \vee  a^{\gamma_2b^{q+1}} \right\}.
			\end{align*}
			The maximum above exists since $m<b$ and $\gamma_1,\gamma_2>0$, and simple algebraic manipulations show that $q_0$ satisfies
			\begin{align*}
				q_0<\frac{\log\log M_L+\log2p+ \log m -\log \log a-\log \gamma_2 \wedge \log \gamma_1}{\log b- \log m} \leq C(1+\log \log M_L)
			\end{align*}
			for some constant $C$ not depending on $M_L$. 
		With this choice of $\vartheta$ and $q_0$, we estimate
		\begin{align*}
			\sum_{q\geq2}\|v_{q+1}-v_q\|_{L^p([0,T_L];C^\vartheta)} &\leq 3\bar{M}M_0 \sum_{q=2}^{q_0} M_L^{2m^{q+1}}\left( a^{-\frac{\gamma_1}{p}b^q}+a^{-\frac{\gamma_2}{p}b^{q+1}} \right)
			\\ &\qquad +3\bar{M}M_0 \sum_{q>q_0} M_L^{2m^{q+1}}\left( a^{-\frac{\gamma_1}{p}b^q}+a^{-\frac{\gamma_2}{p}b^{q+1}}\right)
			\\ &\leq3\bar{M}M_0(q_0M_L^{2m^{q_0+1}}+\bar{C}) \leq 3\bar{M}M_0 (M_L^{2m^{2+\log \log M_L}}+\bar{C}),
		\end{align*}
		where the constant $\bar{C}$ may vary from line to line. Thus, $v_q$ converges in $L^p([0,T_L];C^\vartheta)$ towards a limit $v$ with uniform bound in $\omega \in \Omega$
		\begin{align*}
			\|v\|_{L^p([0,T_L];C^\vartheta)} & \leq \|v_2\|_{L^p([0,T_L];C^\vartheta)} + \sum_{q\geq2}\|v_{q+1}-v_q\|_{L^p([0,T_L];C^\vartheta)} 
			\\&\leq 7\bar{M}M_0M_L^{2m^4}\lambda_{2}^{\frac{3}{2}\vartheta}+3\bar{M}M_0 (M_L^{2m^{2+\log \log M_L}}+\bar{C}).
		\end{align*}
		Moreover, it is easy to check that $v$ restricted to $[0,T_{L-1}]$ is simply the limit of $v_q$ in $L^p([0,T_{L-1}];C^\vartheta)$ for every $L>1$. 
		This uniquely identifies a limit object in $L^p_{\rm{loc}}([0,+\infty);C^\vartheta)$, denoted again by $v$, by simply gluing together limits at different values of $L$. 
		
		Furthermore, it follows from \eqref{eq:R}, \eqref{bd:R} that for all $p\in[1,\infty)$
		\begin{align*}
			\int_0^{T_L}\|\mathring{R}_q(t)\|^{p}_{L^\infty}\dif t\lesssim LM_L^{pm^{q}}\delta^{p}_{q+1}+M_0^{2p}M_L^{pm^{q}}\tau_{q-1} \to 0, \quad\mbox{as}\quad q\to\infty.
		\end{align*}
		Thus, $v$ is an analytically weak solution to \eqref{nonlinear1}.
		By setting $u=v+z$, we obtain an $(\mathcal{F}_t)_{t\geq0}$-adapted  analytically weak solution 
		to \eqref{eul1} which belongs to $L^p_{\rm{loc}}([0,+\infty);C^\vartheta)$. By the choice of parameters in Section~\ref{s:par} below, actually we can take $\vartheta\in \big(0, \min{\{\frac{\varkappa}{12{\cdot}30^3p},\frac{1}{3{\cdot}30^3p}\}} \big) $. Since $v_q(0)=0$ we deduce $v(0)=0$, which implies that $u(0)=u_0$.
		
		\emph{Step 3.}
		We have obtained global-in-time solutions to \eqref{eul1} of class $L^p_{\rm{loc}}([0,+\infty);C^\vartheta)$. Now we explain that another regularity of the solution claimed in the statement of the theorem holds. We write 
		\begin{align*}
			u(t)=u_0-\int_{0}^{t} \mathbb{P} \div (u(r)\otimes u(r)) \dif r +B(t).
		\end{align*}
		In view of the proof at the second step, we have $u\in L^2_{\rm{loc}}([0,+\infty);L^\infty)$. Then it follows that
		\begin{align*}
			\left\|  \int_{s}^{t} \mathbb{P} \div (u(r)\otimes u(r)) \dif r \right\|_{H^{-1}} \lesssim  \int_{s}^{t} \| u(r)\otimes u(r)\|_{L^\infty} \dif r \to 0,\quad \mbox{as} \quad |t-s|\to 0.
		\end{align*}
		By the assumption of noise $B\in C_{\rm{loc}}^{\frac12-}([0,+\infty);H_{\sigma}^{-1})$ we deduce $u\in C_{\rm{loc}}([0,+\infty);H_{\sigma}^{-1})$.
		
		\emph{Step 4.}
		Finally, we prove the non-uniqueness of solutions, still  under the additional assumption \eqref{eq:u0}.	In view of  \eqref{p:gamma}, we have on $t\in (2\tau_{-1}\wedge T_L,T_L]$
		\begin{align}\label{eq:K}
			\begin{aligned}
				\big|\|v\|_{L^2}^2-3K\big|&\leq \left|\sum_{q=0}^\infty(\|v_{q+1}\|_{L^2}^2-\|v_{q}\|_{L^2}^2-3\sigma_{q+1})\right|+3\sum_{q\neq1}\sigma_{q+1}
				\\
				&\leq M_1\sum_{q=0}^{\infty}L_{q+1}\delta_{q+1}+3\sum_{q\neq 1}\sigma_{q+1}=:c,
			\end{aligned}
		\end{align}
		where $c$ is a constant independent of $K$. This implies non-uniqueness by choosing different $K$ and using the same argument as in \cite[Theorem 5.4]{HZZ21markov}. 
	\end{proof}

	\subsection{Proof of Proposition~\ref{p:iteration'}}
	\label{s:it'}
	
	\subsubsection{Choice of parameters}
	\label{s:par} 
	First, for a fixed integer $b\geq 30$,  we let $\beta, \gamma \in (0,1)$ be small parameters satisfying 
	\begin{equation*}
		4b^2\beta<1, \quad b^2\beta<\varkappa, \quad
		16b\beta<\gamma\varkappa, \quad 2b^2\beta+\frac{b\gamma}{8}<1.
	\end{equation*}
	This can be obtained by choosing $\gamma=\frac{1}{b}$ and $\beta$ sufficiently  small such that
	$\beta<\min{\{\frac{\varkappa}{16b^2},\frac{1}{4b^2}\}}$. The parameter $m$ is chosen such that $L_q^9=M_L^{9m^q}\leq M_L^{m^{q+1}}=L_{q+1}$ and we can take $m=10$.
	The last free parameter $a\in 2^{8b\mathbb{N}} $ is chosen sufficiently large ensuring $f(q)\in \mathbb{N}$. We further require $a$ satisfies
	\begin{align*}
		30 \left(7\bar{M}M_0+1\right)  <\min{\{ a^{4\varkappa-2b^2\beta}, a^{1-2b^2\beta-\frac{b }{8}\gamma}, a^{\frac{b\varkappa}{8}\gamma-2b^2\beta}  \}},
	\end{align*}
	and
	\begin{align*}
		30  \left( 2+7\bar{M}M_0\right)^4 \left(1+M_0^2 \right)^5 < \min{\{  a^{b-28-4b^2\beta},  a^{\frac{b}{2}-10-b\beta}   \}}.
	\end{align*}
	We emphasize that these parameters are independent of $L$. However, choosing different $K$, we get different $a$.

	\subsubsection{Construction of $v_{q+1}$}\label{512}

	Now, we extend $v_q, z_q, \mathring{R}_q$ and $p_q$ to $t<0$ by taking them equal to the value at $t=0$. Then $(v_q,z_q,\mathring{R}_q)$ also satisfies equation \eqref{induction ps} for $t<0$ as $\p_tv_q(0)=0$. In order to guarantee smoothness throughout the construction, we also replace $(v_q,z_q,\mathring{R}_q)$  by a mollified field $(v_\ell,z_\ell,\mathring{R}_\ell)$. In addition, we denote $z^{in}_\ell=({z^{in}}*_x\phi_\ell)*_t\varphi_\ell$ and $Z_\ell=({Z_q}*_x\phi_\ell)*_t\varphi_\ell$.
	Using a mollification estimate and (\ref{inductionv C1}), we have for $t\in[0, T_L]$
	\begin{equation}\label{error ps1}
		\|v_q(t)-v_\ell(t)\|_{L^\infty}\lesssim \ell\|v_q\|_{C_{t,x}^1}\lesssim  L_q  \ell \lambda_q^{\frac{3}{2}}\delta_q^{\frac12}\leq \frac{1}{2} L_q \delta_{q+1}^{\frac12},
	\end{equation}
	where we  used the fact that $\ell\lambda_q^\frac32<\lambda_{q+1}^{-\beta}$ and we  chose $a$ sufficiently large to absorb the implicit constant. For $N\geq 1$, $t\in[0, T_L]$ using \eqref{inductionv C1} yields 
	\begin{align}\label{inductionv l}
		\|v_\ell\|_{C^N_{t,x}}\lesssim \ell^{-(N-1)} 	\|v_q\|_{C^1_{t,x}} \leq L_q \ell^{-N+1}\lambda_{q}^{\frac32}\delta_{q}^\frac12.
	\end{align}
	
	Following \cite{HZZ21markov}, we modify the definition of $\rho$ by replacing $\gamma_\ell$ with the constant $\sigma_{q+1}$, namely,
	$$
	\rho:=\sqrt{\ell^2+|\mathring{R}_\ell|^2}+\frac{r_0\sigma_{q+1}}{(2\pi)^3}.
	$$
	Consequently, the amplitude functions $a_k$ are modified as well via the formula \eqref{defak} with $\rho$ replaced by the above. Through a similar calculation as used in the proof of Proposition \ref{estiak}, we can deduce for any $s\in [0,T_L]$, $N\geq1$
	\begin{align}
		\|a_k\|_{C^N_{s,y}}
		&\lesssim\ell^{-2N-\frac{1}{2}}(1+\|v_q\|_{C^0_{s,y}}+\|z_q\|_{C^0_{s,y}})^N(1+\|\mathring{R}_q\|_{C^0_{s,y}})^{N+1},\label{a_k1'}
		\\	\|\partial_\tau a_k\|_{C^N_{s,y}} 
		&\lesssim\ell^{-2N-\frac{1}{2}}(1+\|v_q\|_{C^0_{s,y}}+\|z_q\|_{C^0_{s,y}})^{N+1}(1+\|\mathring{R}_q\|_{C^0_{s,y}})^{N+1},\label{a_k2'}
		\\ 	\|\partial_\tau a_k+ik\cdot( v_\ell+z_\ell)a_k\|_{C^N_{s,y}} 
		&\lesssim \ell^{-2N-\frac{1}{2}}(1+\|v_q\|_{C^0_{s,y}}+\|z_q\|_{C^0_{s,y}})^{N}(1+\|\mathring{R}_q\|_{C^0_{s,y}})^{N+1}.\label{derivat a_k'}
	\end{align}
	And for $N=0$, we also have
	\begin{align}
		\|a_k(s)\|_{L^{\infty}}&\leq Mr_0^{-\frac12}(\ell^{\frac{1}{2}}+\|\mathring{R}_\ell(s)\|^{\frac{1}{2}}_{L^\infty}+r_0\delta_{q+1}^\frac12), \label{a_k0'}
		\\	\|\partial_\tau a_k\|_{C^0_{s,y}}&\lesssim (1+\|\mathring{R}_q\|^{\frac{1}{2}}_{C^0_{s,y}})(1+\|v_q\|_{C^0_{s,y}}+\|z_q\|_{C^0_{s,y}}),\label{tauak0'}
		\\ 	\|\partial_\tau a_k+ik\cdot( v_\ell+z_\ell)a_k\|_{C^0_{s,y}} & \lesssim \lambda_{q}^{-1}(\|\mathring{R}_q\|^{\frac{1}{2}}_{C^0_{s,y}}+\ell^\frac{1}{2}+\delta_{q+1}^\frac12), \label{deriva ak0'}
	\end{align}
	where $M$ is a universal constant given in \eqref{A5} below.
	
	Given a smooth non-decreasing cut-off function
	\begin{equation}\label{def:chi}
		\aligned
		\chi(t)=\begin{cases}
			0,& t\leq\tau_{q+1},\\
			\in (0,1),& t\in (\tau_{q+1},{\tau_q} ),\\
			1,&t\geq {\tau_q}.
		\end{cases}
		\endaligned
	\end{equation}
	Let  $ w^{(p)}_{q+1}$ and $  w_{q+1}^{(c)}$  be given as in Section \ref{413} with $a_{k}$ replaced by the above.
	We define the perturbations $\tilde w^{(p)}_{q+1}$, $ \tilde w_{q+1}^{(c)}$ and new velocity ${v}_{q+1}$  as follows
	$$\tilde w^{(p)}_{q+1}:=w^{(p)}_{q+1}\chi,\quad \tilde w^{(c)}_{q+1}:=w^{(c)}_{q+1}\chi, \quad {v}_{q+1}:=v_\ell +\tilde{w}_{q+1}^{(p)}+\tilde{w}_{q+1}^{(c)}.$$
	Using \eqref{eq:R}, \eqref{a_k0'} and the fact $\ell\leq\delta_{q+1}$, $2\tau_{q-1}-\ell \geq \tau_{q-1}$, we obtain
	for $t\in(2\tau_{q-1}\wedge T_L, T_L]$ 
	\begin{equation}\label{estimate wqp ps}
		\aligned
		\|\tilde w_{q+1}^{(p)}(t)\|_{L^\infty}&\leq |\Lambda|r_0^{-\frac12}M(\ell^{\frac{1}{2}}+\|\mathring{R}_\ell(t)\|^{\frac{1}{2}}_{L^\infty}+r_0^{\frac12}\sigma_{q+1}^\frac12)
		\\ &\leq r_0^{-\frac12}|\Lambda|M\delta_{q+1}^\frac12+|\Lambda| M\sigma_{q+1}^\frac12+r_0^{-\frac12}|\Lambda|M \sup_{s\in[t-\ell,t]}\|\mathring{R}_q(s)\|^{\frac{1}{2}}_{L^\infty}
		\\&\leq (\bar{M}-1)L_q^{\frac12}(\delta_{q+1}^{\frac12}+\sigma_{q+1}^\frac12),
		\endaligned
	\end{equation}
where $r_0$ and $|\Lambda|$ retain the same definitions as provided in Section~\ref{313}. Specifically, $r_0$ is the constant defined in Lemma~\ref{Belw2} and $|\Lambda|$ denotes the cardinality of the set $\Lambda :=\cup_{j=1}^{8} \Lambda_j$ introduced in Lemma~\ref{Belw2}. Additionally, $M$ is the geometric constant defined in \eqref{A5} and $\bar{M}$ is a universal constant satisfying
	\begin{align}\label{bar M2}
		2r_0^{-\frac12}|\Lambda|M+|\Lambda|M+1<\bar{M}.
	\end{align}
	Similarly, using \eqref{bd:R} and \eqref{a_k0'} again, we obtain for $t\in (\tau_{q + 1}\wedge T_L,2\tau_{q-1}\wedge T_L]$
	\begin{equation}\label{estimate wqp ps1}
		\aligned
		\|\tilde w_{q+1}^{(p)}(t)\|_{L^\infty} &\leq r_0^{-\frac12}|\Lambda|M\delta_{q+1}^\frac12+|\Lambda|M\sigma_{q+1}^\frac12+r_0^{-\frac12}|\Lambda|M \sup_{s\in[t-\ell,t]}\|\mathring{R}_q(s)\|^{\frac{1}{2}}_{L^\infty}
		\\&\leq (\bar{M}-1)L_q^{\frac12}(\delta_{q + 1}^{\frac12}+\sigma_{q+1}^\frac12+M_0).
		\endaligned
	\end{equation}
	
	For the $L^\infty$-norm of $\tilde w_{q+1}^{(c)}$, combining \eqref{induction z}, \eqref{inductionv ps}, \eqref{bd:R} and \eqref{a_k1'} we deduce for $t\in[0, T_L]$
	
	\begin{equation}\label{correction est ps}
		\aligned
		\|\tilde w_{q+1}^{(c)}(t)\|_{L^\infty}&\lesssim \frac{1}{\lambda_{q+1}}\sum_{k\in \Lambda} \left\|  \nabla a_k \times \frac{B_k}{|k|}e^{i\lambda_{q+1}k\cdot x}\right\| _{C^0_{t,x}}
		\\&\lesssim \frac{\ell^{-\frac52}}{\lambda_{q+1}} \left( 1+\|v_q\|_{C^0_{t,x}}+\|z_q\|_{C^0_{t,x}}\right) \left(1+ \|\mathring{R}_q\|_{C^0_{t,x}}\right)^2
		\\ &\leq \frac{\lambda_{q}^{10}}{\lambda_{q+1}}  \left( 1+7\bar{M}M_0L_q+M_L\right) \left(1+L_qM_0 ^2\right)^2
		\\ &\leq \frac{\lambda_{q}^{10}}{\lambda_{q+1}}L_q^3  \left( 2+7\bar{M}M_0\right) \left(1+M_0 ^2\right)^2 \leq \frac12 L_{q+1}^{\frac12} \delta_{q+1}^{\frac12},
		\endaligned
	\end{equation}
	where the last inequality is justified by $L_q^9\leq L_{q+1}$, $10+b\beta<\frac{b}{2}$ and $a$ being sufficiently large to ensure
	\begin{align}\label{a1}
		10 \left( 2+7\bar{M}M_0\right)^2 \left(1+M_0 ^2\right)^4<  a^{\frac{b}{2}-10-b\beta}.
	\end{align}
	
	With these bounds established, we now have all the necessary components to finalize the proof of Proposition~\ref{p:iteration'}. To streamline the presentation, we will divide the details into several subsections.
	
	\subsubsection{Estimates on the velocity} \label{413'} We split into three parts to complete the estimates of velocity.
	
	\textbf{I.} Proof of \eqref{iteration ps}.
	First, combining  \eqref{error ps1},  (\ref{estimate wqp ps}) and \eqref{correction est ps} we obtain for $t\in( 2\tau_{q-1}\wedge T_L, T_L]$
	\begin{align*}
		\|v_{q+1}(t)-v_{q}(t)\|_{L^{\infty }}&\leq \|\tilde w_{q+1}^{(p)}(t)\|_{L^{\infty }}+ \|\tilde w_{q+1}^{(c)}(t)\|_{L^{\infty }}+\|v_{\ell}(t)-v_{q}(t)\|_{L^{\infty }}
		\\ &\leq (\bar{M}-1)L_q^{\frac12}(\delta_{q+1}^{\frac12}+\sigma_{q+1}^\frac12)+\frac12 L_{q+1}^{\frac12}\delta_{q+1}^{\frac12} +\frac12  L_q \delta_{q+1}^{\frac12} 
		\leq \bar{M}L_{q+1}^{\frac12}(\delta_{q+1}^{\frac12}+\sigma_{q+1}^\frac12).
	\end{align*}
	For $t\in (\tau_{q+1} \wedge T_L, 2\tau_{q-1} \wedge T_L]$, using
	\eqref{error ps1}, (\ref{estimate wqp ps1}) together with \eqref{correction est ps} yields 
	
	\begin{align*}
		\|v_{q+1}(t)-v_{q}(t)\|_{L^{\infty }}&\leq \|\tilde w_{q+1}^{(p)}(t)\|_{L^{\infty }}+ \|\tilde w_{q+1}^{(c)}(t)\|_{L^{\infty }}+\|v_{\ell}(t)-v_{q}(t)\|_{L^{\infty }}
		\\ &\leq  (\bar{M}-1) L_q^{\frac12} (\delta_{q + 1}^{\frac12}+\sigma_{q + 1}^{\frac12}+M_0)+\frac12 L_{q+1}^{\frac12}\delta_{q+1}^{\frac12}+\frac12 L_q \delta_{q+1}^{\frac12} 
		\\ & \leq \bar{M}L_{q+1}^{\frac12}(\delta_{q + 1}^{\frac12}+\sigma_{q + 1}^{\frac12}+M_0).
	\end{align*}
	For $t\in  [0,\tau_{q+1}\wedge T_L]$, by the definition of $\chi$, it holds $\chi(t)=0$ as well as $v_{q}(t)=0$ by \eqref{inductionv ps}. This implies
	$$
	\|v_{q+1}(t)-v_{q}(t)\|_{L^{\infty }}=\|v_{\ell}(t)-v_{q}(t)\|_{L^{\infty}}=0.
	$$
	Hence \eqref{iteration ps} follows.

	\textbf{II.} Proof of \eqref{inductionv ps} on the level $q+1$.
	For the first bound in \eqref{inductionv ps} on the level $q+1$, we use \eqref{iteration ps} to deduce for $t\in (\tau_{q+1}\wedge T_L,T_{L}]$
	
	\begin{align*}
		\|v_{q+1}(t)\|_{L^{\infty}}&\leq \sum_{ r=0}^{q}\|v_{r+1}(t)-v_{r}(t)\|_{L^{\infty}}
		\\&\leq \bar{M} \left(\sum_{ r=0}^{q}L_{r+1}^{\frac12}(\delta_{r+1}^{\frac12}+\sigma_{r + 1}^{\frac12})+\sum_{ r=0}^{q}L_{r+1}^{\frac12}(\delta_{r+1}^{\frac12}+\sigma_{r + 1}^{\frac12}+M_0)\mathbf{1}_{ \{t\in (\tau_{r+1}\wedge T_{L},2\tau_{r-1}\wedge T_{L}] \}}\right)
		\\&\leq 2\bar{M} L_{q+1}^{\frac12}\sum_{r=1}^{q+1}(\delta_{r}^{\frac12}+\sigma_{r}^{\frac12})+3\bar{M}M_0L_{q+1}^{\frac12},
	\end{align*}
	where we used the fact that by $2\tau_{r }\leq \tau_{r-1}$  each $t\in [0,T_L]$ only belongs to three intervals $(\tau_{r+1}\wedge T_{L},2\tau_{r-1} \wedge T_{L}]$. For the second bound in \eqref{inductionv ps}, by \eqref{iteration ps} we find for $t\in [0,\tau_{q+1}\wedge T_L]$ that
	$$
	\|v_{q+1}(t)\|_{L^{\infty }}\leq \sum_{r=0}^{q}\|v_{r+1}(t)-v_{r}(t)\|_{L^{\infty }}=0.
	$$
	This implies  \eqref{inductionv ps} on the level $q+1$.  
	
	\textbf{III.} Proof of \eqref{inductionv C1} on the level $q+1$. From the definition of $\chi$ and $2\leq a^{b(b-1)\beta}$, we have
	\begin{align}\label{deriv chi}
		\|\chi'\|_{C^0_t}\leq (\tau_{q}-\tau_{q + 1})^{-1}\leq \tau_{q+1}^{-1}\leq \lambda_{q+2}^{2\beta }.
	\end{align}
	By using \eqref{inductionv ps}, \eqref{bd:R}, \eqref{a_k1'}, \eqref{a_k0'}, \eqref{estimate wqp ps1} and the same argument as in \eqref{wp 1}, we obtain for any $t\in [0,T_L]$ 
	\	\begin{equation}\label{principle est2 ps}
		\aligned
		\|\tilde w_{q+1}^{(p)}\|_{C_{t,x}^1} &\lesssim 
		\ell^{-\frac{5}{2}}\left( 1+\|v_q\|_{C^0_{t,x}}+\|z_q\|_{C^0_{t,x}} \right) \left(1+ \|\mathring{R}_q\|_{C^0_{t,x}}\right)^2
		\\ & \quad +\lambda_{q+1}\left( 1+\|v_q\|_{C^0_{t,x}}+\|z_q\|_{C^0_{t,x}} \right)\left(1+\|\mathring{R}_q\|_{C^0_{t,x}}^{\frac{1}{2}}\right)+ \|\chi'\|_{C^{0}_{t}}\| w_{q+1}^{(p)}\|_{C_{t,x}^0}
		\\  & \lesssim L_q^3(\lambda_{q+1}+\lambda_{q}^{10}+\lambda_{q+2}^{2\beta }) (2+7\bar{M}M_0)(1+M_0^2)^2\leq \frac14L_{q+1} \lambda_{q+1}^{\frac32}\delta_{q + 1}^{\frac12},
		\endaligned
	\end{equation}
	which requires $b>20$, $2b^2\beta<1$ and $a$ large enough satisfying \eqref{a1} to absorb the implicit constant.
	Similarly, it follows from \eqref{wc 1}, \eqref{inductionv ps}, \eqref{bd:R}, \eqref{a_k1'}, \eqref{a_k0'} and \eqref{correction est ps} that for any $t\in [0,T_L]$
	\begin{equation}\label{correction est2 ps}
		\aligned
		\|\tilde w_{q+1}^{(c)}\|_{C_{t,x}^1} &\lesssim 
		\frac{\ell^{-\frac{9}{2}}}{\lambda_{q+1}} \left( 1+ \|v_q\|_{C^0_{t,x}}+\|z_q\|_{C^0_{t,x}} \right)^2\left(1+\|\mathring{R}_q\|_{C^0_{t,x}}\right)^3
		\\&\quad +\ell^{-\frac{5}{2}}\left( 1+\|v_q\|_{C^0_{t,x}}+\|z_q\|_{C^0_{t,x}} \right)^2 \left(1+ \|\mathring{R}_q\|_{C^0_{t,x}}\right)^2+ \|\chi'\|_{C^{0}_{t}}\| w_{q+1}^{(c)}\|_{C_{t,x}^0}
		\\ & \lesssim L_q^5 \left(\frac{\lambda_{q}^{18}}{\lambda_{q+1}}+ \lambda_{q}^{10}+\frac{\lambda_{q}^{10}\lambda_{q+2}^{2\beta}}{\lambda_{q+1}} \right)  (2+7\bar{M}M_0)^2(1+M_0^2)^3\leq \frac14 L_{q+1}\lambda_{q+1}^{\frac32}\delta_{q + 1}^{\frac12},
		\endaligned
	\end{equation}
	where we used $L_q^5\leq L_{q+1}$, $b>20$, $2b^2\beta<1$ and chose $a$ sufficiently large satisfying \eqref{a1} to absorb the implicit constant. Then, we combine \eqref{inductionv l}, \eqref{principle est2 ps} and \eqref{correction est2 ps} to obtain for any $t\in[0, T_L]$
	\begin{equation*}
		\aligned
		\|v_{q+1}\|_{C^1_{t,x}}\leq \|v_\ell\|_{C^1_{t,x}}+\|\tilde w_{q+1}^{(p)}\|_{C^1_{t,x}}+\|\tilde w_{q+1}^{(c)}\|_{C^1_{t,x}}
		\leq L_q\lambda_{q}^{\frac32}\delta_{q}^{\frac12}+\frac12 L_{q+1}\lambda_{q+1}^\frac32\delta_{q + 1}^{\frac12}  \leq   L_{q+1}\lambda_{q+1}^\frac32\delta_{q + 1}^{\frac12} ,
		\endaligned
	\end{equation*}
	which requires $\frac32+b\beta<\frac32 b$ and $a$ large enough to absorb the constant.
	Hence \eqref{inductionv C1} follows.

	\subsubsection{Estimates on the energy}
	
	We control the energy similarly to Section \ref{s:en}.
	By the definition of $v_{q+1}$, we find
	\begin{align}\label{eq:deltaE ps}
		\begin{aligned}
			&\big|\|v_{q+1}\|_{L^2}^2-\|v_q\|_{L^2}^2-3\sigma_{q+1}\big|
			\leq \left| \int_{\mathbb{T}^3} \big(|\tilde w_{q+1}^{(p)}|^2 -\frac{3\sigma_{q+1}}{(2\pi)^3}\big) \dif x \right|+  \int_{\mathbb{T}^3} |\tilde w_{q+1}^{(c)}|^2 \dif x 
			\\ &\qquad \qquad 
			+2\left| \int_{\mathbb{T}^3}\tilde w_{q+1}^{(p)}w_{q+1}^{(c)}  \dif x \right|+ 2\left| \int_{\mathbb{T}^3}v_\ell w_{q+1}\dif x \right|  +\left|\|v_\ell\|_{L^2}^2-\|v_q\|_{L^2}^2\right| .
		\end{aligned}
	\end{align}
	We start with the first term on the right-hand side of \eqref{eq:deltaE ps}. Taking trace on the both sides of \eqref{wpxwp} and using the fact that $\mathring{R}_\ell$ is traceless we deduce for $t\in (2\tau_{q-1}\wedge T_L,T_L]$
	\begin{align*}
		|\tilde w_{q+1}^{(p)}|^2-\frac{3\sigma_{q+1}}{(2\pi)^3}=3r_0^{-1}\sqrt{\ell^2+|\mathring{R}_\ell|^2}+\sum_{k\in \Lambda}\tr(U_k) e^{i\lambda_{q+1}k\cdot x}
		,
	\end{align*}
	hence
	\begin{equation}\label{eq:gg ps}
		\aligned
		\left| \int_{\mathbb{T}^3} |\tilde w_{q+1}^{(p)}|^2 -\frac{3\sigma_{q+1}}{(2\pi)^3} \dif x \right| \leq 3\cdot(2\pi)^3r_0^{-1}\ell
		&+3\cdot(2\pi)^3r_0^{-1}\sup_{t\in(2\tau_{q-1}\wedge T_L,T_L]}\|\mathring{R}_\ell(t)\|_{L^\infty}
		\\ &+\sum_{k\in \Lambda}\left| \int_{\mathbb{T}^3}\tr(U_k) e^{i\lambda_{q+1}k\cdot x}\dif x\right|.
		\endaligned
	\end{equation}
	Here we estimate each term separately, by the choice of parameters, we find
	\begin{align*}
		3\cdot(2\pi)^3r_0^{-1}\ell\leq 3\cdot(2\pi)^3r_0^{-1}\lambda_{q}^{-4}\leq\frac{1}{ 10}\delta_{q+1},
	\end{align*}
	which requires $b\beta <1$ and $a$ large enough to absorb the implicit constant. 
	We use \eqref{eq:R} on $\mathring{R}_q$ and $2\tau_{q-1}-\ell\geq \tau_{q-1}$ to deduce 
	\begin{align*}
		3\cdot(2\pi)^3r_0^{-1}\sup_{t\in(2\tau_{q-1}\wedge T_L,T_L]}\|\mathring{R}_\ell(t)\|_{L^\infty} \leq  3\cdot(2\pi)^3r_0^{-1} \sup_{t\in(\tau_{q-1}\wedge T_L,T_L]}\|\mathring{R}_q(t)\|_{L^\infty}\leq  \frac{M_1}{10}L_q\delta_{q+1},
	\end{align*}
	where  $M_1$ is a universal constant satisfying
	\begin{align}\label{def M_1}
		M_1\geq 30\cdot (2\pi)^3r_0^{-1}.
	\end{align} 
	For the last term in \eqref{eq:gg ps}, we use the same argument as in \eqref{eq:Uk} to obtain
	\begin{align*}
		\sum_{k\in \Lambda} \left| \int_{\mathbb{T}^3}\tr(U_k) e^{i\lambda_{q+1}k\cdot x}\dif x\right|  &\lesssim \frac{\ell^{-\frac{5}{2}}}{\lambda_{q+1}}(1+\|v_q\|_{C_{t,x}^0}+\|z_q\|_{C_{t,x}^0}) (1+\|\mathring{R}_q\|^\frac{1}{2}_{C_{t,x}^0})^5
		\\&\lesssim \frac{\lambda_{q}^{10}}{\lambda_{q+1}}L_q^4(2+7\bar{M}M_0)(1+M_0)^5
		\leq\frac{1}{10}L_{q+1}\delta_{q+1},
	\end{align*}
	where we have used $L_q^9\leq L_{q+1}$, $10+2b\beta<b$ and chose $a$ large enough satisfying \eqref{a1} to absorb the constant. 
	This completes the bound for \eqref{eq:gg ps}.

	Going back to \eqref{eq:deltaE ps}, we control the  remaining terms as follows.
	We use \eqref{correction est ps} and $L_q^9\leq L_{q+1}$ to deduce for $t\in (2\tau_{q-1}\wedge T_L,T_L]$
	\begin{align*}
		\int_{\mathbb{T}^3} |\tilde w_{q+1}^{(c)}|^2 \dif x &
		\lesssim \frac{\lambda_{q}^{20}}{\lambda_{q+1}^2} L_q^6  \left( 2+7\bar{M}M_0\right)^2 \left(1+M_0^2 \right)^4 \leq \frac{1}{10} L_{q+1}\delta_{q+1},
	\end{align*}
	where we used  $10+b\beta<b$ and chose $a$ sufficiently large satisfying \eqref{a1} to absorb the constant. Similarly, using the estimates \eqref{estimate wqp ps} and \eqref{correction est ps} we have for $t\in (2\tau_{q-1}\wedge T_L,T_L]$
	\begin{align*}
		2\left| \int_{\mathbb{T}^3}\tilde w_{q+1}^{(p)}w_{q+1}^{(c)}  \dif x \right| 
		\lesssim \frac{\lambda_{q}^{10}}{\lambda_{q+1}} L_q^4 \left( 2+7\bar{M}M_0\right) \left(1+M_0^2 \right)^2\bar{M}M_0
		\leq \frac{1}{10} L_{q+1}\delta_{q+1},
	\end{align*}
	which requires $10+2b\beta<b$ and $a$ satisfying \eqref{a1} to absorb the constant.
	
	Using the same argument as \eqref{vl w q+1},
	integrating by parts  and using the estimates \eqref{bd:R}, \eqref{inductionv l}, \eqref{a_k0'} and $L_q^9\leq L_{q+1}$, we obtain
	\begin{align*}
		2 \left| \int_{\mathbb{T}^3}v_\ell\cdot w_{q+1} \dif x \right| &\lesssim \frac{1}{\lambda_{q+1}} \left(  \sum_{k\in \Lambda}\left\| a_k\frac{B_k}{|k|}e^{i\lambda_{q+1}k\cdot x}\right\| _{C_{t,x}^0}\|v_\ell\|_{C^0_tC^1_x}\right) 
		\\&\lesssim \frac{\lambda_{q}^{\frac32}}{\lambda_{q+1}}L_qMr_0^{-\frac12}(\ell^{\frac{1}{2}}+\|\mathring{R}_\ell(t)\|^{\frac{1}{2}}_{L^\infty}+r_0\delta_{q+1}^\frac12) 
		\\&\lesssim \frac{\lambda_{q}^{\frac32}}{\lambda_{q+1}} L_q^2 \bar{M}(\ell^{\frac{1}{2}}+M_0+\delta_{q+1}^\frac12)  \leq  \frac{1}{10}L_{q+1}\delta_{q+1},
	\end{align*}
	where we have used $\frac32 +2b\beta<b$ in the last inequality and we chose $a$ sufficiently large to absorb the constant. 
	
	For the last  term in \eqref{eq:deltaE ps}, by \eqref{error ps1} we have
	\begin{align*}
		|\|v_\ell\|_{L^2}^2-\|v_q\|_{L^2}^2|&\leq \|v_\ell-v_q\|_{L^2}(\|v_\ell\|_{L^2}+\|v_q\|_{L^2})\\ &\lesssim \|v_\ell-v_q\|_{C_tL^\infty}\|v_q\|_{C_tL^\infty}
		\lesssim L_q^2 \ell\lambda_q^\frac32\delta_{q}^{\frac12}\bar{M}M_0  \leq \frac{1}{10}L_{q+1}\delta_{q+1},
	\end{align*}
	which requires $L_q^9\leq L_{q+1}$, $b\beta<1$ and $a$ sufficiently  large to absorb the extra constant.

	Combining the above estimates, then \eqref{p:gamma} follows.
	
	\subsubsection{Estimates on the Reynolds stress $\mathring{R}_{q+1}$}	\label{415}
	The new Reynolds stress is defined following the same approach as in \eqref{def R q+1}, with the original two correctors $w_{q+1}^{(p)}$ and $w_{q+1}^{(c)}$ replaced by their truncated counterparts $\tilde w_{q+1}^{(p)}$ and $\tilde w_{q+1}^{(c)}$. For completeness, we recover the expression for the Reynolds stress at the level $q+1$:
	\begin{align*}
		& \div \mathring{R}_{q+1}-\nabla p_{q+1}
		=\underbrace{\chi(\partial_t+(v_\ell+z_\ell)\cdot \nabla)w_{q+1}^{(p)}}
		_{\div(R^{\textrm{trans}})}+ \underbrace{\chi^2 \div (w_{q+1}^{(p)} \otimes w_{q+1}^{(p)}+\mathring{R}_\ell)}
		_{\div(R^{\textrm{osc}})+\nabla p^{\textrm{osc}}}
		\\&+
		\underbrace{\chi \partial_tw_{q+1}^{(c)}}_{\div(R^{\textrm{error}}_1)}+\underbrace{\chi \div (w_{q+1}^{(p)}\otimes(v_\ell+z_\ell))}_{\div (R^{\textrm{error}}_2)}+\underbrace{\partial_t \chi(w_{q+1}^{(p)}+w_{q+1}^{(c)})+(1-\chi^2)\div \mathring{R}_\ell}_{\div (R^\textrm{cut-off})}
		\\&+\underbrace{\div\left(\chi(v_{q+1}+z_\ell)\otimes w_{q+1}^{(c)}+\chi w_{q+1}^{(c)}\otimes (v_{q+1}+z_\ell)-\chi^2w_{q+1}^{(c)}\otimes  w_{q+1}^{(c)}\right)}_{\div(R^{\textrm{error}}_3)+\nabla p^{\textrm{error}}_3}
		\\&+\underbrace{\div(v_{q+1} \otimes z_{q+1}-v_{q+1} \otimes z_{\ell} +z_{q+1} \otimes v_{q+1}-z_{\ell} \otimes v_{q+1}+z_{q+1}\otimes z_{q+1}-z_{\ell} \otimes z_{\ell})-z_{q+1}+z_{\ell}}_{\div(R^{\textrm{cor}})+\nabla p^{\textrm{cor}}}
		\\ & +\underbrace{\div\left( (v_\ell+z_\ell)\mathring{\otimes}(v_\ell+z_\ell)-((v_q+z_q)\mathring{\otimes}(v_q+z_q))*_x\phi_\ell*_t\varphi_\ell\right)  }_{\div(R^{\textrm{com}})}
		-\nabla p_{\ell}.
	\end{align*}
	We divide these error terms into three parts to estimate. Note that the inductive assumptions \eqref{induction z} and \eqref{inductionv ps} state that $v_q$ and $z_q$ are uniformly bounded on $[0,T_L]$. 
	As a result, the estimates regarding $R^{\textrm{trans}}$, $R^{\textrm{osc}}$, $R^{\textrm{error}}_1$, $R^{\textrm{error}}_2$ and $R^{\textrm{error}}_3$ will not cause significant changes compared to Section~\ref{31}. 
	
	\textbf{I.} 
	First, we employ the same approach as outlined in Section~\ref{31} to estimate the above five terms respectively. For $R^{\textrm{trans}}$, using \eqref{Rtran 1}, \eqref{Rtran 2}, \eqref{induction z}, \eqref{inductionv ps} and \eqref{bd:R}, we deduce for $t\in [0,T_L]$
	\begin{align*}
		\|R^{\textrm{trans}}(t)\|_{L^\infty}&\lesssim\frac{\|\mathring{R}_q\|^{\frac{1}{2}}_{C^0_{t,x}}+\delta_{q+1}^{\frac12}}{\lambda_{q}\lambda_{q+1}^{-\alpha}}+\frac{\ell^{-\frac{13}{2}}(1+\|v_q\|_{C_{t,x}^0}+\|z_q\|_{C_{t,x}^0})^4(1+\|\mathring{R}_q\|_{C^0_{t,x}})^4}
		{\lambda_{q+1}^{1-\alpha}}
		\\ &\lesssim \frac{M_0+M_0L_q}{\lambda_{q}\lambda_{q+1}^{-\alpha}}+
		\frac{\lambda_{q}^{26} L_q^8  \left( 2+7\bar{M}M_0\right)^4 \left(1+M_0^2 \right)^4}{\lambda_{q+1}^{1-\alpha}}
		\leq \frac{1}{10}L_{q+1}\delta_{q+2},
	\end{align*}
	where we have used $L_q^9\leq L_{q+1}$ in the last inequality. We have also used $b>29$, $2b^2\beta+b\alpha<1$ and $a$ large enough to have
	\begin{align}\label{a2}
		10  \left( 2+7\bar{M}M_0\right)^4 \left(1+M_0^2 \right)^5<a^{b-28-2b^2\beta-b\alpha}.
	\end{align}
	Moving to $R^{\textrm{osc}}$, we use \eqref{Rosc}, \eqref{induction z}, \eqref{inductionv ps}, \eqref{bd:R} and $L_q^9\leq L_{q+1}$ to have for $t\in [0,T_L]$
	\begin{align*}
		\|R^{\textrm{osc}}(t)\|_{L^\infty}
		&\lesssim
		\frac{\ell^{-7}}{\lambda_{q+1}^{1-\alpha}}(1+\|v_q\|_{C_{t,x}^0}+\|z_q\|_{C_{t,x}^0})^3 (1+\|\mathring{R}_q\|_{C^0_{t,x}})^5
		\\&\lesssim \frac{\lambda_{q}^{28}}{\lambda_{q+1}^{1-\alpha}}L_q^8(2+7\bar{M}M_0)^3 (1+M_0^2)^5 \leq \frac{1}{10}\delta_{q+2},
	\end{align*}
	where the last inequality is justified by $b>29$, $2b^2\beta+b\alpha<1$ and $a$ being sufficiently large to satisfy \eqref{a2}. Let us focus on $R^{\textrm{error}}$ now, using \eqref{error 1}, \eqref{error 2},  \eqref{induction z}, \eqref{inductionv ps}, \eqref{bd:R} and $L_q^9\leq L_{q+1}$, we have for $t\in [0,T_L]$
	\begin{align*}
		\|R^{\textrm{error}}_1(t)\|_{L^\infty}+\|R^{\textrm{error}}_2(t)\|_{L^\infty} \lesssim& \frac{\ell^{-\frac{13}{2}}}{\lambda_{q+1}^{1-\alpha}}(1+\|v_q\|_{C_{t,x}^0}+\|z_q\|_{C_{t,x}^0})^4(1+\|\mathring{R}_q\|_{C^0_{t,x}})^4
		\\ &+ \frac{\ell^{-\frac{17}{2}}}{\lambda_{q+1}^{2-\alpha}}(1+\|v_q\|_{C_{t,x}^0}+\|z_q\|_{C_{t,x}^0})^4(1+\|\mathring{R}_q\|_{C^0_{t,x}})^5
		\\ \lesssim& \frac{\lambda_{q}^{26}}{\lambda_{q+1}^{1-\alpha}}L_q^9(2+7\bar{M}M_0)^4 (1+M_0^2)^5 \leq \frac{1}{10}L_{q+1}\delta_{q+2},
	\end{align*}
	which requires $b>27$, $2b^2\beta+b\alpha<1$ and $a$ large enough satisfying \eqref{a2} to absorb the constant.
	For $R^{\textrm{error}}_3$, in view of  \eqref{induction z}, \eqref{inductionv ps} and \eqref{correction est ps}, we deduce for $t\in [0,T_L]$
	\begin{align*}L_{q+1}
		\|R^{\textrm{error}}_3(t)&\|_{L^\infty} \lesssim  \|v_{q+1}(t)+z_\ell(t)\|_{L^\infty}\|w_{q+1}^{(c)}(t)\|_{L^\infty}+\|w_{q+1}^{(c)}(t)\|_{L^\infty}^2
		\\ & \lesssim  \frac{\lambda_{q}^{10}}{\lambda_{q+1}} L_q^4( 2+7\bar{M}M_0)^2 (1+M_0^2 )^2
		+ \frac{\lambda_{q}^{20}}{\lambda_{q+1}^2}   L_q^6( 2+7\bar{M}M_0)^2 \left(1+M_0^2 \right)^4
		\leq \frac{1}{10} L_{q+1}\delta_{q+2},
	\end{align*}
	where in the last inequality we have used $10+2b^2\beta<b$ and chose $a$ sufficiently large satisfying \eqref{a2} to absorb the constant. 
	
	\textbf{II.} 
	Compared to Section~\ref{31}, there is a slight difference due to the estimates involving the process $z$, as its initial value is not zero. Therefore, it is crucial to carefully track down the use of mollification estimates on $z_\ell-z_q$ when estimating $R^{\textrm{com}}$ and $R^{\textrm{cor}}$. We begin with $R^{\textrm{cor}}$. Combining \eqref{zq+1 - zq} and \eqref{z ps} we have for any $t\in [0,T_L]$
	\begin{equation}\label{Rcom1 1}
		\|z_{q+1}(t)-z_q(t)\|_{L^\infty}=\|Z_{q+1}(t)-Z_q(t)\|_{L^\infty}
		\lesssim \lambda_{q+1}^{-\frac{\gamma\varkappa}{8}} \|Z(t)\|_{H^{\frac{3}{2}+\varkappa}}\leq  \lambda_{q+1}^{-\frac{\gamma\varkappa}{8}}M_L.
	\end{equation}
	Using  \eqref{z ps}, \eqref{eq:u0} and mollification estimates, we deduce for any $\delta\in (0,\frac{1}{12})$ and $t\in [0,T_L]$ 
	\begin{equation}\label{Rcom1 2}
		\aligned
		\|z_\ell(t)-z_q(t)\|_{L^\infty} &\leq \|z^{in}_\ell(t)-z^{in}(t)\|_{L^\infty}+	\|Z_\ell(t)-Z_q(t)\|_{L^\infty} 
		\\& \lesssim \ell \|z^{in}\|_{C^1_tC^0_x}+\ell^{\varkappa}\|z^{in}\|_{C^0_tC^\varkappa_x}+ \ell\|Z_q\|_{C^0_{t}C^1_x}+\ell^{\frac12-\delta}\|Z_q\|_{C^{1/2-\delta}_{t}C^0_x}
		\\ &\lesssim \ell^{\varkappa}\|u_0\|_{C^{\varkappa}}+\ell^{\frac{5}{12}}M_L+\ell \lambda_{q+1}^{\frac{\gamma}{8}}M_L\lesssim (\ell^\varkappa+\lambda_{q}^{-1}\lambda_{q+1}^{\frac{\gamma}{8}})M_L.
		\endaligned
	\end{equation}
	Then we combine \eqref{Rcom1 1} with \eqref{Rcom1 2} and use inductive estimates \eqref{induction z} and \eqref{inductionv ps} to have for $t\in[0,T_L]$
	\begin{align*}
		\|R^{\textrm{cor}}(t)\|_{L^\infty} &\leq  (\|v_{q+1}(t)\|_{L^\infty}+\|z_{q+1}(t)\|_{L^\infty}+\|z_{\ell}(t)\|_{L^\infty})\|z_\ell(t)-z_{q+1}(t)\|_{L^\infty} 
		\\&\lesssim (\lambda_q^{-4\varkappa}+\lambda_{q}^{-1}\lambda_{q+1}^{\frac{\gamma}{8}}+\lambda_{q+1}^{-\frac{\gamma\varkappa}{8}})(7\bar{M}M_0L_{q+1}^{\frac12}+M_L)M_L
		\\& \leq (\lambda_q^{-4\varkappa}+\lambda_{q}^{-1}\lambda_{q+1}^{\frac{\gamma}{8}}+\lambda_{q+1}^{-\frac{\gamma\varkappa}{8}})L_{q+1}(7\bar{M}M_0+1) \leq \frac{1}{10} L_{q+1} \delta_{q+2}.
	\end{align*}
	In the last inequality, we require
	$b^2\beta<\varkappa$,  $16b\beta<\gamma\varkappa$, $2b^2\beta+\frac{b\gamma}{8}<1$ and increase $a$ sufficiently large to have
	\begin{align}\label{a3}
		7\bar{M}M_0+1&\leq \frac{1}{30}\min{\{ a^{4\varkappa-2b^2\beta}, a^{1-2b^2\beta-\frac{b }{8}\gamma}, a^{\frac{b\varkappa}{8}\gamma-2b^2\beta}\}}.
	\end{align} 
	
	For $R^{\textrm{com}}$,	using \eqref{z ps}, \eqref{induction z}, $L_q^9\leq L_{q+1}$ and standard mollification estimates as in \eqref{moll}, we obtain for $\delta \in (0,\frac{1}{12})$ and $t\in [0,T_L]$
	\begin{align*}
		\|R^{\textrm{com}}(t)\|_{L^\infty}\lesssim & \left(\ell \|v_q\|_{C^1_{t,x}}+\ell \|Z_q\|_{C^0_tC^1_x}+\ell^{\frac12-\delta}\|Z_q\|_{C^{1/2-\delta}_tC^0_x}\right)
		\left( \|v_q\|_{C^0_{t,x}} +  \|Z_q\|_{C^0_{t,x}} +  \|z^{in}\|_{C^0_{t,x}}\right) 
		\\ &+ \left( \ell \|z^{in}\|_{C^1_tC^0_x} +\ell^{\varkappa } \|z^{in}\|_{C^0_tC^\varkappa_x} \right) 	\left( \|v_q\|_{C^0_{t,x}} +  \|Z_q\|_{C^0_{t,x}} +  \|z^{in}\|_{C^0_{t,x}}\right) 
		\\ \leq & \left( \ell \lambda_{q}^{\frac32}+ \ell \lambda_{q+1}^{\frac{\gamma}{8}}L+\ell^{\frac{5}{12}}L+\ell M_L+\ell^{\varkappa}M_L \right) \left( 7\bar{M}M_0L_q^{\frac12}+M_L \right) 
		\\ \leq & \left( \lambda_{q}^{-1}\lambda_{q+1}^{\frac{\gamma}{8}}+\lambda_{q}^{-4\varkappa}\right) L_q^2\left(7 \bar{M}M_0+1\right)
		\leq  \frac{1}{10}L_{q+1}\delta_{q+2},
	\end{align*}
	where we have used $2b^2\beta+\frac{b}{8}\gamma<1$, $b^2\beta<\varkappa$ and chose $a$ large enough satisfying \eqref{a3} to absorb the constant.
	
	\textbf{III.} 
	At last, we estimate $R^{\textrm{cut-off}}$. We first bound $\partial_t \chi(w_{q+1}^{(p)}+w_{q+1}^{(c)})$, applying Lemma~\ref{lemma1} with $m=1$ together with \eqref{a_k1'}, \eqref{a_k0'} and , we have for $t\in [0,T_L]$ 
	\begin{align*}
		\|\partial_t \chi \mathcal{R}(w_{q+1}^{(p)}(t)+w_{q+1}^{(c)}(t)) \|_{L^\infty} \lesssim & \frac{\ell^{-\frac{9}{2}}\lambda_{q+2}^{2\beta}}{\lambda_{q+1}^{1-\alpha}}(1+\|v_q\|_{C^0_{t,x}}+\|z_q\|_{C^0_{t,x}})^2(1+\|\mathring{R}_q\|_{C^0_{t,x}})^{3}
		\\ &+\frac{\ell^{-\frac{13}{2}}\lambda_{q+2}^{2\beta}}{\lambda_{q+1}^{2-\alpha}}(1+\|v_q\|_{C^0_{t,x}}+\|z_q\|_{C^0_{t,x}})^3(1+\|\mathring{R}_q\|_{C^0_{t,x}})^{4}
		\\ \lesssim &\frac{\lambda_{q}^{18}\lambda_{q+2}^{2\beta}}{\lambda_{q+1}^{1-\alpha}} L_q^7 \left( 2+7\bar{M}M_0 \right)^3\left( 1+M_0^2\right)^4
		\leq  \frac{1}{10}L_{q+1}\delta_{q+2},
	\end{align*}
	where we have used $b>27$, $18+4b^2\beta+b\alpha<b$ and chose $a$ large enough to have 
	\begin{align*}
		\left( 2+7\bar{M}M_0 \right)^3\left( 1+M_0^2 \right)^4<a^{b-4b^2\beta-b\alpha-18}.
	\end{align*}
	
	For $(1-\chi^2)\mathring{R}_\ell$, it is necessary to divide the time interval into two segments. 
	By the definition of $\chi$, we know  $(1-\chi^2)\mathring{R}_\ell\equiv0$ for $t\in(\tau_q \wedge T_L,T_L]$. Hence, summarizing all the estimates we have established above, we have for $t\in(\tau_q \wedge T_L,T_L]$
	\begin{align*}
		\|\mathring{R}_{q+1}(t)\|_{L^\infty}\leq L_{q+1}\delta_{q+2}.
	\end{align*}
	For $t\in [0,\tau_q\wedge T_L]$, 	we have 
	\begin{align}\label{estimate R q+1}
		\|\mathring{R}_{q+1}(t)\|_{L^\infty}\leq L_{q+1}\delta_{q+2}+\|\mathring{R}_\ell(t)\|_{L^\infty} \leq L_{q+1}\delta_{q+2}+ \sup_{s\in [t-\ell,t] }\|\mathring{R}_q(s)\|_{L^\infty}.
	\end{align}
	This completes the proof of \eqref{iteration R}. We then use \eqref{bd:R} and \eqref{estimate R q+1} to obtain for $t\in [0,T_L]$
	\begin{align*}
		\|\mathring{R}_{q+1}(t)\|_{L^\infty}\leq L_{q+1}\delta_{q+2}+ \sup_{s\in [t-\ell,t] }\|\mathring{R}_q(s)\|_{L^\infty} \leq L_{q+1}\delta_{q+2}+ L_q \sum_{r=1}^{q+1}\delta_r =L_{q+1}\sum_{r=1}^{q+2}\delta_r,
	\end{align*} 
	which implies \eqref{bd:R} at the level $q+1$.
	
	Hence, the proof of Proposition~\ref{p:iteration'} is complete.
	
	\appendix
	\renewcommand{\appendixname}{Appendix~\arabic{section}}
	\renewcommand{\theequation}{A.\arabic{equation}}
	\section{Beltrami waves}\label{Bw2}
	In this section, we recall Beltrami waves defined in \cite[Section 3]{DelSze13}, which are adapted to the convex integration scheme in Proposition~\ref{p:iteration}.
	It is crucial to emphasize that this construction is purely deterministic, with none of the functions depending on $\omega$. We start by defining Beltrami waves.
	Let $\lambda_0\geq1$ and let $A_k\in \mathbb{R}^3$ be such that 
	\begin{align*}
		A_k \cdot k =0, \quad |A_k|=\frac{1}{\sqrt{2}}, \quad A_{-k}=A_k
	\end{align*}
	for $k\in \mathbb{Z}^3$ with $|k|=\lambda_0$. Furthermore, we define the complex vector
	\begin{align*}
		B_k=A_k+i \frac{k}{|k|}\times A_k\in \mathbb{C}^3.
	\end{align*}
	By construction, the vector $B_k$ has the properties
	\begin{align*}
		|B_k|=1,\quad B_k \cdot k =0,\quad ik\times B_k=\lambda_0B_k,\quad B_{-k}=\overline{B_k},
	\end{align*}
	and 
	\begin{align}\label{A1}
		(B_k\otimes B_{k'}+B_{k'}\otimes B_{k})(k+k')=(B_k\cdot B_{k'})(k+k').
	\end{align}
	The following lemma presents a useful property of linear combinations of complex Beltrami plane waves, as detailed in  \cite[Proposition 3.1]{DelSze13}.
	\begin{lemma}\label{Belw11}
		Let $\lambda_0\geq1$ and $B_k$ be defined as above. Then for any choice of coefficients $a_k\in \mathbb{C}$ with $\overline{a_k}=a_{-k}$ the vector field
		\begin{equation}
			W(\xi)=\sum_{|k|=  \lambda_0}a_k B_k e^{i k \cdot \xi} 
		\end{equation}
		is a real-valued, divergence-free Beltrami vector field with $\mathrm{curl}W=\lambda_0 W$. 
		Furthermore, since $B_k \otimes B_{-k}+B_{-k}\otimes B_k= \mathrm{Id}- \frac{k}{|k|} \otimes \frac{k}{|k|}$, we have
		\begin{equation}\label{A3}
			\frac{1}{(2\pi)^3}\int_{\mathbb{T}^3} W\otimes W \dif \xi=\frac{1}{2} \sum_{|k|= \lambda_0}|a_k|^2\left( \mathrm{Id}-\frac{k}{|k|} \otimes\frac{k}{|k|} \right) .
		\end{equation}
	\end{lemma}
	Moreover, the following geometric lemma can be found in \cite[Lemma 3.2]{DelSze13}.
	\begin{lemma}\label{Belw2}
		For every $N\in \mathbb{N}$ we can choose $r_0>0$ and $\lambda_0>1$ with the following property. There exist pairwise disjoint subsets $\Lambda_j \subset \{k\in \mathbb{Z}^3:|k|=\lambda_0 \}$, for $j\in\{1,\dots,N\}$
		and smooth positive functions 
		\begin{align*}
			\gamma_k^{(j)}\in C^\infty(B_{r_0}(\mathrm{Id})),\quad  j\in\{1,\dots,N\}, \ k\in \Lambda_j,
		\end{align*}
		such that 
		$ k\in \Lambda_j$  implies $-k\in \Lambda_j $ and  $\gamma_k^{(j)}=\gamma_{-k}^{(j)}$. And for each $R\in B_{r_0}(\mathrm{Id})$ we have the identity
		\begin{align}\label{A6}
			R=\frac{1}{2}\sum_{k\in \Lambda_j}\left( \gamma_k^{(j)}(R)\right)^2\left( \mathrm{Id}-\frac{k}{|k|}\otimes\frac{k}{|k|}\right) .
		\end{align}
	\end{lemma}
	Also, it is convenient to denote by $M$ a geometric constant such that 
	\begin{align}\label{A5}
		\sup_{k\in \Lambda_j}\sum_{0\leq l\leq n}\|\gamma_k^{(j)}\|_{C^l(B_{r_0}(\Id))}\leq M 
	\end{align}
	holds for $n$ large enough and $j\in \{1,\dots,	N\}$. This parameter is universal.

	\renewcommand{\theequation}{B.\arabic{equation}}
	\section{Proof of Proposition~\ref{estiak} and Proposition~\ref{cor:U}}\label{ap:A}  
	For completeness, we include here the detailed proof of Proposition~\ref{estiak} and Proposition~\ref{cor:U}, and the notation for the norms retains the same as in Section~\ref{ci2}.
	
	\begin{proof}[Proof of Proposition~\ref{estiak}]
		First, we have the following estimates from \cite[Appendix B]{LZ24}
		\begin{align}
			\|h_k\|_{C^N_{s,y}} \lesssim & \ell^{-2N-\frac{1}{2}}(1+\|\mathring{R}_q\|_{C^0_{[s-1,s+1],y}})^{(N+1)}, \quad N\geq2, \label{estimate hN1}
			\\	\|h_k\|_{C^0_{s,y}} \lesssim\|\mathring{R}_q\|^{\frac{1}{2}}_{C^0_{[s-1,s+1],y}}&+\ell^{\frac{1}{2}}+\delta_{q+1}^\frac12 ,\quad
			\|h_k\|_{C^1_{s,y}} \lesssim\ell^{-\frac{3}{2}}(\|\mathring{R}_q\|_{C^0_{[s-1,s+1],y}}+1)^{\frac32}. \label{estimate hN01}
		\end{align}
		
		We first estimate \eqref{a_k1} and \eqref{a_k0}. 	Recalling the definition of $a_k$ in Section~\ref{313},
		\begin{align*}
			a_k(s,y,\tau)=\mathbf{1}_{\{k\in \Lambda_j\}}h_k(s,y)\phi_k^{(j)}( \tau,(v_\ell+z_\ell)(s,y)).
		\end{align*} 
		By \eqref{phi1}, \eqref{estimate hN1}, \eqref{estimate hN01} and applying chain rule and \cite[Proposition C.1]{BDLIS16} to $a_k$, we obtain for $N\geq 1$
		\begin{equation}\label{ak}
			\aligned
			&\|a_k\|_{C^N_{s,y}} 
			\\ &\lesssim \sum_{l=1}^{N}  \|\phi_k^{(j)}( \tau,v_\ell+z_\ell)\|_{C^l_{s,y}}    \|h_k\|_{C^{N-l}_{s,y}}  +\|h_k\|_{C^N_{s,y}}
			\\ &\lesssim \sum_{l=1}^{N} \|h_k\|_{C^{N-l}_{s,y}} (\lambda_{q}\|v_\ell+z_\ell\|_{C^l_{s,y}}+\lambda_{q}^l \|v_\ell+z_\ell\|_{C^1_{s,y}}^l)  +\ell^{-2N-\frac{1}{2}}(1+\|\mathring{R}_q\|_{C^0_{[s-1,s+1],y}})^{N+1}
			\\&\lesssim \sum_{l=1}^{N}\ell^{-2(N-l)-\frac12}(1+\|\mathring{R}_q\|_{C^0_{[s-1,s+1],y}})^{N-l+1} \ell^{-2l}(\|v_q+z_q\|_{C^0_{[s-1,s+1],y}}+\|v_q+z_q\|_{C^0_{[s-1,s+1],y}}^l)
			\\ &\qquad \qquad + \ell^{-2N-\frac{1}{2}}(1+\|\mathring{R}_q\|_{C^0_{[s-1,s+1],y}})^{N+1}
			\\&\lesssim \ell^{-2N-\frac{1}{2}}(1+\|v_q\|_{C^0_{[s-1,s+1],y}}+\|z_q\|_{C^0_{[s-1,s+1],y}})^N(1+\|\mathring{R}_q\|_{C^0_{[s-1,s+1],y}})^{N+1}.
			\endaligned
		\end{equation}
		For $N=0$, we have
		\begin{align}\label{C4}
			\|a_k\|_{C^0_{s,y}}\leq \|h_k\|_{C^0_{s,y}} \lesssim \|\mathring{R}_q\|^{\frac{1}{2}}_{C^0_{[s-1,s+1],y}}+\ell^{\frac{1}{2}}+\delta_{q+1}^\frac12.
		\end{align}

		Next, we estimate \eqref{a_k2} and \eqref{tauak0}.
		Similarly, by \eqref{phi2}, \eqref{estimate hN1}, \eqref{estimate hN01} and applying chain rule and \cite[Proposition C.1]{BDLIS16} to $\partial_\tau a_k$, we obtain for $N\geq 1$
		\begin{equation}\label{tauak1}
			\aligned
			\|\partial_\tau a_k\|_{C^N_{s,y}} 
			&\lesssim \sum_{l=1}^{N}  \|\partial_\tau \phi_k^{(j)}( \tau,v_\ell+z_\ell)\|_{C^l_{s,y}}    \|h_k\|_{C^{N-l}_{s,y}}  +\|h_k\|_{C^N_{s,y}} \|\partial_\tau \phi_k^{(j)}( \tau,v_\ell+z_\ell)\|_{C^0_{s,y}}
			\\ &\lesssim \sum_{l=1}^{N} \|h_k\|_{C^{N-l}_{s,y}} (\lambda_{q}\|v_\ell+z_\ell\|_{C^l_{s,y}}+\lambda_{q}^l \|v_\ell+z_\ell\|_{C^1_{s,y}}^l)(1+\|v_\ell+z_\ell\|_{C^0_{s,y}}) 
			\\ &\qquad+\ell^{-2N-\frac{1}{2}}(1+\|\mathring{R}_q\|_{C^0_{[s-1,s+1],y}})^{N+1}(1+\|v_\ell+z_\ell\|_{C^0_{s,y}}) 
			\\&\lesssim \ell^{-2N-\frac{1}{2}}(1+\|v_q\|_{C^0_{[s-1,s+1],y}}+\|z_q\|_{C^0_{[s-1,s+1],y}})^{N+1}(1+\|\mathring{R}_q\|_{C^0_{[s-1,s+1],y}})^{N+1}.
			\endaligned
		\end{equation}
		For $N=0$, using \eqref{phi2} and \eqref{C4}, we obtain
		\begin{align}\label{C7}
			\|\partial_\tau a_k\|_{C^0_{s,y}} \lesssim (1+\|v_q\|_{C^0_{[s-1,s+1],y}}+\|z_q\|_{C^0_{[s-1,s+1],y}})(1+\|\mathring{R}_q\|^{\frac{1}{2}}_{C^0_{[s-1,s+1],y}}).
		\end{align}
		
		At last, we estimate \eqref{derivat a_k} and \eqref{deriva ak0}.
		From \eqref{pro phi}, it follows that in a neighborhood of $(\tau,v_\ell+z_\ell)\in \R \times \R^3$ 
		\begin{equation}\label{B7}
			\aligned
			\partial_\tau \phi_k^{(j)}+ik\cdot( v_\ell+z_\ell)\phi_k^{(j)}&=ik\cdot(v_\ell+z_\ell-\frac{\bar{l}}{\lambda_{q}})\phi_k^{(j)},
			\\ \partial_\tau a_k+ik\cdot( v_\ell+z_\ell)a_k&=ik\cdot(v_\ell+z_\ell-\frac{\bar{l}}{\lambda_{q}})a_k,
			\endaligned
		\end{equation}
		where $\bar{l}\in \mathcal{C}_j$ satisfies 
		\begin{align}\label{B8}
			|\lambda_{q}(v_\ell+z_\ell)-\bar{l}|\leq 1.
		\end{align}
		Then using \eqref{ak}, \eqref{B7}, \eqref{B8} and chain rule, we obtain for $N\geq 1$
		\begin{equation*}\label{tau+ak}
			\aligned
			\|&\partial_\tau a_k+ik\cdot( v_\ell+z_\ell)a_k\|_{C^N_{s,y}} 
			\lesssim \|(v_\ell+z_\ell-\frac{\bar{l}}{\lambda_{q}})a_k\|_{C^N_{s,y}}
			\\&\lesssim \|a_k\|_{C^N_{s,y}}+ \sum_{j=0}^{N-1} \|a_k\|_{C^j_{s,y}}(1+\|v_\ell+z_\ell\|_{C^{N-j}_{s,y}})
			\\&\lesssim\ell^{-2N-\frac{1}{2}}(1+\|v_q\|_{C^0_{[s-1,s+1],y}}+\|z_q\|_{C^0_{[s-1,s+1],y}})^{N}(1+\|\mathring{R}_q\|_{C^0_{[s-1,s+1],y}})^{N+1}.
			\endaligned
		\end{equation*}
		For $N=0$, using \eqref{C4}, \eqref{B7}, \eqref{B8} again, we have
		\begin{equation*}\label{tauak2}
			\aligned
			\|\partial_\tau a_k+ik\cdot( v_\ell+z_\ell)a_k\|_{C^0_{s,y}} &\lesssim \frac{
				\| (\lambda_{q}(v_\ell+z_\ell)-\bar{l} )a_k\|_{C^0_{s,y}}}{\lambda_{q}}
			\lesssim \frac{\|\mathring{R}_q\|^{\frac{1}{2}}_{C^0_{[s-1,s+1],y}}+\ell^\frac{1}{2}+\delta_{q+1}^\frac12}{\lambda_{q}}.
			\endaligned
		\end{equation*}
		Hence, the proof of Proposition~\ref{estiak} is complete.
	\end{proof}
	\begin{proof}[Proof of Proposition~\ref{cor:U}]
		Since \eqref{WxW1} and \eqref{eq:proU} can be considered as a direct consequence of
		\cite[Proposition 6.1]{DelSze13}, here we only provide proof of \eqref{esti:UN} and \eqref{esti:U1}. For each $1\leq |k|\leq 2 \lambda_0$, $U_k$ is the sum of finitely many terms of the form $a_k' a_k''$ for $k=k'+k''$ and $k',k''\in \Lambda$. Thus, the estimates for $U_k$ are similar to those for $a_k$. More precisely, by using \eqref{ak}, \eqref{C4} and chain rule, we deduce for any $s\in \mR$ and $N\geq2$, 
		\begin{align*}\label{Uk13}
			\begin{aligned}
				\| &U_k\|_{C^0_sC^N_y}\lesssim \sum_{k'+k''=k} \| a_{k'}a_{k''}\|_{C^0_sC^N_y}
				\\ &\lesssim \ell^{-2N-\frac{1}{2}}(1+\|v_q\|_{C^0_{[s-1,s+1],y}}+\|z_q\|_{C^0_{[s-1,s+1],y}})^N(1+\|\mathring{R}_q\|_{C^0_{[s-1,s+1],y}})^{N+1}(1+\|\mathring{R}_q\|_{C^0_{[s-1,s+1],y}}^{\frac12})
				\\&+ \sum_{j=1}^{N-1} \bigg[ \ell^{-2j-\frac12}(1+\|v_q\|_{C^0_{[s-1,s+1],y}}+\|z_q\|_{C^0_{[s-1,s+1],y}})^j(1+\|\mathring{R}_q\|_{C^0_{[s-1,s+1],y}})^{j+1}
				\\  & \qquad \qquad \times \ell^{-2(N-j)-\frac12}(1+\|v_q\|_{C^0_{[s-1,s+1],y}}+\|z_q\|_{C^0_{[s-1,s+1],y}})^{N-j}(1+\|\mathring{R}_q\|_{C^0_{[s-1,s+1],y}})^{N-j+1}\bigg]
				\\ &\lesssim \ell^{-2N-1}(1+\|v_q\|_{C^0_{[s-1,s+1],y}}+\|z_q\|_{C^0_{[s-1,s+1],y}})^N(1+\|\mathring{R}_q\|_{C^0_{[s-1,s+1],y}})^{N+2}.
			\end{aligned}
		\end{align*}
		Similarly, using \eqref{ak} and \eqref{C4} again, we obtain
		\begin{align*}
			\|U_k\|_{C^0_sC^1_y}&\lesssim  \ell^{-\frac52}(1+\|v_q\|_{C^0_{[s-1,s+1],y}}+\|z_q\|_{C^0_{[s-1,s+1],y}})(1+\|\mathring{R}_q\|_{C^0_{[s-1,s+1],y}})^{2}(1+\|\mathring{R}_q\|_{C^0_{[s-1,s+1],y}}^{\frac12})\\ &\lesssim \ell^{-\frac52}(1+\|v_q\|_{C^0_{[s-1,s+1],y}}+\|z_q\|_{C^0_{[s-1,s+1],y}})(1+\|\mathring{R}_q\|_{C^0_{[s-1,s+1],y}}^{\frac12})^5. \qedhere
		\end{align*}
	\end{proof}
	
	\section*{Acknowledgment}
	The author would like to thank Rongchan Zhu for quite useful discussion and advice. This work is supported by National Key R\&D Program of China (No. 2022YFA1006300) and the NSFC grant of China (No. 12271030).
  
	\section*{Declarations}
	 
    Data sharing not applicable to this article as no datasets were generated or analyzed during the current study.
    The authors have no competing interests to declare that are relevant to the
    content of this article.

	\def\cprime{$'$} \def\ocirc#1{\ifmmode\setbox0=\hbox{$#1$}\dimen0=\ht0
		\advance\dimen0 by1pt\rlap{\hbox to\wd0{\hss\raise\dimen0
				\hbox{\hskip.2em$\scriptscriptstyle\circ$}\hss}}#1\else {\accent"17 #1}\fi}

\end{document}